\documentclass[10pt,reqno,english]{amsart}
\usepackage{amsfonts}
\usepackage{amsmath}
\usepackage{amssymb}
\usepackage{amsthm}
\usepackage{bbm}
\usepackage{caption}
\usepackage{latexsym}
\usepackage{listings}
\usepackage[left=3cm,top=3cm,right=3cm,bottom=3cm]{geometry}
\usepackage{multicol}
\usepackage{verbatim}
\usepackage{amsfonts,amsmath,latexsym,verbatim,amscd,mathrsfs,color,array}
\usepackage{amsmath,amssymb,amsthm,amsfonts,graphicx,color}
\usepackage{amssymb}
\usepackage{pdfsync}
\usepackage{epstopdf}
\usepackage[english]{babel}
\usepackage[latin1]{inputenc} %%% este comando sierve para que las tildes se pongan como es usual%%%
\usepackage{amsmath}
\usepackage{amssymb}
\usepackage{amsfonts}
\usepackage{amsthm}
\usepackage{latexsym}
\usepackage{bbold}

\numberwithin{equation}{section}

\newcommand{\ep}{\varepsilon}        %epsilon en cursiva
\newcommand{\R}{\mathbb{R}}          %conjunto numeros reales
\newcommand{\K}{\mathbb{K}}          %cuerpo generico
\newcommand{\Nn}{\mathcal{N}}        %vecindad de un punto(neighborhood)
\def \J {\mathcal{J}}         %operador de Jacobi
\def \K {\mathcal{K}}         %constante
\def \H {\mathbb{H}}          %global first solution
\def \w {\mathtt{w}}          %global solution
\def \B {\mathtt{B}}          %operator cut-off.
\def \N {\mathtt{N}}          %operator cut-off
\def \T {\mathcal{T}}          %operator fixed point problem
\def \G {\mathtt{G}}          %projections operator

\providecommand{\pdi}[2]{<#1,#2>} %prodcuto interno
\providecommand{\bs}[1]{\boldsymbol{#1} }       %Letra en Negritas (bold)
\newcommand{\ssi}{\Leftrightarrow}              %si y solo si (equivalencia logica)
\newcommand{\lap}{\Delta}                       %laplaciano
\newcommand{\dist}{\mathop{\rm dist}}           %distancia
\def \div{\mathop{\rm div}}                     %divergencia

\newtheorem{thm}{Theorem}
\newtheorem{lemma}{Lemma}[section]
\newtheorem{prop}{Proposition}[section]
\newtheorem{coro}{Corollary}[section]
\newtheorem{claim}{Claim}[section]

\newtheorem{rmk}{\it Remark}[section]

\numberwithin{equation}{section}

\numberwithin{equation}{section}

\title{A two end family of solutions for the Inhomogeneous
Allen-Cahn equation in $\R^2$}

\author[O. Agudelo]{Oscar Agudelo}
\address{\noindent O. Agudelo - Departamento de Ingenier\'{\i}a Matem\'atica and
CMM, Universidad de Chile, Casilla 170 Correo 3, Santiago, Chile.}
\email{oagudelo@dim.uchile.cl}

\author[A. Zu\~niga]{Andres Zu\~niga}
\address{\noindent A. Zu\~niga - Departamento de Ingenier\'{\i}a Matem\'atica and
CMM, Universidad de Chile, Casilla 170 Correo 3, Santiago, Chile.}
\email{azuniga@dim.uchile.cl}

\begin{document}
\maketitle

% ABSTRACT
\begin{abstract}
In this work we construct a family of entire bounded solution for
the singulary perturbed Inhomogeneous Allen-Cahn Equation $\ep^2
\div\left(a(x)\nabla u\right) -a(x)F'(u)=0$ in $\R^2$, where $\ep
\to 0$. The nodal set of these solutions is close to a
"nondegenerate" curve which is asymptotically two non paralell
straight lines. Here $F'$ is a double-well potential and $a$ is a
smooth positive function. We also provide example of curves and
functions $a$ where our result applies. This work is in connection
with the results found in \cite{DuLai},~\cite{DuGui10} and \cite{35}
handling the compact case.
\end{abstract}

%%%%%%%%%%%%%%%%%%%%%%%%%%%%%%%%%%%%%%%%%%%%%%%%%%%%%%%%%%%%%%%%%%%%%%%%%%%%%%
%%%%%%%%%%%%%%%%%%%%%%%%%%%%%%%%%%%%%%%%%%%%%%%%%%%%%%%%%%%%%%%%%%%%%%%%%%%%%%
%\tableofcontents

%%%%%%%%%%%%%%%%%%%%%%%%%%%%%%%%%%%%%%%%%%%%%%%%%%%%%%%%%%%%%%%%%%%%%%%%%%%%%%
%%%%%%%%%%%%%%%%%%%%%%%%%%%%%%%%%%%%%%%%%%%%%%%%%%%%%%%%%%%%%%%%%%%%%%%%%%%%%%
\section{Introduction}
In this paper, we consider the singularly perturbed Inhomogeneous
Allen-Cahn equation
\begin{equation}
\ep^2\div(a(x)\nabla u(x))- a(x)F'(u)=0,\quad \hbox{in
}\R^2\label{IntrodEq0}
\end{equation}

To fix ideas we can think of equation \eqref{IntrodEq0} as the
\begin{equation}
\ep^2\Delta_{g} u(x)- F'(u)=0,\quad \hbox{in }M\label{IntrodEq1}
\end{equation}
where $(M,g)$ is the whole plane endowed with a planar metric
induced by the potential $a$. In this setting $\Delta_{g}$ denotes
the Laplace Beltrami operator on $M$ and the function $F:\R\to\R$ is
a smooth double-well potential, i.e a function satisfying
\begin{align}
&F(s)>0,\quad  \forall s\neq \pm1,\label{IntrodEq2}\\ 
&F(\pm1)=0, \label{IntrodEq3}\\
&\sigma^2_{\pm}:=F''(\pm 1)>0.\label{IntrodEq3.2}
\end{align}

The typical example for $F$ corresponds to the balanced and
bi-stable twin-pit
\begin{equation}
F(u)=\frac{1}{4}(1-u^2)^2, \quad -F'(u)=u-u^3\notag
\label{IntrodEq4}
\end{equation}
and $\sigma_{\pm}=\sqrt{2}$.

\medskip
Equation \eqref{IntrodEq1} arises in the {\it Gradient Theory of
Phase Transitions} ~\cite{29}. In this physical model, there are two
different states of a material represented by the values $u=\pm 1$.
It is of interest to study nontrivial configurations in which these
two states try to coexist. Hence, the function $u$ represents a
smooth realization of the phase, which except for a narrow region,
is expected to take values close to $\pm 1$, namely the global
minima of $F$.

\medskip
Let us consider the energy functional
\begin{equation}
J_{\ep}(u)=\int_{M}\left[\frac{\ep}{2}|\nabla u|^2+\frac{1}{\ep}
F(u)\right]dV_{g}.\label{IntrodEq6}
\end{equation}

We are interested in critical points of~\eqref{IntrodEq6}, which
correspond to solutions of~\eqref{IntrodEq1}. For any set
$\Lambda\subset M$ the function
\begin{equation}
u^{*}_{\Lambda}:=\chi_{\Lambda}-\chi_{M\setminus
\Lambda},\quad\text{ in }\quad M\label{IntrodEq7}
\end{equation}

minimizes the second term in~\eqref{IntrodEq6}, however, it is not a
smooth solution. By considering an $\ep-$regularization of
$u^{*}_{\Lambda}$, say $u_{\Lambda,\ep}$, it can be checked that
\begin{equation}
J_{\ep}(u_{\Lambda,\ep})\approx \int_{\partial
\Lambda}1dS_{g}\label{IntrodEq8}
\end{equation}

for $\ep>0$ small, where $dS_{g}$ denotes the area element in
$\partial \Lambda$. Relation~\eqref{IntrodEq8} entails that,
transitions varying from $-1$ to $+1$ must be selected, for
instance, in such way that $\partial \Lambda$ minimizes the
perimeter functional
\begin{equation}
\int_{\partial \Lambda}1dS_{g}\label{IntrodEq9}
\end{equation}

In the case that $\partial\Lambda$ is also a smooth submanifold of
$M$, we say that $\partial\Lambda$ is a \emph{minimal submanifold}
of $M$ if $\partial\Lambda$ is critical for the area
functional~\eqref{IntrodEq9}. In particular, it is easy to see that
minimizing submanifolds are critical for~\eqref{IntrodEq9}, and
therefore they are minimal submanifolds of $M$.

\medskip
The intuition behind the previous remarks, was first observed by
Modica in \cite{36}, based upon the fact that when $\partial\Lambda$
is a smooth submanifold of $M$, then transitions varying from $-1$
to $+1$ take place along the normal direction of $\partial\Lambda$
in $M$, having a $1-$dimensional profile in this direction. This profile
corresponds to a function $w$, which is the heteroclinic solution to
the ODE
\begin{equation*}
\left\{
\begin{aligned}
&w''(t)-F'(w(t))=0,\quad\text{ in }\quad\R\\
&w(\pm\infty)=\pm1
\end{aligned}
\right.
\end{equation*}
connecting the two states. The existence of $w$ is ensured by
conditions~\eqref{IntrodEq2}-\eqref{IntrodEq3}-\eqref{IntrodEq3.2} of
$F$. As for the twin pit nonlinearity, we have that
$$
w(t)=\tanh\left(\frac{t}{\sqrt{2}}\right),\quad t\in\R.
$$

This intuition gave a great impulse to the Calculus of Variations,
and the theory of the $\Gamma-$Convergence in the 70th's. Modica
stated that if $M=\Omega\subset\R^N$ is a smooth Euclidean domain
and $\{u_{\ep}\}_{\ep>0}$ is a family of local minimizers
of~\eqref{IntrodEq6} with uniformly bounded energy, then up to
subsequence, $u_{\ep}$ converges in a $L^1_{loc}(\Omega)$-sense to
some limit $u^{*}_{\Lambda}$ of the form~\eqref{IntrodEq7}.

\medskip
There are a number of important works regarding existence and
asymptotic behavior of solutions to~\eqref{IntrodEq1}, under a
variety of different settings. In \cite{35} Pacard and Ritore studied
equation~\eqref{IntrodEq1} in the case that $(M,g)$ is a compact
Riemannian Manifold, they construct a family of solutions
$\{u_{\ep}\}_{\ep>0}$ to~\eqref{IntrodEq1} having transition from
$-1$ to $+1$ on a region $\ep-$close to a compact minimal
submanifold $N$, with positive Ricci-curvature
$$
k_N:=|A_N|^2+Ric(\nu_N,\nu_N)>0.
$$

Under the same conditions del Pino, Kowalczyk, Wei and Yang
constructed in \cite{37} a sequence of solutions with multiple
clustered layers collapsing onto $N$. A gap condition is needed,
related with the interaction between interfaces.

There are related other results under a similar setting for $M$ and
$N$, regarding the equation
$$ \ep^2\Delta_{g}u-V(z)F'(u)=0,\quad \text{ in }\quad M$$
done by B.Lai and Z.Du in \cite{DuLai} where a family of solutions
with a single transition is constructed. Additionally L.Wang and
Z.Du  dealt in~\cite{DuWang12} with the same pro\-blem, considering
multiple transitions this time. In both works the minimal character
and nondegeneracy properties of $N$, are with respect to the
weighted area functional $\int_{M}V^{1/2}$. In the same line,
it is worth to mention a recent work due to Z.Du and C.Gui~\cite{DuGui10} where
they build a smooth solution to the Neumann problem
$$
\ep^2\Delta u-V(z)F'(u)=0\;\text{ in }\;
\Omega,\quad
\frac{\partial u}{\partial n}=0\;\text{ on }\;
\partial\Omega
$$
having a single transition near a smooth closed curve
$\Gamma\subset\Omega$, nondegenerate geodesic relative to the
arclength $\int_{\Gamma}V^{1/2}$. Here, $\Omega$ is a smooth bounded
domain in $\R^2$, and $V$ is an uniformly positive smooth potential.

As for the noncompact case, recently in \cite{9} the authors
considered equation~\eqref{IntrodEq1} when $M=\R^3$. Here the
authors build for any small $\ep>0$ a family of solutions with
transitions close to a non-degenerate complete embedded minimal
surface with finite total curvature.

%In addition, entire solutions with multiple transition layers
%to~\eqref{IntrodEq1} in $\R^2$ were found
%in~\cite{delpinokowalczykpacardwei08}. In this case the nodal set of
%the solutions consists on multiple noncompact curves, not
%intersecting with themselves, whose location is governed by the Toda
%system of ODEs.

%Finally, we mentioned that all these works take advantage of a very
%versatile and powerful tool, namely, the infinite dimensional
%Lyapunov-Schmidt reduction method, which is in the spirit of the
%pioneering work \cite{floerWeinstein86} due to Floer and Weinstein
%for the standing wave problem in the cubic Sch\"odinger equation.

\medskip
\subsection{The Main Result.} We consider equation
\eqref{IntrodEq1} in a slightly general form, but we restrict
ourselves to dimension $N=2$. More precisely, let us consider the
equation
\begin{equation}
\ep^2\div(a(x)\nabla u(x))- a(x)F'(u)=0,\quad \hbox{in
}\R^2\label{IntrodEq12}
\end{equation}
where the function $a(x)$ is a smooth positive function.

\medskip
As far as our knowledge goes, little is known about entire solutions
to~\eqref{IntrodEq12} having a single transition close to a
noncompact curve, in the case that $a(x)$ is not identically
constant.

In this work, we will consider a smooth noncompact curve
$\Gamma=\gamma(\R)$, where $\gamma:\R\to\Gamma\subset\R^2$ is
parameterized by arc-length. We denote by $\nu:\Gamma\to\R^2$ the
choice of the normal vector to $\Gamma$, so that the curve is
positively oriented.

\medskip
In order to state the main result, let us first list our set of
assumptions on the function $a(x)$ and the curve $\Gamma$.

As for the function $a:\R^2\to \R$, we assume that
\begin{equation}\label{smoothnessfora}
a\in C^5(\R^2),\quad 0<m < a(x)\leq M
\end{equation}
for some positive constants $m,M$. We also assume that the curvature
$k(s)$ of $\Gamma$ in arch-length variable satisfies
\begin{equation}
|k(s)|+|k'(s)|+|k''(s)|\leq
\frac{C}{(1+|s|)^{1+\alpha}}\label{IntrodEq15}
\end{equation}
for some $\alpha>0$. In particular, condition~\eqref{IntrodEq15}
implies that
$$
\dot{\gamma}_{\pm}:=\lim_{s\to\pm \infty}\dot{\gamma}(s)
$$
are well defined directions in $\R^2$. We must assume further some
non-parallelism condition, namely
\begin{equation}\label{nonparallelism}
-1\leq \dot{\gamma}_{+} \cdot \dot{\gamma}_{-}<1.
\end{equation}

Points $x\in\R^2$ that are $\delta-$close to this curve, with
$\delta$ small, can be represented using Fermi coordinates as
follows
$$
x=\gamma(s)+z\nu(s)=:X(s,z),\quad |z|<\delta,\quad s\in\R.
$$

Condition \eqref{nonparallelism} implies that the mapping
$x=X(s,z)$
provides local coordinates in a region of the form
$$
\mathcal{N}_{\delta}:=\left\{x=X(s,z): |z|<\delta+c_0|s|\right\}
$$
for some fixed constant $c_0$ and the mapping $x\mapsto (s,z)$
defines a local diffeomorphism.

\medskip
Finally, abusing notation, by setting $a(s,z)=a(X(s,z))$, we assume
that
\begin{equation}
|\nabla_{s,z}a(s,z)|\leq \frac{C}{(1+|s|)^{1+\alpha}},\quad
|D^2_{s,z}\,a(s,z)|\leq
\frac{C}{(1+|s|)^{2+\alpha}}\label{IntrodEq14}
\end{equation}
where $\alpha>0$ is as in \eqref{IntrodEq15}.

\medskip
Any smooth curve $\delta-$close to $\Gamma$ in a $C^{m}-$topology
can be parameterized by
$$ \gamma_{h}(s)=\gamma(s)+h(s)\nu(s)$$
where $h$ is a small $C^{m}-$function. The weighted length of
$\Gamma_{h}$, where $\gamma_h:\R\to \Gamma_h$ is given by
\begin{align*}
l_{\Gamma}(h)&:=\int_{\Gamma_{h}}a(x)d\vec{r}=\int^{+\infty}_{-\infty}a\left(\gamma_{h}(s)\right)|\dot{\gamma}_{h}(s)|ds\\
&=\int^{+\infty}_{-\infty}a(s,h(s))|\dot{\gamma}+h\dot{\nu}+h'\nu|ds.
\end{align*}

Since $|\dot{\gamma}|=1$ and $\dot{\nu}(s)=k(s)\dot{\gamma}(s)$,
where $k$ is the signed curvature of $\Gamma$, we find that
$$
l_{h}(h)=\int^{+\infty}_{-\infty}a(s,h(s))[(1+kh)^2+|h'|^2]^{1/2}ds.
$$

We say that $\Gamma$ is a stationary (or a critical) curve respect
to the function $a(x)$, if and only if,
\begin{align*}
l'_{\Gamma}[h]&=\int_{\Gamma_{h}}(\partial_{z}a(s,0)-a(s,0)k(s))h(s)ds
=0,\quad \forall h\in C^{\infty}_c(\R)
\end{align*}
which is equivalent to say that the curve $\Gamma$ satisfies
\begin{equation}
\partial_z a(s,0)=k(s)a(s,0),\quad s\in\R.\label{IntrodEq13}
\end{equation}

Regarding the stability properties of the stationary curve $\Gamma$
and the second va\-riation of the length functional $l_{\Gamma}$,
$$
l''_{\Gamma}(h,h)=\int^{+\infty}_{-\infty}\big(\, a(s,0)|h'(s)|^2
+[\partial_{zz}a(s,0)-2k^2(s)]h^2(s)\,\big)ds
$$
we have the Jacobi operator of $\Gamma$, corresponding to the linear
differential operator
\begin{equation}
\J_{a,\Gamma}(h)=h''(s)+\frac{\partial_{s}a(s,0)}{a(s,0)}h'(s)
-\left[\partial_{zz}a(s,0)-2k^2(s)\right]h(s).\
\label{IntrodEqExtra}
\end{equation}

We say that the stationary curve $\Gamma$ is also nondegenerate
respect to the potential $a(x)$, if the bounded kernel of
$\J_{a,\Gamma}$ is the trivial one. The nondegeneracy condition
basically implies that $\J_{a,\Gamma}$ has an appropriate right
inverse,  so that the curve is isolated in a properly chosen
topology.

\medskip
Next, we proceed to state the main result.

%%%--------------------------------- BEGIN THEOREM ---------------------------------------------------------------------------------%%%
\begin{thm}\label{TheoremMainResult}
Assume that $a(x)$ is a smooth potential and let $\Gamma$ be a
smooth stationary and nondegenerate curve respect to the length
functional  $\int_{\Gamma} a(x)d \vec{r}$. Assume also that
conditions \eqref{smoothnessfora}-\eqref{IntrodEq14} are satisfied.
Then for any $\ep>0$ small enough, there exists a smooth bounded
solution $u_{\ep}$ to the inhomogeneous Allen-Cahn
equation~\eqref{IntrodEq12}, such that
\begin{align*}
u_{\ep}(x)=w\left(\frac{z-h(s)}{\ep}\right)+\mathit{O}(\ep^2), \quad\text{
for }\quad x=X(s,z), \quad |z|<\delta
\end{align*}
where the function $h$ satisfies
\begin{equation*}
\|h\|_{C^1(\R)}\leq C\ep\label{ThmEq2}.
\end{equation*}

This solution converges to $\pm1$, away from $\Gamma$, namely
\begin{align*}
u^2_{\ep}(x)\to \pm 1, \quad dist(\Gamma,x) \to
\infty\label{ThmEq3}.
\end{align*}
\end{thm}

\begin{rmk} Throughout the proof of this theorem, we obtain an explicit
description for $u_{\ep}$ and its derivatives. We apply infinite
dimensional reduction method in the spirit of the pioneering work
due to Floer and Weinstein \cite{10}.
\end{rmk}
\medskip

The paper is organized as follows. Section 2 deals with the
geometrical setting need to set up the proof of theorem
\ref{TheoremMainResult}. In section 3 we present the invertibility
theory for the Jacobi operator of the curve $\Gamma$ while in
section 4 we give some examples of the function $a(x)$ and the curve
$\Gamma$, for which our result applies.  Section 5 is devoted to
find a good approximation of a solution to \eqref{IntrodEq12}.
Next,we sketch the proof of theorem \ref{TheoremMainResult} in
section 6, while leaving complete details for subsequent sections.

\section{Geometrical background}\label{chap:Two}

In this section we write the differential operator
\begin{equation}\label{diffOper1}
\ep^2\lap_{\bar{x}}u +\ep^2\dfrac{\nabla_{\bar{x}} a}{a}\cdot
\nabla_{\bar{x}} u
\end{equation}
involved in equation \eqref{IntrodEq12}, in some appropriate
coordinate system.

\medskip
First, observe that the obvious scaling $\bar{x}=\ep x$ and setting
$v(x):=u(\ep x)$, transforms \eqref{diffOper1} into
\begin{equation}
\lap_{x} v+\ep\dfrac{\nabla_{\bar{x}} a}{a}\cdot \nabla_{x}
v\label{EqAllenCahn}
\end{equation}

Let us consider a large dilation of the curve $\Gamma$, namely
$\Gamma_{\ep}:=\ep^{-1}\Gamma$, for  $\ep>0$ small. Next, we
introduce local translated Fermi coordinates near $\Gamma_{\ep}$
\begin{eqnarray*}
 X_{\ep,h}(s,t)&=&X_{\ep}(s,t+h(\ep s))\nonumber\\
&=& \frac{1}{\ep}\gamma(\ep s)\,+\,(t\,+\,h(\ep s))\,\nu(\ep s)
\end{eqnarray*}
where $h\in C^2(\R)$ is a bounded smooth function. From assumption
\eqref{nonparallelism}, we see that the mapping $X_{\ep,h}(s,t)$
gives local coordinates in the region
\begin{equation*}
\Nn_{\ep,h}=\left\{x=X_{\ep,h}(s,t)\in \R^2:  |t+h(\ep
s)|<\frac{\delta}{\ep}+c_0|s|\right\}
\label{DilatedTranslatedNeighborhood}
\end{equation*}
which is a dilated tubular neighborhood around $\Gamma_\ep$
translated in $h$.

\medskip
Now, we give an expression for the euclidean laplacian in terms of
the coordinates $X_{\ep,h}$. A detailed proof of this fact can be
found in the Appendix.

\begin{lemma}\label{LemmaEuclideanLaplacian}
On the open neighborhood $\Nn_{\ep,h}$ of $\Gamma_{\ep}$, the
euclidean laplacian has the following expression when is computed in
the coordinate $x=X_{\ep,h}(s,t)$, which reads as
\begin{eqnarray}
\Delta_{X_{\ep,h}}&=&\partial_{tt}+\partial_{ss}-2\ep h'(\ep
s)\partial_{st}
-\ep^2 h''(\ep s)\partial_{t}+\ep^2 |h'(\ep s)|^2\partial_{tt}\notag\\[0.25cm]
&&-\,\ep[k(\ep s)+\ep(t+h(\ep s))k^2(\ep s)]\partial_{t}+D_{\ep,h}(s,t)\label{LaplacianXeph}
\end{eqnarray}
where
    \begin{align*}
    D_{\ep,h}(s,t)&:=\ep(t+h)A_0(\ep s,\ep(t+h))[\partial_{ss}-2\ep h'(\ep s)\partial_{ts}-\ep^2h''(\ep s)
    \partial_t+\ep^2|h'(\ep s)|^2\partial_{tt}]\\
    &+\ep^2(t+h)B_0(\ep s,\ep(t+h))[\partial_{s}-\ep h'(\ep s)\partial_t]\\
    &+\ep^3(t+h)^2C_0(\ep s,\ep(t+h))\partial_{t}
    \end{align*}
for which
    \begin{align}
    A_0(\ep s,\ep [t+h(\ep s)])&=2k(\ep s)+\ep \mathit{O}(|[t+h(\ep s)]k^2(\ep s)|)\label{ErrorA0LaplacianXeph}\\
    B_0(\ep s,\ep [t+h(\ep s)])&=\dot{k}(\ep s)+\ep \mathit{O}(|(t+h(\ep s))\dot{k}(\ep s)\cdot k^2(\ep s)|)\label{ErrorB0LaplacianXeph}\\
    C_0(\ep s,\ep [t+h(\ep s)])&=k^3(\ep s)+\ep \mathit{O}(|(t+h(\ep s))k^4(\ep s)|) \label{ErrorC0LaplacianXeph}
    \end{align}
are smooth functions and these relations can be differentiated.
\end{lemma}

Next, we derive an expression for the second term in
equation~\eqref{EqAllenCahn}, in terms of the Fermi coordinates. We
collect the computations in the following lemma, whose proof can be
found also in the Appendix.

\begin{lemma}{\label{LemmaEuclideanGradients}
In the open neighborhood $\Nn_{\ep,h}$ of $\Gamma_{\ep}$, we have
the validity of the fo\-llowing expression
$$
\ep \dfrac{\nabla_{X}a}{a}\cdot \nabla_{X_{\ep,h}}\,= \,\ep
\left[\frac{\partial_{\bs{t}}a}{a}(\ep s,0)+\ep (t+h(\ep s))
\left(\frac{\partial_{\bs{tt}}a}{a}(\ep
s,0)-\biggl\lvert\frac{\partial_{\bs{t}}a}{a}(\ep
s,0)\biggr\rvert^2\right)\right]
\partial_{t}\notag
$$
\begin{equation}\label{GradientsXeph}
+\,\ep\frac{\partial_{\bs{s}}a}{a}(\ep s,0)[\partial_{s}-\ep h'(\ep
s)\partial_{t}]\notag +E_{\ep,h}(s,t)
\end{equation}
where
\begin{align}
E_{\ep,h}(s,t)&:=\ep^2 (t+h(\ep s)) D_0(\ep
s,\ep(t+h))[\partial_{s}-\ep h'(\ep s)\partial_{t}]\notag\\[0.2cm]
&+\ep^3 (t+h(\ep s))^2F_0(\ep s,\ep(t+h))\partial_{t}\label{ErrorEGradientsXeph}
    \end{align}
and for which the next functions are smooth
    \begin{align*}
    D_0(\ep s ,\ep (t+h))&=\partial_{\bs{t}}\left[\frac{\partial_{\bs{s}}a}{a}\right](\ep s,0)
    +\ep \mathit{O}\left((t+h(\ep s))\partial_{\bs{t}\bs{t}}\left[\frac{\partial_t a}{a}\right]\right)\\[0.1cm]
    &+A_0(\ep s,\ep (t+h))\frac{\partial_{\bs{s}}a}{a}(\ep s,\ep(t+h))\\
    F_0(\ep s ,\ep (t+h))&=\dfrac{1}{2}\partial_{\bs{t}\bs{t}}\left[\frac{\partial_t a}{a}\right](\ep s,0)+\ep
    \mathit{O}\left((t+h(\ep s))\partial_{\bs{t}\bs{t}\bs{t}}\left[\frac{\partial_t a}{a}\right]\right)
    \end{align*}
and where $A_0(\ep s,\ep (t+h))$ given in
\eqref{ErrorA0LaplacianXeph}. Further, these relations can be
differentiated. }\end{lemma}

\medskip
Using lemmas \ref{LemmaEuclideanLaplacian} and
\ref{LemmaEuclideanGradients}, the fact that the curve $\Gamma$
satisfies condition \eqref{IntrodEq13}, we can write expression
\eqref{EqAllenCahn} in coordinates $x=X_{\ep,h}(s,t)$ as
$$
\Delta_{X_{\ep,h}}+\ep \dfrac{\nabla_{X}a(\ep x)}{a(\ep x)}
\cdot\nabla_{X_{\ep,h}}\,=\,
$$
$$
\partial_{tt} \,+\,
\partial_{ss}\,+\, \ep\frac{\partial_{\bs{s}}a}{a}(\ep s,0)
\partial_{s}
$$

$$
\,-\,\ep^2\left\{ h''(\ep s)+ \dfrac{\partial_{\bs{s}}a}{a}(\ep
s,0)\,h'(\ep s) +\left[2\,k^2(\ep
s)-\dfrac{\partial_{\bs{t}\bs{t}}a}{a}(\ep s,0)\right]h(\ep
s)\right\}\partial_{t}
$$

$$
-\ep^2\left[k^2(\ep s)-\dfrac{\partial_{\bs{t}\bs{t}}a}{a}(\ep s,0)
+\biggl\lvert\dfrac{\partial_{\bs{t}}a}{a}(\ep s,0)
\biggr\rvert^2\right]t\,\partial_{t} -2\ep h'(\ep s)\partial_{st}
+\ep^2|h'(\ep s)|^2\partial_{tt}
$$

$$
+\,\ep(t+h(\ep s))A_0(\ep s,\ep(t+h))[\partial_{ss}-2\ep h'(\ep s)
\partial_{ts}-\ep^2h''(\ep s)\partial_t+\ep^2|h'(\ep s)|^2\partial_{tt}
$$

\begin{equation}
+\,\ep^2\,(t+h(\ep s))\tilde{B}_0(\ep s,\ep(t+h))[\partial_{s}- \ep
h'(\ep s)
\partial_t]\,+\,\ep^3(t+h(\ep s))^2\tilde{C}_0(\ep
s,\ep(t+h))\partial_{t}\label{EcuacionAllen-CahnEnVecindadDeFermiEq4}
\end{equation}
where
\begin{equation}
\tilde{B}_0(\ep s,\ep(t+h)):=B_0(\ep s,\ep(t+h))+ D_0(\ep
s,\ep(t+h))\label{EcuacionAllen-CahnEnVecindadDeFermiEq5}
\end{equation}
\begin{equation}
\tilde{C}_0(\ep s,\ep(t+h)):=C_0(\ep s,\ep(t+h))+F_0(\ep s,\ep
(t+h)).\label{EcuacionAllen-CahnEnVecindadDeFermiEq6}
\end{equation}

We postpone detailed computations to the Appendix, in order to
keep the presentation as clear as possible.

\section{The Jacobi Operator $\J_{a,\Gamma}$}\label{section3}

This section is meant to provide a complete proof of proposition
\ref{PropNonlinearJacobiTheory}. Recall that the Jacobi operator of
the curve $\Gamma$ associated to the potential $a$, corresponds to
the linear operator
\begin{align}
\J_{a,\Gamma}[h](\bs{s})= h''(\bs{s})+\dfrac{\partial_{\bs{s}}
a(\bs{s},0)}{a(\bs{s},0)}h'(\bs{s})-Q(\bs{s})h(\bs{s})
\label{JacobiOperator}
\end{align}

where we recall that
\begin{align}
Q(\bs{s}):=\biggl[\dfrac{
\partial_{\bs{t}\bs{t}}a(\bs{s},0)}{a(\bs{s},0)}
-2k^2(\bs{s})\biggr]\label{OperatorQ}
\end{align}

Recall also that we are assuming the curve $\Gamma$ to be
nondegenerate, which means that the only bounded solution to
$$\J_{a,\Gamma}[h](\bs{s})=0,\quad \forall \bs{s}\in\R$$
is the trivial one.

\medskip

\subsection{The kernel of  $\J_{a,\Gamma}$. } In order to
find accurate information on the kernel of \eqref{JacobiOperator},
we consider the auxiliary equation
\begin{align}
\frac{d}{d\bs{s}}\left(p(\bs{s})\frac{d}{d\bs{s}}h\right)-q(\bs{s})h=0,
\;\text{ in }\;\R\label{AE}
\end{align}
where we assume that $p,q:\R\to\R$ satisfy the following
\begin{align}
&p\in C^1[0,\infty)\cap L^{\infty}[0,\infty),\ q\in C^1[0,\infty)\label{StudyKernelJacobiOpEq1}\\[0.2cm]
&p(\bs{s})\geq p_0>0,\quad \forall \bs{s}\geq0
\label{StudyKernelJacobiOpEq2_1}\\[0.2cm]
&\lim_{\bs{s}\to\pm\infty}p(\bs{s})=:p(\pm \infty)>0
\label{StudyKernelJacobiOpEq2}\\[0.2cm]
&|p(\bs{s})|+(1+|\bs{s}|)^{2+\alpha}|p'(\bs{s})|\leq C,\quad \forall
\bs{s}\geq0
\label{StudyKernelJacobiOpEq3}\\[0.2cm]
&|q(\bs{s})|+|q'(\bs{s})|\leq \frac{C}{1+|\bs{s}|^{2+\alpha}},\quad
\forall \bs{s}\geq0\label{StudyKernelJacobiOpEq4}
\end{align}
for some constants $\alpha>-1$, $\beta_0>0$ and $C>0$.

The first result concerns the decay for the derivative of a solution
to the auxiliary equation, provided that $p$ and $q$ decay
sufficiently fast.
%%%-----------------------------------------------------------------------------------------------------------------------------------%%%
\begin{lemma}\label{LemmaDecayDerivativeJacobi}
Suppose $\alpha>-1$, and consider a one-sided bounded solution $h\in
L^{\infty}[0,\infty)$ of $\eqref{AE}$, for which functions $p$ and
$q$ fulfill \eqref{StudyKernelJacobiOpEq1} to
\eqref{StudyKernelJacobiOpEq4}. Then there is a constant
$C=C(p,q,\alpha,h)>0$ such that
\begin{align*}
|h'(\bs{s})|\leq \frac{C}{|\bs{s}|^{1+\alpha}}, \quad \forall
\bs{s}>0
\end{align*}
where
$C(p,q,\alpha,h)=\|p^{-1}\|_{L^{\infty}[0,\infty)}\|h\|_{L^{\infty}
[0,\infty)}\|(1+|\bs{s}|)^{2+\alpha}q\|_{L^{\infty}[0,\infty)}$.
\end{lemma}

\begin{proof}
Observe first that thanks to assumptions
\eqref{StudyKernelJacobiOpEq1}-\eqref{StudyKernelJacobiOpEq2}, it
holds
\begin{align}
p(\bs{s})=p(+\infty)-\int^{+\infty}_{\bs{s}}p'(\xi)d\xi\label{StudyKernelJacobiOpEq5}
\end{align}
Now, since $h$ solves the equation, then for $\bs{s}_1>\bs{s}_2>0$
we have
\begin{align*}
|p(\bs{s}_1)h'(\bs{s}_1)-p(\bs{s}_2)h'(\bs{s}_2)|&\leq
\int_{\bs{s}_2}^{\bs{s}_1}\left|q(\bs{s})h(\bs{s})\right|\\
&\leq \|h\|_{L^{\infty}[0,\infty)}\|(1+|\bs{s}|)^{2
+\alpha}q\|_{L^{\infty}[0,\infty)}\left| \int_{\bs{s}_2}^{\bs{s}_1}\frac{1}{1+|\xi|^{2+\alpha}}d\xi\right|\\
&\leq C(q,h)\left(\frac{1}{|\bs{s}_2|^{1+\alpha}}-\frac{1}{|\bs{s}_1|^{1+\alpha}}\right)
\end{align*}
where $C(q,h):=C\|h\|_{L^{\infty}[0,\infty)}\|(1+|\bs{s}|)^{2+\alpha}q\|_{L^{\infty}[0,\infty)}<\infty$ is fixed.

In particular using that $1+\alpha>0$, it follows that
$$ \lim_{\bs{s}_1\to+\infty}|p(\bs{s}_1)h'(\bs{s}_1)|\leq |p(\bs{s}_2)
h'(\bs{s}_2)|+C(q,h)\frac{1}{|\bs{s}_2|^{1+\alpha}}<+\infty$$ which
implies that $p(+\infty)h'(\infty)\in\R$. From this, we can rewrite
equation \eqref{AE} in its integral form
\begin{align}
p(\bs{s})h'(\bs{s})=p(+\infty)h'(+\infty)-\int^{+\infty}_{\bs{s}}
q(\xi)h(\xi) d\xi.\label{StudyKernelJacobiOpEq6}
\end{align}

but using \eqref{StudyKernelJacobiOpEq5}, we find that
\begin{align}
&p(+\infty)h'(\bs{s})-h'(\bs{s})\int^{+\infty}_{\bs{s}}p'(\xi)d\xi=p(+\infty)
h'(+\infty)-\int^{+\infty}_{\bs{s}}q(\xi)h(\xi) d\xi\notag
\end{align}

and so
\begin{align}
&p(+\infty)h'(s)=p(+\infty)h'(+\infty)+h'(s)\int^{+\infty}_{s}p'(\xi)d\xi-\int^{+\infty}_{s}q(\xi)h(\xi)
d\xi.\notag
\end{align}

Integrating again between $0$ and $\bs{s}$, we obtain an expression
for the solution $h$ of \eqref{AE}
\begin{align}
p(+\infty)h(s)&=p(+\infty)h(0)+p(\infty)h'(+\infty)s\notag\\[0.3cm]
&+\underbrace{\int_{0}^{s}h'(\xi)\int^{+\infty}_{\xi}p'(\tau)d\tau
d\xi}_{I}-\underbrace{\int^{s}_{0}\int^{+\infty}_{\xi}q(\tau)h(\tau)d\tau
d\xi}_{II}\label{StudyKernelJacobiOpEq7}
\end{align}
Let us estimate these integrals. We first estimate integral I
\begin{align*}
|I|&\leq \int_{0}^{s}|h'(\xi)|\int^{+\infty}_{\xi}|p'(\tau)|d\tau\\
&\leq C\|h'\|_{L^{\infty}[0,\infty)}\|(1+|\bs{s}|^{2+\alpha})
p'\|_{L^{\infty}[0,\infty)}
\int_{0}^{\bs{s}}\int^{+\infty}_{\xi}\frac{1}{1+|\tau|^{2+\alpha}}
d\tau d\xi\\
&\leq
C_{h',p',\alpha}\int_{0}^{\bs{s}}\frac{1}{1+|\xi|^{1+\alpha}}d\xi
=\mathit{O}(1+|\bs{s}|^{-\alpha})
\end{align*}
where $C_{h',p',\alpha}:=C\|h'\|_{L^{\infty}[0,\infty)}\|(1+|\bs{s}|^{2+\alpha})p'\|_{L^{\infty}[0,\infty)}$.\\

\medskip
In the same way, we estimate II
\begin{align*}
|II|&\leq \int^{\bs{s}}_{0}\int^{+\infty}_{\xi}|q(\tau)|\ |h(\tau)|
d\tau d\xi\\
&\leq C\|h\|_{L^{\infty}[0,\infty)}\|(1+|\bs{s}|^{2+\alpha})q
\|_{L^{\infty}[0,\infty)}\int^{\bs{s}}_{0}\int^{+\infty}_{\xi}
\frac{d\tau d\xi}{1+|\tau|^{2+\alpha}}\\
&\leq C_{h,q,\alpha}(1+|\bs{s}|)^{-\alpha}
\end{align*}
with $C_{h,q,\alpha}:=C\|h\|_{L^{\infty}[0,\infty)}\|(1+|\bs{s}|)^{2+\alpha}q\|_{L^{\infty}[0,\infty)}$.\\

Since $h$ is bounded, we deduce from \eqref{StudyKernelJacobiOpEq7}
that
\begin{align}
\mathit{O}(1)=p(+\infty)h(0)+p(+\infty)h'(+\infty)\bs{s}+\mathit{O}(1+
|\bs{s}|^{-\alpha}). \label{StudyKernelJacobiOpEq8}
\end{align}

Dividing \eqref{StudyKernelJacobiOpEq8} by $\bs{s}>0$ and taking
$\bs{s}\to+\infty$, we get that
\begin{align*}
0= p(+\infty)h'(+\infty)
\end{align*}

provided that $\alpha>-1$. From \eqref{StudyKernelJacobiOpEq2}, it
follows that $h'(+\infty)=0$.

\medskip
In particular, the latter fact together with formula
\eqref{StudyKernelJacobiOpEq6}, imply that
\begin{align*}
p(\bs{s})h'(\bs{s})&=\int^{\infty}_{\bs{s}}q(\xi)h(\xi)d\xi
\end{align*}

and consequently
\begin{align*}
\quad |h'(\bs{s})|&\leq
C\|p^{-1}\|_{L^{\infty}[0,\infty)}\|h\|_{L^{\infty}[0,\infty)}\|(1+
|\bs{s}|^{2+\alpha})q\|_{L^{\infty}[0,\infty)} \dfrac{1}{1+|\bs{s}|^{1+\alpha}}
\end{align*}

which completes the proof of the estimate.
\end{proof}

The core of this section is reflected in the next result.

\begin{lemma}{ \label{LemmaBehaviourKernel}
Let $\alpha>0$, and suppose function $q$ satisfies
\eqref{StudyKernelJacobiOpEq1}-\eqref{StudyKernelJacobiOpEq4}. Then
the equation
\begin{align}
u''(\bs{s})-q(\bs{s})u(\bs{s})=0,\quad \text{ in }\quad \R
\label{StudyKernelJacobiOpEq9}
\end{align}
has two linearly independent smooth solutions $u,\tilde{u}$, so that
as $s\to+\infty$
      \begin{align*}
       u(\bs{s})&=\bs{s}+\mathit{O}(1)+\mathit{O}(|\bs{s}|^{-1-\alpha}),\qquad \tilde{u}(\bs{s})
       =1+\mathit{O}(|\bs{s}|^{-1}+|\bs{s}|^{-\alpha})\\[0.2cm]
       u'(\bs{s})&=1+\mathit{O}(|\bs{s}|^{-1}+|\bs{s}|^{-\alpha}),
       \qquad \tilde{u}'(\bs{s})=\mathit{O}(|\bs{s}|^{-1}+|\bs{s}|^{-1-\alpha})       
      \end{align*}
}
\end{lemma}

\begin{proof}
To begin with, we look for a solution $u(\bs{s})=\bs{s}v(\bs{s})$,
so that, multiplying equation \eqref{StudyKernelJacobiOpEq9} by
$\bs{s}$, we find that $v$ satisfies
\begin{align}
\frac{d}{d\bs{s}}\left(\bs{s}^2v'(\bs{s})\right)-q(\bs{s})\bs{s}^2
v(\bs{s})=0 \label{StudyKernelJacobiOpEq12}
\end{align}

Now, consider the functions
\begin{align*}
x(\bs{s}):=\bs{s}^2v'(\bs{s}), \quad
y(\bs{s}):=v(\bs{s})
\end{align*}
so that, equation~\eqref{StudyKernelJacobiOpEq12} becomes the linear
system of differential equations
\begin{equation*}
\left\{
\begin{aligned}
x'(\bs{s})&=q(\bs{s})\bs{s}^2y(\bs{s})\\[0.1cm]
y'(\bs{s})&=\frac{1}{\bs{s}^2}x(\bs{s})
\end{aligned}
\right. ,\quad \forall
\bs{s}\in[\bs{s}_0,+\infty)
\end{equation*}

Integrating this system between $\bs{s}_0$ and $\bs{s}$ we obtain
the identities
\begin{align}
y(\bs{s})&=y(\bs{s}_0)+\int_{\bs{s}_0}^{s}\frac{1}{\xi^2}x(\xi)d\xi\notag\\[0.2cm]
x(\bs{s})&=x(\bs{s}_0)+\int_{\bs{s}_0}^{\bs{s}}q(\xi)\xi^2y(\xi)d\xi
\label{IntegralX}
\end{align}
In particular, we deduce an explicit formula for $y(\bs{s})$, given
by
\begin{align}
y(\bs{s})&=y(\bs{s}_0)+x(\bs{s}_0)\left(\frac{1}{\bs{s}_0}-
\frac{1}{\bs{s}}\right)+
\int^{\bs{s}}_{\bs{s}_0}y(\tau)q(\tau)\tau^2
\left(\frac{1}{\tau}-\frac{1}{\bs{s}}\right)
d\tau\label{StudyKernelJacobiOpEq15}
\end{align}

In this way, we can estimate $y(\bs{s})$ for $\bs{s}\geq \bs{s}_0$
as
\begin{align}
|y(\bs{s})|&\leq |y(\bs{s}_0)|+|x(\bs{s}_0)|\left(\frac{1}{\bs{s}_0}
-\frac{1}{\bs{s}}\right)+\int^{\bs{s}}_{\bs{s}_0} |y(\tau)|\
|q(\tau)|\tau\left(1-\frac{\tau}{\bs{s}}\right)d\tau.\notag
\end{align}

From Gronwall's inequality we find the estimate
\begin{align*}
|y(\bs{s})|&\leq
\left(|y(\bs{s}_0)|+\frac{2|x(\bs{s}_0)|}{\bs{s}_0}\right)
\exp\left(\int^{\bs{s}}_{\bs{s}_0}|q(\tau)|
\tau\left(1-\frac{\tau}{\bs{s}}\right)d\tau\right)
\end{align*}

Notice that, for any $\bs{s}\geq\tau>\bs{s}_0:\;
\left|\tau\left(1-\frac{\tau}{\bs{s}}\right)\right|\leq
2\tau=\mathit{O}(\tau)$. This fact combined with the decay of $q(\bs{s})$,
leads to
\begin{align}
|y(\bs{s})|\leq
C_{q,\alpha}(|y(\bs{s}_0)|+\frac{2}{\bs{s}_0}|x(\bs{s}_0)|)
\label{StudyKernelJacobiOpEq16}
\end{align}
where $C_{q,\alpha}:=C\|(1+|\bs{s}|)^{2+\alpha}q\|_{L^{\infty}[
\bs{s}_0,+\infty)}\int^{\infty}_{\bs{s}_0}|\tau|^{-1-\alpha}d\tau$.

From \eqref{StudyKernelJacobiOpEq15} it follows that for any
$\bs{s}_1>\bs{s}_2\geq \bs{s}_0>0$:
\begin{align*}
|y(\bs{s}_1)-y(\bs{s}_2)|&\leq
|x(\bs{s}_0)|\left(\frac{1}{\bs{s}_2}-\frac{1}{\bs{s}_1}\right)
+C\int^{\bs{s}_1}_{\bs{s}_2}|q(\tau)|\tau d\tau
\end{align*}

implying that $y(+\infty)\in\R$. Moreover, same formula
\eqref{StudyKernelJacobiOpEq15} yields
$$
y(+\infty)=y(\bs{s}_0)+\frac{x(\bs{s}_0)}{\bs{s}_0}+
\int^{+\infty}_{\bs{s}_0} y(\tau)q(\tau)\tau d\tau
$$

which allows us to write
\begin{align*}
y(\bs{s})-y(+\infty)=-\frac{x(\bs{s}_0)}{\bs{s}_0}-
\int^{\bs{s}}_{\bs{s}_0}y(\tau)q(\tau)\frac{\tau^2}{\bs{s}}
d\tau-\int^{+\infty}_{\bs{s}}y(\tau)q(\tau)\tau d\tau
\end{align*}

In particular, by choosing the constants to be $y(+\infty)=1,\
x(\bs{s}_0)=0$, we finally deduce
\begin{align}
y(\bs{s})=1-\int^{\bs{s}}_{\bs{s}_0}y(\tau)q(\tau)\frac{\tau^2}{\bs{s}}
d\tau-\int^{+\infty}_{\bs{s}}y(\tau)q(\tau)\tau d\tau
\label{StudyKernelJacobiOpEq18}
\end{align}

Additionally, the derivative $y'(\bs{s})=v'(\bs{s})$ can be obtained
from $x(\bs{s})$ using relation \eqref{IntegralX}, as
\begin{align}
v'(\bs{s})=\frac{x(\bs{s})}{\bs{s}^2}=\frac{0}{\bs{s}^2}+
\frac{1}{\bs{s}^2} \int^{\bs{s}}_{\bs{s}_0}q(\xi)\xi^2
y(\xi)d\xi.\label{StudyKernelJacobiOpEq19}
\end{align}
Now that $y(\bs{s})$ is bounded in $[\bs{s}_0,+\infty)$, similar
arguments as shown in \eqref{StudyKernelJacobiOpEq16} imply the same
estimates for the integrals in
\eqref{StudyKernelJacobiOpEq18}-\eqref{StudyKernelJacobiOpEq19},
since
$$
\left|\int^{+\infty}_{\bs{s}}y(\tau)q(\tau)\tau d\tau\right|
=\mathit{O}(|\bs{s}|^{\alpha}),\quad
\left|\int^{\bs{s}}_{\bs{s}_0}y(\tau)q(\tau)\frac{\tau^2}{\bs{s}}
d\tau\right|=\mathit{O}(|\bs{s}|^{-1}+|\bs{s}|^{-\alpha})
$$

From these estimates, we conclude that
\begin{align*}
 &v(\bs{s})=y(\bs{s})=1+\mathit{O}(|\bs{s}|^{-1}+|\bs{s}|^{-\alpha})\\[0.2cm]
 &v'(\bs{s})=\mathit{O}(|\bs{s}|^{-2}+|\bs{s}|^{-1-\alpha})
\end{align*}
So the asymptotic behavior of the first solution follows, as
$\alpha>0$ and by definition of $u$:
\begin{equation*}
\begin{aligned}
&u(\bs{s})=\bs{s}\left(1+\mathit{O}(|\bs{s}|^{-1}+|\bs{s}|^{-\alpha})\right)=
\bs{s}+\mathit{O}(1+|\bs{s}|^{1-\alpha})\\[0.2cm]
&u'(\bs{s})=v(\bs{s})+\bs{s}v'(\bs{s})=1+\mathit{O}(|\bs{s}|^{-1}+|\bs{s}|^{-\alpha})
\end{aligned}, \quad \bs{s}\geq \bs{s}_0
\end{equation*}

which finishes the analysis of the profile of the first solution to
equation~\eqref{StudyKernelJacobiOpEq9}.

To conclude, we find $\tilde{u}$ using reduction of order formula,
to find that
$$
\tilde{u}(\bs{s})=\left(\int^{\infty}_{\bs{s}}u^{-2}(\xi)d\xi\right)\cdot
u(\bs{s})
$$

and directly from this, one gets
\begin{align*}
\begin{aligned}
&\tilde{u}(\bs{s})=C(\bs{s})u(\bs{s})=1+
\mathit{O}(|\bs{s}|^{-1}+|\bs{s}|^{-\alpha})\\[0.2cm]
&\tilde{u}'(\bs{s})=C'(\bs{s})u(\bs{s})+C(\bs{s})u'(\bs{s})=
\mathit{O}(|\bs{s}|^{-1}+|\bs{s}|^{-2}+|\bs{s}|^{-1-\alpha})
\end{aligned}\;\;, \quad \text{ for } \bs{s}>>\bs{s}_0
\end{align*}
which concludes the proof of lemma \ref{LemmaBehaviourKernel}.
\end{proof}

Now proceed to state the main result of this section, which
characterize the profile of the kernel of the Jacobi operator.

%%-------------------------------------------------------------------------------------------------------------------------------
\begin{prop}\label{PropConstructionKernel}
Let $\Gamma\subset\R^2$ be a stationary non-degenerate curve as in
respect to $a$. Assume also that conditions
\eqref{smoothnessfora}-\eqref{IntrodEq14} are satisfied, for
$\alpha>0$ and additionally the potential stabilizes on the curve at
infinity, namely
\begin{equation*}
a(\pm\infty,0):=\lim_{\bs{s}\to\pm\infty}a(\bs{s},0)>0\in\R.
\end{equation*}

Then, there are two linearly independent elements in the kernel of
$h_1,h_2$  of \eqref{JacobiOperator} satisfying that
\begin{equation}
\begin{aligned}
h_i(\bs{s})&=|\bs{s}|+\mathit{O}(1)+\mathit{O}(|\bs{s}|^{-1}+|\bs{s}|^{-\alpha})\\
h'_i(\bs{s})&=\mathit{O}(1)+\mathit{O}(|\bs{s}|^{-1}+|\bs{s}|^{-1-\alpha})
\end{aligned},\quad \text{ as }\quad (-1)^{i}\bs{s}\to
+\infty\label{PropConstrucionKernelEq2}
\end{equation}
and they are bounded functions as $(-1)^{i+1}\bs{s}\to\infty$.
Furthermore, in the region where the latter happens, it holds
\begin{align}
|h_i(\bs{s})|+(1+|\bs{s}|^{1+\alpha})|h'_i(\bs{s})|\leq C,\quad
\text{ as }\quad
(-1)^{i+1}\bs{s}\to+\infty\label{PropConstrucionKernelEq3}
\end{align}
\end{prop}

\begin{proof}
We look for solutions $h(s)=a(s,0)^{-1/2}\cdot u(s)$ to
\eqref{JacobiOperator}, which means that $u$ solves the auxiliary
equation
$$
u''(\bs{s})-\tilde{q}(\bs{s})u(\bs{s})=0,\quad\hbox{in
}\R\label{StudyKernelJacobiOpEq20}
$$

where
\begin{align*}
\tilde{q}(\bs{s}):
&=\frac{\partial_{\bs{t}\bs{t}}a(\bs{s},0)}{a(\bs{s},0)}-2k^2(\bs{s})
+\frac{1}{2}\frac{\partial_{\bs{s}\bs{s}}a(\bs{s},0)}{a(\bs{s},0)}
-\frac{1}{4}\left|\frac{\partial_{\bs{s}}a(\bs{s},0)}{a(\bs{s},0)}\right|^2
\end{align*}
Now, thanks to the assumptions we have made on $a(\bs{s},\bs{t})$
and $\Gamma$, it follows that
$$
(1+|\bs{s}|)^{2+\alpha}|\tilde{q}(\bs{s})|\leq C.
$$

Therefore, applying lemma \ref{LemmaBehaviourKernel} on the region
$[0,\infty)$, we deduce the existence of two solutions linearly
independent of equation~\eqref{StudyKernelJacobiOpEq20} in $\R$,
denoted by $u(\bs{s})$ and $\tilde{u}(\bs{s})$, which satisfies the
right-sided asymptotic behavior as $s\to+\infty$
\begin{align*}
u(\bs{s})&=s+\mathit{O}(1)+\mathit{O}(|\bs{s}|^{1-\alpha}),\qquad
\tilde{u}(\bs{s})=1+
\mathit{O}(|\bs{s}|^{-1}+|\bs{s}|^{-\alpha})\\
u'(\bs{s})&=1+\mathit{O}(|\bs{s}|^{-1}+|\bs{s}|^{-\alpha}),\qquad
\tilde{u}'(\bs{s})=\mathit{O}(|\bs{s}|^{-1}+|\bs{s}|^{-1-\alpha})
\end{align*}
Applying Lemma \ref{LemmaBehaviourKernel} again, but this time on
the region $(-\infty,0]$, we obtain two other solutions $v(s)$ and
$\tilde{v}(s)$ linearly independent of
equation~\eqref{StudyKernelJacobiOpEq20} in $\R$, that now satisfy
the left-sided asymptotic behavior as $s\to-\infty$
\begin{equation*}
\begin{aligned}
v(\bs{s})&=|\bs{s}|+\mathit{O}(1)+\mathit{O}(|\bs{s}|^{1-\alpha}),\qquad
\tilde{v}(\bs{s})=1+\mathit{O}(|\bs{s}|^{-1}+|\bs{s}|^{-\alpha})\\
v'(\bs{s})&=1+\mathit{O}(|\bs{s}|^{-1}+|\bs{s}|^{-\alpha}),\qquad
\quad\tilde{v}'(\bs{s})=\mathit{O}(|\bs{s}|^{-1}+|\bs{s}|^{-1-\alpha})
\end{aligned}
\end{equation*}

We remark  that the non-degeneracy of curve $\Gamma$, implies that
$\tilde{u}(\bs{s})$ cannot be bounded on $(-\infty,0]$. Recall also
that $\{u,\tilde{u}\}$ and $\{v,\tilde{v}\}$ represent two different
basis of the vector space of solutions to the
equation~\eqref{StudyKernelJacobiOpEq20}. So that, for some constats
 $\alpha_i$, for $i=1,\ldots,4$, we have that
\begin{equation*}
\forall \bs{s}\in\R:\quad u(\bs{s})=\alpha_1
v(\bs{s})+\alpha_2\tilde{v}(\bs{s}), \quad
\tilde{u}(\bs{s})=\alpha_3 v(\bs{s})+\alpha_4\tilde{v}(\bs{s})
\end{equation*}

From the previous discussion about $\tilde{u}$, we observe that not
only that $\tilde{u}$ grows at most at a linear rate on
$(-\infty,0]$, but also that the non-degeneracy property implies
$\alpha_3\neq 0$. Hence, the function
$h_1(\bs{s}):=\alpha_3^{-1}a(\bs{s},0)^{-1/2}\tilde{u}(\bs{s})$
belongs to the kernel of $\J_{a,\Gamma}$, satisfying
\eqref{PropConstrucionKernelEq2}-\eqref{PropConstrucionKernelEq3}
for $i=1$.

\medskip
Likewise, the same argument can be applied to $\tilde{v}(\bs{s})$ to
find the function $h_2(\bs{s}):=a(\bs{s},0)^{-1/2}u(\bs{s})$
behaving as predicted and clearly, being linear independent with
$h_1$. This completes the proof of the proposition.
\end{proof}

Once we have described the kernel of \eqref{JacobiOperator}, it is
straightforward to check the following proposition, whose proof is
left to the readers.

\begin{prop}\label{InvertibilityJacobiOperator}
Under the same set of assumptions as in proposition
\ref{PropConstructionKernel} and given $\alpha>0$, $\lambda\in(0,1)$
and a function $f$ with
$\|f\|_{C^{0,\lambda}_{2+\alpha,*}(\R)}<\infty$, then the equation
$$
\J_{a}[h](\bs{s})=f(\bs{s}), \quad \bs{s}\in \R
$$

has a unique bounded solution, given by the variation of parameters
formula
\begin{align*}
h(\bs{s})=-h_1(\bs{s})\int^{\bs{s}}_{-\infty}
a(\xi,0)h_2(\xi)f(\xi)d\xi-h_2(\bs{s})\int^{+\infty}_{\bs{s}}
a(\xi,0)h_1(\xi)f(\xi)d\xi 
\end{align*}
In addition, there is some positive constant $C=C(a,\Gamma,\alpha)$
such that
\begin{align*}
\|h\|_{C^{2,\lambda}_{2 + \alpha,*}(\R)}\leq
C\|f\|_{C^{0,\lambda}_{2+\alpha,*}(\R)}
\end{align*}
where
$$
\|h\|_{C^{2,\lambda}_{2 + \alpha,*}(\R)}:=\|h\|_{L^{\infty}(\R)}+
\|(1+|s|)^{1+\alpha}h'\|_{L^{\infty}(\R)}+
\sup_{s\in\R}(1+|\bs{s}|)^{2+\alpha}\|h''\|_{C^{0,\lambda}(\bs{s}-1,
\bs{s}+1)}
$$
\end{prop}

%%%%%%%%%%%%%%%%%%%%%%%%%%%%%%%%%%%%%%%%%%%%%%%%%%%%%%%%%%%%%%%%%%%%%%%%
%%%%%%%%%%%%%%%%%%%%%%%%%%%%%%%%%%%%%%%%%%%%%%%%%%%%%%%%%%%%%%%%%%%%%%%%
\section{Examples:} \label{sec:examples}

It is of interest to mention that Theorem~\ref{TheoremMainResult}
relies on a very important fact, whose nature is essentially
geometrical. This concerns the existence of a curve
$\Gamma\subset\R^2$, given a fixed suitable potential $a(x)$ of the
Allen-Cahn equation~\eqref{IntrodEq12}, that be a stationary curve
with respect to the weighted length functional $l_{a,\Gamma}$ in the
sense of \eqref{IntrodEq13} and also be a non-degenerate curve as in
definition \eqref{IntrodEqExtra}. To get a better understanding of
the geometrical settings of this problem, it would be useful to
present some examples that portray the nature of the curves and of
the potentials we are thinking of and how they interact in order for
this configuration to be admissible.

The following section is devoted to provide concrete examples of
such curves associated to some nontrivial potential $a(x)$, in a way
that they meet all the hypothesis of
Theorem~\ref{TheoremMainResult}.

\medskip
\subsection{Characterization of Non-degeneracy.} Now we
state a result that provides precise conditions on $a$ and $\Gamma$,
where in such case the non-degeneracy property of the curve holds.

\begin{coro}{\label{CorollaryNondegeneracy}
\textsf{Non-degeneracy in the minimizing case}\\
Let $\alpha>-1/2$, $\Gamma$ be a stationary curve with respect to
$l_{a,\Gamma}$ as in \eqref{IntrodEq13}, and let the potential
$a(\bar{s},t)$ along with $\Gamma$ be such that
$Q(\bar{s}):=\partial_{tt}a(\bar{s},0)/a(\bar{s},0)-2k^2(\bar{s})$
satisfies the following conditions
\begin{align}
&Q(\bar{s})\geq 0,\quad \text{ and }\quad Q(\bar{s})\not\equiv 0\label{Corollary1Hyp1}\\
\intertext{and the asymptotic polinomial decay} &|Q(\bar{s})|\leq
\frac{C}{(1+|\bar{s}|)^{2+\alpha}},\quad\text{ for
}\quad|\bar{s}|>>\bar{s}_0\label{Corollary1Hyp2}
\end{align}
then the curve $\Gamma$ is non-degenerate, in the sense of
\eqref{IntrodEqExtra}. }
\end{coro}
\begin{proof}
Let $h$ be a bounded element in the kernel of the Jacobi operator,
so that $h\in L^{\infty}(\R)$ solves $\J_{a}[h](\bar{s})=0$ in $\R$.
Lemma \ref{LemmaDecayDerivativeJacobi} assures the existence of a
constant $C>0$ such that
\begin{align}
|h'(\bar{s})|\leq \frac{C}{1+|\bar{s}|^{1+\alpha}},\quad \forall
\bar{s}\in\R\label{StudyKernelJacobiOpEq25}
\end{align}
which implies $h'\in L^2(\R)$, as $\alpha>-1/2$. Furthermore, condition \eqref{Corollary1Hyp2} plus $\alpha>-1/2$
imply the integrability of $Q(s)$ in $\R$ since
$$ \int_{\R}|Q(\bar{s})|d\bar{s}=\int_{|\bar{s}|>\bar{s}_0}Q(\bar{s})d\bar{s}
+\int_{|\bar{s}|\leq \bar{s}_0}Q(\bar{s})d\bar{s}\leq
C\int_{|\bar{s}|>\bar{s}_0}
\frac{d\bar{s}}{1+|\bar{s}|^{2+\alpha}}+\int_{|\bar{s}|\leq
\bar{s}_0}Q(\bar{s})d\bar{s}$$ In particular, the foregoing
guarantees that the following expression is well defined
\begin{align*}
\|h\|_{Q}&:=\langle \J_{a}[h], h\rangle_{L^2(a(s,0))}= \int_{\R}[h''(\bar{s})+\frac{\partial_{\bs{s}}a(\bar{s},0)}{a(\bar{s},0)}h'(\bar{s})-Q(\bar{s})h(\bar{s})]h(\bar{s})a(\bar{s},0)d\bar{s}\\
&=
a(\bar{s},0)h'(\bar{s})h(\bar{s})\biggr\rvert^{\bar{s}=+\infty}_{\bar{s}
=-\infty}-\int_{\R}a(\bar{s},0)[|h'(\bar{s})|^2+Q(\bar{s})h^2(\bar{s})]d\bar{s}
\end{align*}
where the boundary value terms at infinity vanishes, as $h$ is
bounded and from the decay \eqref{StudyKernelJacobiOpEq25}. Finally
as $h$ solves $\J_{a}[h]=0$, the latter together with hypothesis
\eqref{Corollary1Hyp1} imply
$$ 0=\int_{\R}a(\bar{s},0)[|h'(\bar{s})|^2+Q(\bar{s})h^2(\bar{s})]d\bar{s}\geq \int_{\R}a(\bar{s},0)|h'(\bar{s})|^2d\bar{s}$$
As $a(\bar{s},0)>0$, we deduce that $h'=0$ a.e. in $\R$. Moreover,
the smoothness of $h$ guarantees that $h'\equiv 0$ in $\R$. Using
again the last inequality we get that
$$ 0=\int_{\R}Q(\bar{s})h^2(\bar{s})d\bar{s}$$
However, condition \eqref{Corollary1Hyp1} on $Q$ assures that
$Q(\bar{s})>0$ on a neighborhood of some point $\bar{s}_0\in\R$.
Therefore, last equality gives that $h(\bar{s})=0$ in this
neighborhood, but as $h\equiv C$ is a constant function in the
entire space, we conclude $h=0$, which concludes the proof of
Corollary~\ref{CorollaryNondegeneracy}.
\end{proof}

\medskip
\subsection{Geodesics for $l_{a,\Gamma}$ in Euclidean coordinates.}
In what follows, we will consider a curve $\Gamma$ that can be
represented as the graph of some function. Let us consider a smooth
function $f:\R\to\R$, $f=f(\bs{x})$, and a parametrized curve
$\Gamma:=\{\gamma(\bs{x})/\ \bs{x}\in\R\}\subset\R^2$ such that
\begin{align}
\gamma(\bs{x})=(\bs{x},f(\bs{x})),\quad
\dot{\gamma}(\bs{x})=(1,f'(\bs{x}))\label{ExampleEq0}
\end{align}
In addition, let the normal field $\nu$ of $\Gamma$ be oriented
negatively, meaning that vector $\dot{\gamma}(\bs{x}) \wedge \nu(\bs{x})$ points in the opposite direction than $e_3$, the
generator of the $z-$axis in $\R^3$. This forces
$$ \nu(\bs{x})=\dfrac{1}{\sqrt{1+|f'(\bs{x})|^2}}(f'(\bs{x}),-1)$$
Let us also consider a potential defined in Euclidean coordinates
$a=a(\bs{x},\bs{y})$.
Recall from the criticality condition \eqref{IntrodEq13}, that in
order for $\Gamma$ to be a stationary curve with respect to the
weighted arc-length $l_{a,\Gamma}$, is necessary that the potential
$a$ and the curvature $k$ satisfy the equation
\begin{align}
\partial_{\bs{t}}a(\bs{s},0)=k(\bs{s})\cdot a(\bs{s},0)\ , \;\;  \text{ a.e. } \bs{s}\in \R \label{ExampleEq1}
\end{align}
Denoting $X(\bs{x},\bs{t}):=\gamma(\bs{x})+\bs{t}\nu(\bs{x})$, we
can now set the potential written in this coordinates as
\begin{align}
\tilde{a}(\bs{x},\bs{t}):=a\circ
X(\bs{x},\bs{t})=a\left(\bs{x}+\dfrac{\bs{t}f'(\bs{x})}{\sqrt{1+|f'(\bs{x})|^2}},\
f(\bs{x})-\dfrac{\bs{t}}{\sqrt{1+|f'(\bs{x})|^2}}\right)\label{ExampleEq2}
\end{align}
Accordingly, relation \eqref{ExampleEq2} implies that the
criticality condition \eqref{ExampleEq1} amounts to the following
equation in Euclidean coordinates
\begin{align}
\frac{\partial_{\bs{x}}a(\bs{x},f(\bs{x}))
f'(\bs{x})}{\sqrt{1+|f'(\bs{x})|^2}} - \frac{\partial_y
a(\bs{x},f(\bs{x}))}{\sqrt{1+|f'(\bs{x})|^2}}
=\frac{f''(\bs{x})}{(1+|f'(\bs{x})|^2)^{3/2}}\cdot
a(\bs{x},f(\bs{x}))\label{CriticalityEuclideanCoordinates}
\end{align}
where it has been used the classical formula for the curvature of
$\Gamma$ as given in \eqref{ExampleEq0},
$$k(\bs{x})=f''(\bs{x})(1+|f'(\bs{x})|^2)^{-3/2}$$

\medskip
%%---------------------------------------------------------------------------------------------------------------------------------------
{\bf Example 1: The $x$-axis. }\label{subsec:FirstExample} We
consider $\Gamma\subset\R^2$ to be the $x$-axis. Nonetheless, the
stationarity of this line must be with respect to some nontrivial
potential $a(\bs{x},\bs{y})\not\equiv 1$ that does not represent the
classic Euclidean metric in $\R^2$, case in which all straight lines
are trivially known as stationary curves.

With this purpose, let us set the function $f(\bs{x})\equiv 0$ in
\eqref{ExampleEq0}, implying that $\Gamma=\overrightarrow{0X}$. In
particular, we have that $\nu(\bs{x})\equiv e_2$, thus the Fermi
coordinates are reduced simply to the Euclidean coordinates, namely
$X(\bs{x},\bs{t})=\bs{x}e_1+\bs{t}e_2=(\bs{x},\bs{t})$.

In this simplified context, it turns out that the criticality
condition \eqref{CriticalityEuclideanCoordinates} is reduced to
\begin{align}
\partial_{\bs{y}} a(\bs{x},0)=0,\quad \forall
\bs{x}\in\R.\label{ExampleEq4}
\end{align}

Therefore, we only need to find a nontrivial potential
$\tilde{a}(\bs{x},\bs{t})=a(\bs{x},\bs{y})$ in such way the $x$-axis
becomes a stationary curve, and also a nondegenerate curve as in the
sense of~\ref{IntrodEqExtra}.

%%%------ BEGIN CLAIM  ------------------------------------------------------------------------------------------------------------------
\begin{claim}{\label{ClaimPotential1}
Given any $\alpha>0$, the following potential
\begin{align}
a(\bs{x},\bs{y}):=1\,+\,\frac{1}{(1+|\bs{x}|)^{2+\alpha}}\cdot\left(
\frac{\bs{y}^2}{\cosh(\bs{y})}\right)\label{ExampleA1}
\end{align}
satisfies all the requirements previously indicated, in relation
with the curve $\Gamma=\overrightarrow{0X}$. See Figure
\ref{fig:PotentialExample1}.
}\end{claim}
%%%------ END CLAIM  ------------------------------------------------------------------------------------------------------------------
\begin{proof}
Let us note that $a(\bs{x},\bs{y})$ is smooth, globally bounded, and
bounded below far away from zero. Further, it is direct that
$\overrightarrow{0X}$ is a stationary curve relative to
$l_{a,\Gamma}$ since solves equation~\eqref{ExampleEq4}
\begin{align*}
\partial_{\bs{y}}a(\bs{x},\bs{y})= \frac{1}{(1+|\bs{x}|)^{2+\alpha}}\left(\frac{2\bs{y}-\bs{y}^2\sinh(\bs{y})}{\cosh^2(\bs{y})}\right)\quad\Rightarrow\quad
\partial_{\bs{y}}a(\bs{x},0)= 0,\quad \forall \bs{x}\in\R
\end{align*}
Now to see that $\overrightarrow{0X}$ is a nondegenerate curve, just
note that the potential achieves it minimum exactly on the region
defined by the $\bs{x}$-axis, and moreover, around this curve the
potential is strictly convex in the $\bs{y}$-direction. The latter
translates in the fact that $\partial_{\bs{y}\bs{y}}a(\bs{x},0)>0$,
since
\begin{align*}
\partial_{\bs{y}\bs{y}}a(\bs{x},\bs{y})= \frac{1}{(1+|\bs{x}|)^{2+\alpha}}\left(\frac{2-2\bs{y}\sinh(\bs{y})-\bs{y}^2\cosh(\bs{y})}{\cosh^2(\bs{y})}
-\frac{2(2\bs{y}-\bs{y}^2\sinh(\bs{y}))
\sinh(\bs{y})}{\cosh^3(\bs{y})}\right)\\[-0.7cm]
\end{align*}
and
\begin{align*}
\partial_{\bs{y}\bs{y}}a(\bs{x},0)=\frac{2}{(1+|\bs{x}|)^{2+\alpha}}>0,\quad
\forall \bs{x}\in\R
\end{align*}
Taking this into account, note that $a(\bs{x},\bs{y})$ and
$k(\bs{x})\equiv 0$ are such that term
$$Q(\bs{x}):=\frac{\partial_{\bs{y}\bs{y}}a(\bs{x},0)}{a(\bs{x},0)}-2k^2(\bs{x})$$
fulfills the following conditions
$$Q(\bs{x})>0,\quad \text{ and }\quad |Q(\bs{x})|\leq \frac{2}{(1+|\bs{x}|)^{2+\alpha}},\quad \forall \bs{x}\in\R$$

Hence we deduce from Corollary~\ref{CorollaryNondegeneracy} that
$\Gamma=\overrightarrow{0X}$ is a nondegenerate curve with respect
to the potential $a(\bs{x},\bs{y})$ given in \eqref{ExampleA1},
finishing the proof of the claim.
\end{proof}

%%-------------  BEGIN FIGURA  ----------------------------------------------------------------------------------------------------
    \begin{minipage}[b]{\linewidth}
      \centering
      \includegraphics[width=0.52\linewidth]{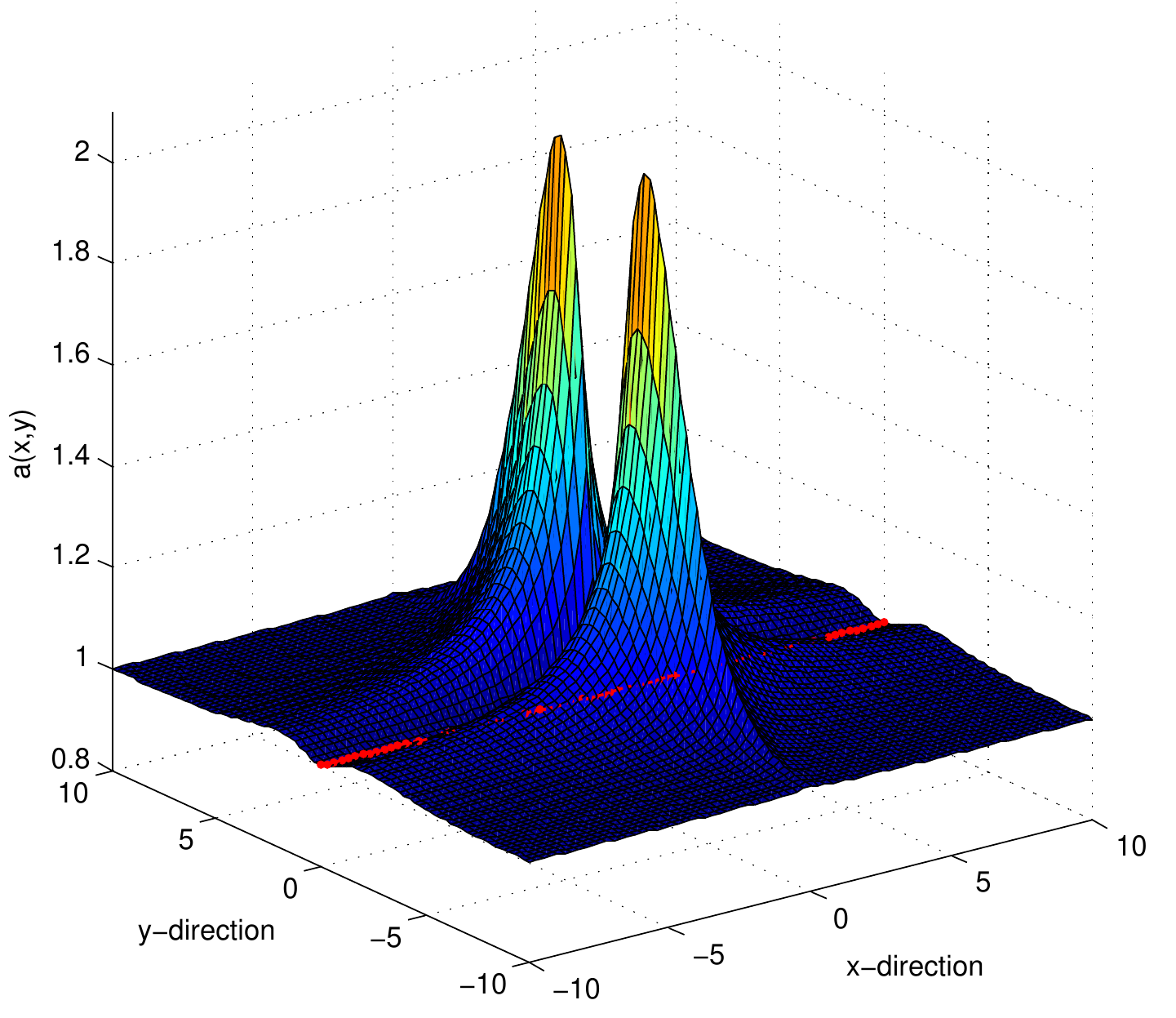}\\[-0.27cm]
      \captionof{figure}[Potential $a(x,y)$ with the $x-$axis as non-degenerate geodesic]
      {Potential $a(\bs{x},\bs{y})$ \eqref{ExampleA1} with geodesic $\Gamma=\vec{0X}$, for $\alpha=0.01$.\label{fig:PotentialExample1}}     \end{minipage}
%%-------------  END FIGURA  ----------------------------------------------------------------------------------------------------

\medskip
{\bf Example 2: Asymptotic straight line.
}\label{subsec:SecondExample} This time we consider the function
$f(\bs{x}):=\sqrt{1+\omega^2 x^2}$, so that $\Gamma$ converges
asymptotically to straight lines as $|\bs{x}|\to\infty$. We have to
exhibit some nontrivial potential $a(\bs{x},\bs{y})$ for which
$\Gamma$ is a nondegenerate geodesic relative to the arclength
$\int_{\Gamma}a(\vec{x})$.

We will assume a weaker dependence of the potential in Euclidean
variables, namely $a=a(\bs{y})$, so that criticality
condition~\eqref{CriticalityEuclideanCoordinates} amounts to
\begin{align}
&\frac{a'(f(\bs{x}))}{a(f(\bs{x}))}=
\frac{-f''(\bs{x})}{1+|f'(\bs{x})|^2}=g(f(x))\label{ExampleEq5}
\end{align}
with $g(\bs{y}):=-\omega^2[(1+\omega^2)\bs{y}^3-\omega^2\bs{y}]^{-1}$.\\
We can solve directly this ordinary differential
equation~\eqref{ExampleEq5}, for $a$ in $\bs{y}-$variable.
$$ \log(a(\bs{y}))=\int g(\bs{y})d\bs{y}+M \quad \ssi\quad
a(\bs{y})=M\exp\left(\int \frac{-\omega^2
d\bs{y}}{(1+\omega^2)\bs{y}^3-\omega^2\bs{y}}\right)$$ This integral
can be computed using partial fraction decomposition, that leads to
\begin{align*}
a(\bs{y})=\frac{M\sqrt{1+\omega^2}\bs{y}}{\sqrt{(1+\omega^2)\bs{y}^2-\omega^2}}
\end{align*}
For this construction, we will need to consider a slight
modification of function $a$ as follows. We say that the potential
$\hat{a}:\R^2\to\R$ is an \emph{admissible left-extension} of
function $a(x,y)$, provided that
\begin{itemize}
\item $\hat{a}$ be smooth bounded function, of at least $C^2(\R^2)$ class.
\item $\hat{a}(x,y)=a(x,y)$ for points with $y\geq \omega^2/(1+\omega^2)$.
\item $\hat{a}$ is uniformly positive, bounded below away from zero.
\end{itemize}
We state the following
%%%------ BEGIN CLAIM  ------------------------------------------------------------------------------------------------------------------
\begin{claim}{\label{ClaimPotential2}
Given $|\omega|\leq 1/\sqrt{2}$, any admissible left-extension of
the potential given below
\begin{align}
a(\bs{x},\bs{y}):=\frac{\sqrt{1+\omega^2}\bs{y}}{\sqrt{(1+\omega^2)\bs{y}^2-\omega^2}}\label{ExampleA2}
\end{align}
induces a metric in $\R^2$ for which
$\Gamma=\left\{(\bs{x},\sqrt{1+\omega^2
\bs{x}^2})\right\}_{\bs{x}\in\R}$ is a nondegenerate geodesic. See
Figure~\ref{fig:PotentialExample2}. }\end{claim}
%%%------ END CLAIM  ------------------------------------------------------------------------------------------------------------------
\begin{proof}
Regardless the value of the parameter $\omega\neq 0$, it can be
readily checked that within the region $y\geq
2\omega/\sqrt{1+\omega^2}$, function~\eqref{ExampleA2} is smooth,
bounded, and uniformly positive. Moreover, this potential satisfies
the asymptotic stability on the curve $\Gamma$, since
$f(\bs{x})\to+\infty$ as $|\bs{x}|\to+\infty$ and additionally
$\lim\limits_{\bs{y}\to+\infty}a(\bs{x},\bs{y})=1,\;
\forall\bs{x}\in\R$. The previous construction of $a(\bs{x},\bs{y})$
was intended to build a potential satisfying the criticality
condition~\eqref{ExampleEq5} for the curve generated by
$f(\bs{x})=\sqrt{1+\omega^2 x^2}$. Thus $\Gamma$ is a geodesic for
the arclength $\int_{\Gamma}a(\vec{x})$. All these features of $a$
ensure that any \emph{admissible left-extension} will provide a
potential with the desired properties to induce a smooth metric in
$\R^2$, fulfilling hypothesis~\eqref{IntrodEq13} of
Theorem~\ref{TheoremMainResult}. Moreover, a tedious but simple
calculation shows that
$$a'(\bs{y})=\frac{-\omega^2\sqrt{1+\omega^2}}{(\bs{y}^2+\omega^2(\bs{y}^2-1))^{3/2}},\quad a''(\bs{y})=\frac{3\omega^2(1+\omega^2)^{3/2}\bs{y}}{[(1+\omega^2)\bs{y}^2-\omega^2]^{5/2}}$$
Therefore, taking into account the decay of the derivatives of
$f(\bs{x})$, it follows that this potential satisfies condition
\eqref{IntrodEq14} of Theorem~\ref{TheoremMainResult}, for
$\alpha=2>0$. It only remains to prove the nondegeneracy property of
the curve $\Gamma$, and to do this, we will make use of
Corollary~\ref{CorollaryNondegeneracy}. It can be checked the
positiveness of the term $Q(\bs{x})$, in fact
\begin{align*}
2k^2(\bs{x})&=\frac{2\omega^2}{(1+(\omega^2+\omega^4)\bs{x}^2)^3},\quad
\partial_{\bs{t}\bs{t}}\tilde{a}(\bs{x},0)=a''(f(\bs{x}))\frac{1+\omega^2\bs{x}^2}{1+(\omega^2+\omega^4)\bs{x}^2}
\end{align*}
so by the definition
$Q(\bs{x})=\partial_{tt}\tilde{a}(\bs{x},0)/\tilde{a}(\bs{x},0)-2k^2(\bs{x})$
we obtain
\begin{align*}
Q(\bs{x})%&\geq \min\{1,\|a\|^{-1}_{\infty}\}
%&\left(\frac{3\omega^2(1+\omega^2)^{3/2}\sqrt{1+\omega^2\bs{x}^2}}{[(1+\omega^2)(1+\omega^2\bs{x}^2)-\omega^2]^{5/2}}\cdot \frac{1+\omega^2\bs{x}^2}{1+(\omega^2+\omega^4)\bs{x}^2}-\frac{2\omega^2}{(1+(\omega^2+\omega^4)\bs{x}^2)^3}\right)\\[0.2cm]
&\geq
C_a\left(\frac{3\omega^2(1+\omega^2)^{3/2}(1+\omega^2\bs{x}^2)^{1/2}}{(1+\omega^2)^{5/2}(1+\omega^2\bs{x}^2)^{5/2}}
-\frac{2\omega^2}{(1+(\omega^2+\omega^4)\bs{x}^2)^3}\right)\\[0.2cm]
&\geq
C_a\left(\frac{3\omega^2}{(1+\omega^2)(1+\omega^2\bs{x}^2)^{2}}
-\frac{2\omega^2}{(1+\omega^2\bs{x}^2)^3}\right)\\
&=\frac{C_a\omega^2}{(1+\omega^2\bs{x}^2)^{2}}\left(\frac{3}{1+\omega^2}
-\frac{2}{1+\omega^2\bs{x}^2}\right).
\end{align*}

Hence, choosing $\omega\in\R\setminus\{0\}$ with $|\omega|\leq
1/\sqrt{2}$, we get that $Q(\bs{x})>0$ in the entire domain $\R$.
Finally, the term $Q(\bs{x})$ decays polynomially at a rate
$\mathit{O}((1+|\bs{x}|)^{-4})$ as a consequence of the decay of the
potential and the squared curvature, which finishes the proof of
Claim~\ref{ClaimPotential2}.
\end{proof}

\begin{rmk}{
We emphasize the fact that the criticality condition for $\Gamma$
and the nondegeneracy property are tested only within the semi-space
$\bs{y}\geq 1$, which involve only the part~\eqref{ExampleA2} of the
admissible left-extension, since $\hat{a}(\bs{x},\bs{y})=a(\bs{y})$
in this region and the curve complies $|f(\bs{x})|\geq 1$. }
\end{rmk}

%%-------------  BEGIN FIGURA  ----------------------------------------------------------------------------------------------------
    \begin{minipage}[b]{\linewidth}
      \centering
      \includegraphics[width=0.6\linewidth]{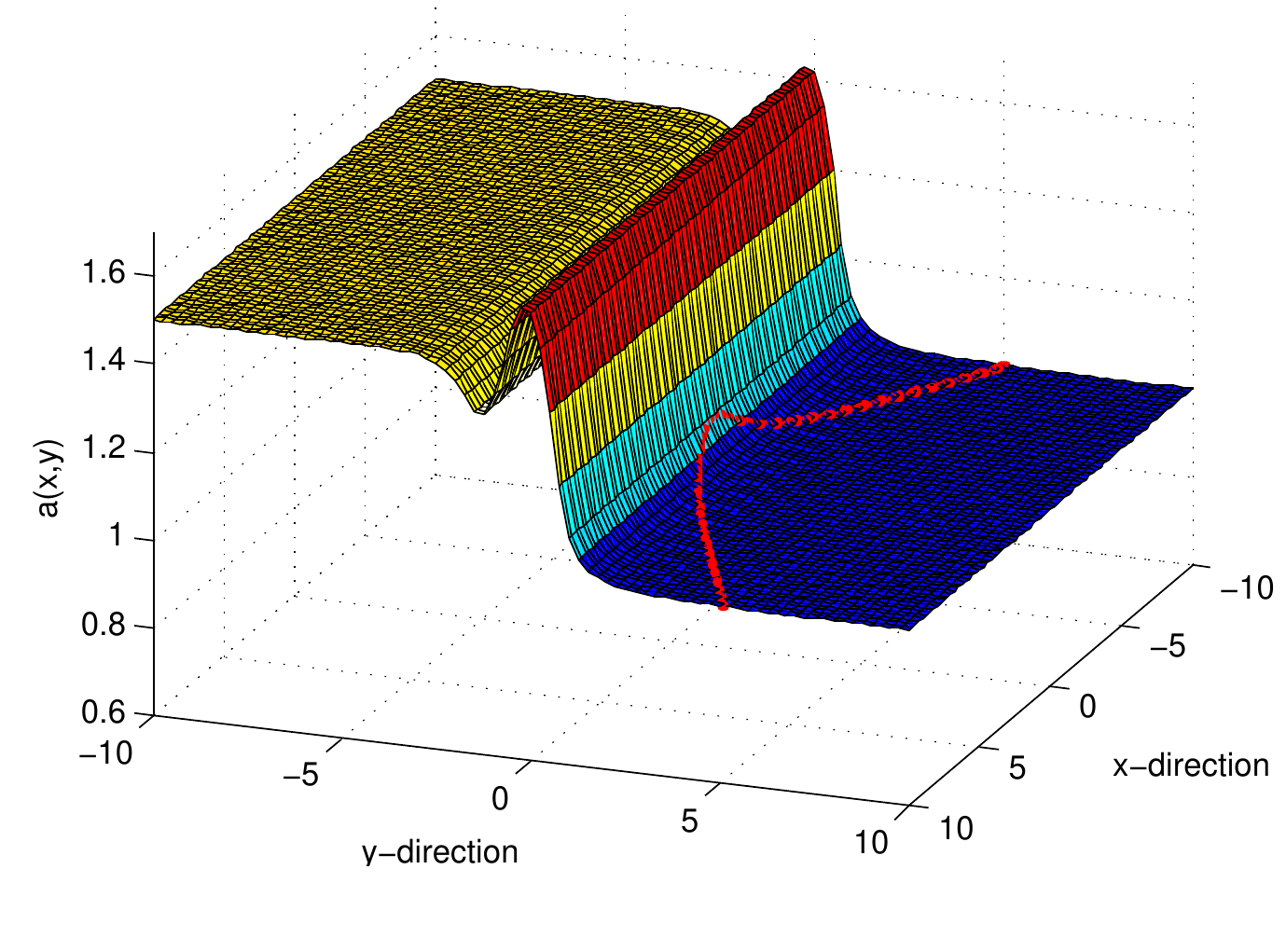}\\[-0.25cm]
      \captionof{figure}[Potential $\hat{a}(x,y)$ with asymptotic straight line as non-degenerate geodesic]
      {Potential $\hat{a}(\bs{x},\bs{y})$ \eqref{ExampleA2} with $\Gamma_{\omega}$ as nondegenerate geodesic, for $\omega=1/2$.\label{fig:PotentialExample2}}
    \end{minipage}
%%-------------  END FIGURA  ----------------------------------------------------------------------------------------------------

%%%%%%%%%%%%%%%%%%%%%%%%%%%%%%%%%%%%%%%%%%%%%%%%%%%%%%%%%%%%%%%%%%%%%%%%
%%%%%%%%%%%%%%%%%%%%%%%%%%%%%%%%%%%%%%%%%%%%%%%%%%%%%%%%%%%%%%%%%%%%%%%%

\section{Approximation of the solution and preliminary
discussion }\label{section5}
Here and after we assume
$$
-F'(s):=s(1-s^2), \quad s\in \R
$$
but we remark that the developments we make follows,  for a general
double well potential $F$ satisfying
\eqref{IntrodEq2}-\eqref{IntrodEq3}-\eqref{IntrodEq3.2}, with no
significant changes.

\medskip

\subsection{The approximation local Approximation.} To begin with,
let us consider the parameter function $h\in C^2(\R)$, for which we
assume further $h=h(\bs{s})$ satisfies the apriori estimate
\begin{multline}
\|h\|_{C^{2,\lambda}_{2+\alpha,*}(\R)}:= \|h\|_{L^{\infty}(\R)}+
\|(1+|\,\bs{s}\,|)^{1+\alpha}h'\|_{L^{\infty}(\R)}+\\
\sup_{\bs{s}\in\R}(1+|\,\bs{s}\,|)^{2+\alpha}
\|h''\|_{C^{0,\lambda}(\bs{s}-1,\bs{s}+1)}\leq \K \ep
\label{BallConditionH1}
\end{multline}
for a certain constant $\K>0$ that will be chosen later, but
independent of $\ep>0$.

Let us consider $w(t)$, the solution to the ODE
$$
w''(t) -F'(w(t))=0, \quad w'(t)>0, \quad w(\pm\infty)=\pm1.
$$

As mentioned in the previous section, by an obvious rescaling,
equation \eqref{IntrodEq12} becomes
\begin{equation*}
S(v):= \Delta_{x} v+\ep\dfrac{\nabla_{\bar{x}} a}{a}\nabla_x v
+v(1-v^2)=0, \quad \hbox{in }\R^2
\end{equation*}

Let us choose in the region $\Nn_{\ep,h}$, as first approximate for
a solution to \eqref{IntrodEq12}, the function
\begin{equation*}
v_0(x):=w(z-h(\ep s))=w(t), \quad x=X_{\ep,h}(s,t)\in \Nn_{\ep,h}
\end{equation*}
where $z=t+h(\ep s)$ designates the normal coordinate to
$\Gamma_{\ep}=\ep^{-1}\Gamma$.

\medskip
Using expression \eqref{EcuacionAllen-CahnEnVecindadDeFermiEq4} we
compute the error $S(v_0)$ in the region $\Nn_{\ep,h}$  to find that

$$
S(v_0)\,=\,-\ep^2\J_{a}[h](\ep s)w'(t)-\ep^2\left[2k^2(\ep
s)\,-\,\dfrac{\partial_{\bs{t}\bs{t}}a}{a}(\ep s,0)\right]tw'(t)
\,+\,\ep^2|h'(\ep s)|^2w''(t)
$$
\begin{align}
&+\ep(t+h(\ep s))A_0(\ep s,\ep(t+h))[-\ep^2h''(\ep s)w'(t)+\ep^2|h'(\ep s)|^2w''(t)]\notag\displaybreak[3]\\[0.2cm]
&+\ep^2(t+h(\ep s))\tilde{B}_0(\ep s,\ep(t+h))\left(-\ep h'(\ep s)w'(t)\right)\notag\displaybreak[3]\\[0.2cm]
&+\ep^3(t+h(\ep s))^2\tilde{C}_0(\ep
s,\ep(t+h))w'(t)\label{ApproximationOfASolutionEq1}
\end{align}
with $A_0,\tilde{B}_0,\tilde{C}_0$ are given in
\eqref{ErrorA0LaplacianXeph}-\eqref{EcuacionAllen-CahnEnVecindadDeFermiEq5}
-\eqref{EcuacionAllen-CahnEnVecindadDeFermiEq6}, respectively. We
emphasize that we have broken formula
\eqref{ApproximationOfASolutionEq1} into powers of $\ep$, keeping in
mind that $h=\mathit{O}(\ep)$.

\medskip
For every fixed $s\in\R$, let us consider the $L^2$-projection,
given by
$$ \Pi(s):=\int^{+\infty}_{-\infty}S(v_0)(s,t)w'(t)dt$$
where for simplicity we are assuming that coordinates are defined
for all $t$, since the difference with the integration taken in all
the actual domain for $t$, produces only exponentially small terms.

From \eqref{ApproximationOfASolutionEq1}, we observe that
$$
\Pi(\ep s)=-\ep^2\J_{a}[h](\ep s)\int_{\R}w'(t)^2dt
$$
$$
\,-\,\ep^3\int_{\R}(t+h)A_0(\ep s,\ep(t+h))\left[h''(\ep
s)\,|w'(t)|^2 - |h'(\ep s)|^2\,w''(t) w'(t)\right]dt
$$
\begin{equation}\label{ApproximationOfASolutionEq3}
-\ep^3\int_{\R}\left[(t+h)\tilde{B}_0(\ep s,\ep(t+h))
 h'(\ep s)\,-\,(t+h)^2\tilde{C}_0(\ep
s,\ep(t+h))\right]\,|w'(t)|^2dt
\end{equation}
where we used that $\int^{+\infty}_{-\infty}tw'(t)^2dt=0,\int^{+\infty}_{-\infty}w''(t)w'(t)dt=0$, to get rid of the terms of
order $\ep^2$.

\medskip
Making this projection equal to zero is equivalent to the nonlinear
differential equation for $h$
\begin{align}
\J_{a,\Gamma}[h]= h''(\bs{s})+\dfrac{\partial_s
a(\bs{s},0)}{a(\bs{s},0)}h'(\bs{s})-Q(\bs{s})h(\bs{s})=G_0[h]\ ,
\;\forall \bs{s}\in\R\label{ApproximationOfASolutionEq4}
\end{align}
where we have set
$$
Q(\bs{s})=\left[\frac{\partial_{tt}a(\bs{s},0)}{a(\bs{s},0)}-2k^2(\bs{s})\right]
$$
and $G_0$ consists in the remaining terms of
\eqref{ApproximationOfASolutionEq3}.

\medskip
$G_0$ is easily checked to be a Lipschitz in $h$, with small
Lipschitz constant. Here is where the \emph{nondegeneracy condition}
on the curve $\Gamma$ makes its appearance, as we need to invert the
operator $\J_{a,\Gamma}$, in such way that
equation~\eqref{ApproximationOfASolutionEq4} can be set as a fixed
problem for a contraction mapping of a ball of the form
\eqref{BallConditionH1}.

\medskip
It will be necessary to pay attention to the size of $S(v_0)$ up to
$\mathit{O}(\ep^2)$, because the solvability of the nonlinear Jacobi
equation~\eqref{ApproximationOfASolutionEq4} depends strongly on the
fact that the error created by our choice of the approximation, is
sufficiently small in  $\ep>0$.

We improve our choice of the approximation throughout the following
argument. Let us consider the ODE
$$ \psi''_0(t)-F''(w(t))\psi_0(t)=tw'(t)$$
which has a unique bounded solution with $\psi_0(0)=0$, given
explicitly by the variation of parameters formula
$$
\psi_0(t)=w'(t)\int^t_{0}w'(\tau)^{-2}\int^{\tau}_{-\infty}sw'(s)^2dsd\tau.
$$

Since $\int_{\R}sw'(s)^2ds=0$, the function $\psi_0(s)$ satisfies
that
$$
\|e^{\sigma|t|}\partial^j\psi_0\|_{L^{\infty}(\R)}\leq C_{j,
\sigma}, \quad j=1,2,\ldots, \quad 0\leq\sigma<\sqrt{2}.
$$

Due to the finer topology we are considering for $h$, we must also
improve the term
$$
-\ep^3\left[k^3(\ep
s)+\frac{1}{2}\partial_{tt}\left(\frac{\partial_{t}a}{a}\right) (\ep
s,0)\right]t^2w'(t)
$$
present in $\ep^3(t+h(\ep s))^2\tilde{C}_0(\ep s,\ep(t+h))w'(t)$ in
expression \eqref{ApproximationOfASolutionEq1}. Hence, let us
consider $g(t):=t^2w'(t)$ and note that we can write
$$
g=C_{g}w'+g_{\perp}
$$  where $g_{\perp}$ denotes the orthogonal
projection of $g$ onto $w'$ in $L^2(\R)$, given by
$$
g_{\perp}(t):=t^2w'(t)-\frac{\int_{\R}\tau^2
w'(\tau)^2d\tau}{\int_{\R}w'(\tau)^2d\tau}w'(t).
$$
Thus by setting
$$
\psi_1(t)=w'(t)\int^t_{0}
w'(\tau)^{-2}\int^\tau_{-\infty}g_{\perp}(s)w'(s)dsd\tau
$$
this formula provides a bounded smooth solution to
$\psi''_1(t)F''(w(t))\psi_1(t)=g_{\perp}(t)$ with
$$
\|e^{\sigma|t|}\partial^j\psi_1\|_{L^{\infty}(\R)}\leq C_{j,
\sigma}, \quad j=1,2,\ldots, \quad 0\leq\sigma<\sqrt{2}.
$$

Hence, we choose as new approximation, the function
\begin{equation}
v_1(s,t):=v_0(s,t)+\varphi_1(s,t)=w(t)
+\varphi_1(s,t)\label{ImproveAproximation}
\end{equation}
where
\begin{eqnarray*}
\varphi_1(s,t)&:=&\ep^2\left[2k^2(\ep
s)-\dfrac{\partial_{\bs{t}\bs{t}} a(\ep s,0)}{a(\ep s,0)}\right]
\psi_0(t) \nonumber\\ \nonumber\\
&&-\ep^3\left[k^3(\ep
s)+\frac{1}{2}\partial_{tt}\left(\frac{\partial_{t}a}{a}\right) (\ep
s,0)\right]\psi_1(t)
\end{eqnarray*}
and which can be easily seen to behave like
$\varphi_1(s,t)=\mathit{O}(\ep^2(1+|\ep s|)^{-2-\alpha}e^{-\sigma|t|})$, for
sigma $0<\sigma<\sqrt{2}$. This is due to the assumptions
\eqref{IntrodEq15}-\eqref{IntrodEq14} we have made on the curve
$\Gamma$ and the potential $a$, and to the previous observation on
$\psi_0(t),\ \psi_1(t)$.

\medskip
Now, to analyze the error term $S(v_1)$, notice that
\begin{equation*}
S(v_0+\varphi_1)=S(v_0)\,+\,\Delta_{x}\varphi_1+\ep\dfrac{\nabla_{\bar{x}}
a}{a}\nabla_x \varphi_1-F''(v_0)\varphi_1+N_0(\varphi_1)
\end{equation*}
where
\begin{align}
N_0(\varphi_1)=-F'(v_0+\varphi_1)+F'(v_0)+F''(v_0)\varphi_1\label{QuadraticN0}
\end{align}

From the definition of $\varphi_1$, we find that
\begin{eqnarray}
S(v_1)&=&S(v_0)\,+\,\ep^2\,\left[2k^2(\ep s)-\,\dfrac{\partial_{tt}
a(\ep s,0)}{a(\ep s,0)}\right]tw'(t)
\nonumber\\
&&-\,\ep^3\left[k^3(\ep s)+\frac{1}{2}\partial_{tt}\left(
\frac{\partial_{t}a}{a}\right)(\ep s,0)\right]g_{\perp}(t)\nonumber\\
&&+\,\left[\Delta_{x} + \ep \dfrac{\nabla_{\bar{x}} a}{a}\nabla_x
-\partial_{tt}\right]\varphi_1\,+\,N_0(\varphi_1).\label{ApproximationOfASolutionEq6}
\end{eqnarray}

Analyzing the new error created by $\varphi_1$, we readily check
using the expansions for the differential operators
\eqref{LaplacianXeph}-\eqref{GradientsXeph} and the definition
\eqref{QuadraticN0}, that

$$
\left[\Delta_{x}+\ep\dfrac{\nabla_{\bar{x}}
a}{a}\nabla_x-\partial_{tt}\right] \varphi_1
\,+\,N_0(\varphi_1)\,=\,
$$

$$
-\ep^4\,Q''(\ep s)\psi_0+\ep^4\bigl[\J_{a}[h](\ep s)-tQ(\ep s)
\bigr]Q(\ep s)\psi'_0
$$

$$
-\ep^4\left(\dfrac{\partial_{\bs{s}}a(\ep s,0)}{a(\ep s,0)}Q'(\ep
s)\psi_0+2h'_1(-Q'(\ep s)\psi'_0)+\biggl.|h'_1|^2Q(\ep s)\psi''_0
\right)
$$

\begin{align}
+\mathit{O}(\ep^4(1+|\ep
s|)^{-4-\alpha}e^{-\sigma|t|})\label{ApproximationOfASolutionEq8}
\end{align}
where we recall the convention
\begin{equation*}
Q(\bs{s})=\frac{\partial_{\bs{t}\bs{t}}a(\bs{s},0)}{a(\bs{s},0)}
-2k^2(\bs{s}), \quad \bs{s}\in \R
\end{equation*}
and have used that the error terms in the differential operator
evaluated in $\varphi_1$, associated to $A_0(\ep s,\ep(t+h)),
\tilde{B}_0(\ep s,\ep(t+h))$, $\tilde{C}_0(\ep s,\ep(t+h))$ behave
like $\mathit{O}(\ep^5(1+|\ep s|)^{-4-2\alpha}e^{-\sqrt{2}|t|})$, given that
$h$ has a bounded size is $\ep s$ by \eqref{BallConditionH1}, and
since $\varphi_1(s,t)$ has smooth dependence in $\ep s$ with size
$\mathit{O}(\ep^2(1+|\ep s|)^{-2-\alpha}e^{-\sigma|t|})$.

\medskip
Therefore, the error \eqref{ApproximationOfASolutionEq8} is can be
written as
\begin{equation*}
\left[\Delta_{x}+\ep\dfrac{\nabla_{\bar{x}}
a}{a}\nabla_x-\partial_{tt}\right]\varphi_1+N_0(\varphi_1)= \ep^4
Q(\ep s)\psi'_0(t)h''(\ep s)+R_0(\ep s,t,h)
\end{equation*}

where the function $R_0=R_0(\ep s,t,h(\ep s),h'(\ep s))$ has
Lipschitz dependence in variables $h,h'$ on the ball
$$
\|h\|_{L^{\infty}(\R)}+\|h'\|_{L^{\infty}(\R)}\leq \K \ep.
$$

Moreover, under our set of assumptions and the observation made on
$\psi_0$, it turns out that for any $\lambda\in(0,1)$:
$$
\|R_0(\ep s,t,h)\|_{C^{0,\lambda}(B_1(s,t))} \leq C\ep^4(1+|\ep
s|)^{-4-\alpha}e^{-\sigma|t|}.
$$

\medskip
With this remarks, we can write the error created by $v_1$ in
\eqref{ApproximationOfASolutionEq6}, as
$$
S(v_1)=-\ep^2\J_{a}[h](\ep s)w'(t)+\ep^3\left[k^3(\ep
s)+\frac{1}{2}\partial_{tt}\left(\frac{\partial_{t}a}{a}\right) (\ep
s,0)\right]\frac{\int_{\R}\tau^2
w'(\tau)^2d\tau}{\int_{\R}w'(\tau)^2d\tau}w'(t)
$$
\begin{equation}
+\ep^4 Q(\ep s)\psi'_0(t) h''(\ep s)-\ep^3(t+h)A_0(\ep
s,\ep(t+h))h''(\ep s)w'(t) \,+\,R_1(\ep s,t,h(\ep s),h'(\ep s))
\label{ApproximationOfASolutionEq11}
\end{equation}

where
\begin{align}
R_1&=\ep^2|h'|^2w''(t)+R_0(\ep s,t)+\ep^3(t+h)A_0(\ep s,\ep(t+h))|h'|^2w''(t)\notag\\[0.2cm]
&-\ep^3(t+h)\tilde{B}_0(\ep s,\ep(t+h))h'w'(t)+
\ep^4(t+h)\mathit{O}\left(\partial_{\bs{t}\bs{t}\bs{t}}\left(
\frac{\partial_{\bs{t}}a}{a}\right)+k^4\right)t^2w'(t).
\label{DefinitionR1}
\end{align}
Furthermore, $R_1=R_1(\ep s,t,h(\ep s),h'(\ep s))$ satisfies that
$$
|\partial_{\imath} R_1(\ep s,t,\imath,\jmath)|+
|\partial_{\jmath} R_1(\ep s,t,\imath,\jmath)|+
|R_1(\ep s,t,\imath,\jmath)|\leq
 C\ep^4(1+|\ep s|)^{-2-2\alpha}e^{-\sqrt{2}|t|}
 $$

with the constant $C$ above depending on the number $\K$ of
condition \eqref{BallConditionH1}, but independent of $\ep>0$.

We can summarize this discussion by saying that
\begin{multline*}
S(v_1)\,+\,\ep^2\J_{a}[h](\ep s)w'(t)\,\\
-\ep^3\left[k^3(\ep
s)+\frac{1}{2}\partial_{tt}\left(\frac{\partial_{t}a}{a}\right) (\ep
s,0)\right]\frac{\int_{\R}\tau^2
w'(\tau)^2d\tau}{\int_{\R}w'(\tau)^2d\tau}w'(t)\, \\
=\,\mathit{O}(\ep^4(1+|\ep s|)^{-2-\alpha}e^{-\sigma|t|})
\end{multline*}
for $x=X_{\ep,h}(s,t)\in \Nn_{\ep,h}$.

\medskip
\subsection{The global approximation.} The approximation
$v_1(x)$ in \eqref{ImproveAproximation} will be sufficient for our
purposes. However, it is defined only in the region
\begin{align}
\Nn_{\ep,h}=\left\{x=X_{\ep,h}(s,t)\in\R^2/\; |t+h(\ep
s)|<\dfrac{\delta}{\ep}+c_0|s|=:\rho_{\ep}(s)\right\}\label{ApproximationOfASolutionEq12}
\end{align}

Since we are assuming that $\Gamma$ is a connected and simple and
that it also possesses two ends departing from each other, it
follows that $\R^2\setminus\Gamma_{\ep}$ consists of precisely two
components $S_{+}$ and $S_{-}$. Let us use the convention that
$\nu_{\ep}$ points towards $S_{+}$. The previous comments allow us
to define in $\R^2\setminus\Gamma_{\ep}$ the function
\begin{equation*}
\H(x):=\left\{\begin{array}{lr}+1 & \text{ if }x\in S_{+}\\-1 &
\text{ if }x\in S_{-}
\end{array}\right.
\end{equation*}

Let us consider $\eta(s)$ a smooth cut-off function with $\eta(s)=1$
for $s<1$ and $=0$ for $s>2$, and define
\begin{equation*}
\zeta_{3}(x):=\left\{\begin{array}{cr} \eta(|t+h(\ep
s)|-\rho_{\ep}(s)+3) & \text{ if }x\in\Nn_{\ep,h}\\ 0 &\text{ if
}x\notin\Nn_{\ep,h} \end{array}
\right.
\end{equation*}
where $\rho_{\ep}$ is defined in
\eqref{ApproximationOfASolutionEq12}.

Next, we consider as global approximation the function $\w(x)$
defined as
\begin{align}
\w:=\zeta_{3}\cdot v_1+(1-\zeta_{3})\cdot
\H\label{GlobalApproximation}
\end{align}
where $u_1(x)$ is given by \eqref{ImproveAproximation}.

\medskip
Using that $\H(\ep^{-1}\, \bar{x})$ is an exact solution to
\eqref{IntrodEq12} in $\R^2\setminus\Gamma$, the error of global
approximation can be computed as
\begin{equation}
S(\w)\,=\,\Delta_{x}\w+\ep\dfrac{\nabla_{\bar{x}} a}{a}\nabla_x
\w-F'(\w) \,=\,\zeta_{3}S(v_1)+E \label{ErrorGlobalApprox}
\end{equation}
where $S(v_1)$ is computed in \eqref{ApproximationOfASolutionEq11}
and the term $E$ is given by
\begin{multline}
E=\Delta_{x}\zeta_{3}(v_1-\H)+2\nabla_x\zeta_{3}\nabla_x(v_1-\H)
+(v_1-\H)\dfrac{\nabla_{\bar{x}} a}{a}\nabla_x
\zeta_{3}\\
-F'\bigl(\zeta_{3}v_1+(1-\zeta_{3})\H\bigr)+\zeta_{3}F'(v_1)
\label{ErrorGlobalApprox2}
\end{multline}

It is worth to mention that the from the form of the neighborhood
$\Nn_{\ep,h}$ in \eqref{ApproximationOfASolutionEq12}, and from the
choice of $v_1$, one can readily check that for every $x=X_{\ep,h}
\in \Nn_{\ep,h}$
$$
|v_1(x)-\H(x)|\leq e^{-\sqrt{2}|t+h(\ep s)|}, \quad
\rho_{\ep}-2<|t+h(\ep s)|<\rho_{\ep}-1
$$

and therefore
$$
|E|\leq C\ e^{-\sqrt{2}|t+h(\ep s)|}\leq Ce^{-\sqrt{2}\delta/\ep}
\cdot e^{-c|s|}e^{-\sigma|t|}
$$

for  some $0<\sigma<\sqrt{2}$ and $c>0$ small.

%%%%%%%%%%%%%%%%%%%%%%%%%%%%%%%%%%%%%%%%%%%%%%%%%%%%%%%%%%%%%%%%%%%%%%
%%%%%%%%%%%%%%%%%%%%%%%%%%%%%%%%%%%%%%%%%%%%%%%%%%%%%%%%%%%%%%%%%%%%%%

\section{The proof of
Theorem \ref{TheoremMainResult}}

In this section we sketch the proof of Theorem
\ref{TheoremMainResult} leaving the detailed proofs of every
proposition mentioned here for subsequent sections.

We look for a solution $u$ of the inhomogeneous Allen-Cahn
equation~\eqref{EqAllenCahn} in the form
\begin{equation*}
u=\w+\varphi 
\end{equation*}

where $\w$ is the global approximation defined in
\eqref{GlobalApproximation} and $\varphi$ is small in some suitable
sense. We find that $\varphi$ must solve the following nonlinear
equation
\begin{align}
\Delta_{x}\varphi+\ep\dfrac{\nabla_{\bar{x}}
a}{a}\nabla_{x}\varphi-F''(\w)\varphi\,=\,-S(\w)\,-\,N_1(\varphi)\label{ProofThmEq2}
\end{align}
where
\begin{align}
S(\w)&:=\Delta_{x}\w+\ep\dfrac{\nabla_{\bar{x}} a}{a}\nabla_{x}
\w -F'(\w) \label{DefSw}\\[0.3cm]
N_1(\varphi)&:=-F'(\w+\varphi)+F'(\w)+F''(\w)\varphi\label{DefN1}
\end{align}

\medskip
We introduce several norms that will allow us to set up an
appropriate functional scheme to solve \eqref{ProofThmEq2}. Let us
consider $\eta(s)$, a cut-off function with $\eta(s)=1$ for $s<1$
and $\eta=0$ for $s>2$, we define
\begin{align}
\zeta_n(x):=\left\{\begin{array}{cr}\eta
\left(|t+h(\ep s)|-\rho_{\ep}(s)+n\right) & \text{ if }
x\in \Nn_{\ep,h}\\
 0 & \text{ if } x\notin \Nn_{\ep,h}
 \end{array}\right.\label{ProofThmEq3}
\end{align}
where $\rho_{\ep}$ and $\Nn_{\ep,h}$ are set in
\eqref{ApproximationOfASolutionEq12}.

\medskip
Let us consider $\lambda\in(0,1)$, $b_1,b_2>0$ fixed and satisfying
that $b_1^2+b_2^2<(\sqrt{2}-\tau)/2$ for $\tau>0$. Define the weight
function $K(x)$, for $x=(x_1,x_2)\in \R^2$, as follows
\begin{align}
K(x):=\zeta_2(x)\left[e^{\sigma|t|/2}(1+|\ep
s|)^{\mu}\right]+(1-\zeta_2(x))e^{b_1|x_1|+b_2|x_2|}.
\label{WeightFunction}
\end{align}

For a function $g(x)$ defined in $\R^2$, we set the norms
\begin{align}
&\|g\|_{L^{\infty}_{K}(\R^2)}:=\sup_{x\in\R^2}K(x)\|g\|_{L^{\infty}(B_1(x))}\label{DefNorm01}\\[0.2cm]
&\|g\|_{C^{0,\lambda}_{K}(\R^2)}:=\sup_{x\in\R^2}K(x)\|g\|_{C^{0,\lambda}(B_1(x))}\label{DefNorm02}
\end{align}

On the other hand, consider $\ep>0$, $\mu\geq 0$ and
$0<\sigma<\sqrt{2}$. For functions $g(s,t)$ and $\phi(s,t)$ defined
in whole $\R\times\R$, we set
\begin{align}
&\|g\|_{C^{0,\lambda}_{\mu,\sigma}(\R^2)}:=\sup_{(s,t)\in\R\times\R}(1+|\ep
s|)^{\mu}e^{\sigma|t|}
\|g\|_{C^{0,\lambda}(B_1(s,t))}\label{DefNorm10}\\
&\|\phi\|_{C^{2,\lambda}_{\mu,\sigma}(\R^2)}:=\|D^2\phi\|_{C^{0,\lambda}_{\mu,\sigma}(\R^2)}+\|D\phi\|_{L^{\infty}_{\mu,\sigma}(\R^2)}
+\|\phi\|_{L^{\infty}_{\mu,\sigma}(\R^2)}\label{DefNorm13}
\end{align}

Finally, given $\alpha>0$ and $\lambda\in(0,1)$, consider,for a
function $f$ defined in $\R$, the norm
\begin{align}
\|f\|_{C^{0,\lambda}_{2+\alpha,*}(\R)}:=\sup_{\bs{s}\in\R}(1+|\bs{s}|
)^{2+\alpha}\|f\|_{C^{0,\lambda}(\bs{s}-1,\bs{s}+1)}.\label{DefNorm3}
\end{align}

Recall also that the parameter function $h(\bar{s})$ satisfies for
some $\lambda\in(0,1)$
$$
\|h\|_{C^{2,\lambda}_{2+\alpha,*}(\R)}\leq \K\ep
$$
where
\begin{align}
\|h\|_{C^{2,\lambda}_{2+\alpha,*}(\R)}:=
\|h\|_{L^{\infty}(\R)}+\|(1+|\bs{s}|)^{1+\alpha}h'\|_{L^{\infty}(\R)}+
\|h''\|_{ C^{0,\lambda}_{2+\alpha,*}(\R)}.\label{DefNorm4}
\end{align}

\medskip
{\bf 4.1 The gluing procedure. } In order to solve
\eqref{ProofThmEq2}, let us look for a solution $\varphi$ of
problem, having the form
\begin{equation*}
\varphi(x)=\zeta_3(x)\phi(s,t)+\psi(x), \quad \hbox{for } x \in\R^2
\end{equation*}

where $\phi$ is defined in $\Gamma_{\ep}\times\R$ and $\psi$ is
defined in entire $\R^2$. Using that $\zeta_3\cdot\zeta_4=\zeta_4$,
we get that \eqref{ProofThmEq2} reads as
$$
S(\w+\varphi)=\zeta_3\left[\Delta_{x}\phi+\ep
\dfrac{\nabla_{\bar{x}} a}{a}\nabla_{x}\phi+f'(u_1)\phi\right]
$$
$$
\,+\,\zeta_4\left[\,\,[f'(u_1)-f'(H(t))]\psi+N_1(\psi+\phi)+S(u_1)\right]
$$
$$
+\Delta_{x}\psi+\ep \dfrac{\nabla_{\bar{x}}a}{a}\nabla_{x}\psi+
[(1-\zeta_4)f'(u_1)+\zeta_4f'(H(t))]\psi+(1-\zeta_3)S(\w)
$$
$$
+(1-\zeta_4)N_1(\psi+\zeta_3\phi)+2\nabla_{x}\zeta_3\nabla_{x}\phi+\phi\Delta_{x}\zeta_3+\ep\phi\dfrac{\nabla_{\bar{x}}a}{a}\nabla_{x}\zeta_3\label{ProofThmEq5}
$$
where $H(t)$ is some increasing smooth function satisfying
\begin{align}
H(t)=\left\{\begin{array}{ll} +1 & \text{ if } t>1\\ -1 & \text{ if
}t<-1. \end{array} \right.\label{ProofThmEq6}
\end{align}

In this way, we will have constructed a solution
$\varphi=\zeta_3\phi+\psi$ to problem \eqref{ProofThmEq2} if we
require that the pair $(\phi,\psi)$ satisfies the coupled system
below
\begin{multline}
\Delta_{x}\phi+\ep\dfrac{\nabla_{\bar{x}}
a}{a}\nabla_{x}\phi+f'(u_1)\phi\\
+\zeta_4[f'(u_1)-f'(H(t))]\psi+\zeta_4N_1(\psi+\phi)+S(v_1)=0
\;\text{ for } |t|<\dfrac{\delta}{\ep}\label{ProofThmEq7}
\end{multline}

\begin{multline}
\Delta_{x}\psi+\ep\dfrac{\nabla_{\bar{x}} a}{a}\nabla_{x}
\psi+[(1-\zeta_4)f'(u_1)+\zeta_4f'(H(t))]\psi+(1-\zeta_3)S(\w)\\
+(1-\zeta_4)N_1(\psi+\zeta_3\phi)\,+\,2\nabla_{x}\zeta_3\nabla_{x}\phi
+\phi\Delta_{x}\zeta_3+\ep\phi\dfrac{\nabla_{\bar{x}}
a}{a}\nabla_{x}\zeta_3=0,\quad\text{ in }\quad
\R^2\label{ProofThmEq8}
\end{multline}

Next, we will extend equation \eqref{ProofThmEq7} to entire
$\R\times\R$. To do so, let us set
\begin{align}
\B(\phi)=\zeta_0
\tilde{\B}_0(\phi):=\zeta_0[\Delta_{x}-\partial_{tt}-\partial_{ss}]\phi
\label{ProofThmEq9}
\end{align}

where $\Delta_{x}$ is expressed in local coordinates, using formula
\eqref{LaplacianXeph}, and $\B(\phi)$ is understood to be zero for
$|t+h(\ep s)|>\rho_{\ep}(s,t) - 2$. Thus
equation~\eqref{ProofThmEq7} is  extended as
\begin{align}
\partial_{tt}\phi+&\partial_{ss}\phi+\ep\dfrac{\nabla_{\bar{x}}
a}{a}\nabla_{x}\phi+\B(\phi)+f'(w(t))\phi=-\tilde{S}(u_1)\notag\\[0.3cm]
&-\left\{[f'(u_1)-f'(w)]\phi+\zeta_4[f'(u_1)-f'(H(t))]\psi+
\zeta_4N_1(\psi+\phi)\right\},\;\text{ in }\;
\R\times\R\label{ProofThmEq10}
\end{align}

where we have denoted
\begin{align}
\tilde{S}(u_1)(s,t)&= -\ep^2\J_{a}[h](\ep s)w'(t)-\ep^3\left[k^3(\ep
s)+\frac{1}{2}\partial_{tt}\left(\frac{\partial_{t}a}{a}\right) (\ep
s,0)\right]\frac{\int_{\R}\tau^2
w'(\tau)^2d\tau}{\int_{\R}w'(\tau)^2d\tau}w'(t)\,
\notag\\[0.2cm]
&+\ep^4 Q(\ep s)\psi'_0(t)\cdot h''(\ep
s)+\zeta_{0}\biggl\{\ep^3(t+h)A_0(\ep s,\ep(t+h))h''(\ep s)\cdot
w''(t)+R_1\biggr\}.\label{Su1tilde}
\end{align}

Recall from \eqref{DefinitionR1} that
$$ R_1=R_1(\ep s,t,h(\ep s),h'(\ep s))$$
satisfies
\begin{align}
|\partial_{\imath}R_1(\ep s,t,\imath,\jmath)|+
|\partial_{\jmath}R_1(\ep s,t,\imath,\jmath)|+ |R_1(\ep
s,t,\imath,\jmath)| \leq C\ep^4(1+|\ep
s|)^{-2-2\alpha}e^{-\sqrt{2}|t|}\label{R1Size}
\end{align}

In order to solve the resulting system
\eqref{ProofThmEq8}-\eqref{ProofThmEq10}, we focus first on solving
equation~\eqref{ProofThmEq8} in $\psi$ for a fixed and small $\phi$.
We make use of the important observation that the term
$[(1-\zeta_4)f'(u_1)+\zeta_4f'(H)]$, is uniformly negative and so
the operator in \eqref{ProofThmEq8} is qualitatively similar to
$\Delta_{x}+\ep\nabla_{\bar{x}}a/a\cdot\nabla_{x}-2$. A direct
application of the contraction mapping principle lead us to the
existence of a solution $\psi=\Psi(\phi)$, according to the next
proposition whose detailed proof is carried out in Section 5.

\begin{prop}{\label{LemmaNonlinearPsiEq}
Let $\lambda\in(0,1)$, $\sigma\in(0,\sqrt{2})$,
$\mu\in(0,2+\alpha)$. There is $\ep_0>0$, such that for any small
$\ep\in(0,\ep_0)$ the following holds. Given $\phi$ with
$\|\phi\|_{C^{2,\lambda}_{\mu,\sigma}(\R^2)}\leq 1$, there is a
unique solution $\psi=\Psi(\phi)$ to equation \eqref{ProofThmEq8}
with
\begin{equation*}
\|\psi\|_{X}:=\|D^2\psi\|_{C^{0,\lambda}_{K}(\R^2)}+\|D\psi\|_{L^{\infty}_{K}(\R^2)}+\|\psi\|_{L^{\infty}_{K}(\R^2)}\leq
Ce^{-\sigma\delta/2\ep}
\end{equation*}
Besides, $\Psi$ satisfies the Lipschitz condition
\begin{align}
\|\Psi(\phi_1)-\Psi(\phi_2)\|_{X}\leq
Ce^{-\sigma\delta/2\ep}\|\phi_1-\phi_2\|_{C^{2,\lambda}_{
\mu,\sigma}(\R^2)}\label{LipschitzPsi}
\end{align}
where the norms $L^{\infty}_{K}, C^{0,\lambda}_{K},
C^{2,\lambda}_{\mu,\sigma}$ are defined in
\eqref{DefNorm01}-\eqref{DefNorm02}-\eqref{DefNorm13}. }
\end{prop}

\medskip
{\bf 4.2 Solving the Nonlinear Projected Problem.} Using
proposition \ref{LemmaNonlinearPsiEq}, we solve \eqref{ProofThmEq10}
replacing $\psi$ with the nonlocal operator $\psi=\Psi(\phi)$.
Setting
\begin{equation*}
\N(\phi):=\B(\phi)+\ep \dfrac{\nabla_{\bar{x}} a}{a}\nabla_{x}
\phi+[f'(u_1)-f'(w)]\phi+\zeta_4[f'(u_1)-f'(H)]\Psi(\phi)+
\zeta_4N_1(\Psi(\phi)+\phi)
\end{equation*}
our problem is reduced to find a solution $\phi$ to the following
nonlinear, nonlocal problem
\begin{align}
\partial_{tt}\phi+\partial_{ss}\phi+f'(w)\phi=
-\tilde{S}(u_1)-\N(\phi)\;\text{ in }\; \R\times\R.
\label{NonlocalEquation}
\end{align}

Before solving \eqref{NonlocalEquation}, we consider the problem of
finding a $(\phi,c)$  a solution to the following nonlinear
projected problem
\begin{equation}\label{ProjectedProblem}
\scriptstyle{(NPP)}\;\left\{
\begin{aligned}
&\partial_{tt}\phi+\partial_{ss}\phi+f'(w)\phi=
-\tilde{S}(u_1)-\N(\phi)+c(s)w'(t) \;\text{ in }\; \R\times\R\\[0.2cm]
&\int_{\R}\phi(s,t)w'(t)dt=0, \quad\forall s\in\R.
\end{aligned}\right.
\end{equation}

Solving problem \eqref{ProjectedProblem} amounts to eliminate the
part of the right hand side in \eqref{NonlocalEquation}, that do not
contribute to the projections onto $w'(t)$, namely
$\int_{\R}[\tilde{S}(u_1)+N(\phi)]w'(t)dt$.

Since, we have that
\begin{equation*}
\left\|\tilde{S}(u_1)+\ep^2\J_{a}[h](\ep s)\cdot
w'(t)-\ep^3\left[k^3(\ep
s)+\frac{1}{2}\partial_{tt}\left(\frac{\partial_{t}a}{a}\right) (\ep
s,0)\right]\hat{c}w'(t)\,
\right\|_{C^{0,\lambda}_{\mu,\sigma}(\R^2)}\leq C\ep^4
\end{equation*}
where
$$
\hat{c}=\|w'\|^{-2}_{L^2(\R)}\int_{\R}t^2w'(t)^2dt
$$
and due to the fact that $\N(\phi)$ defines a
contraction within a ball centered at zero with radius $\mathit{O}(\ep^4)$
in norm $C^1$, we conclude the existence of a unique small solution
of problem \eqref{ProjectedProblem} whose size is $\mathit{O}(\ep^4)$ in
this norm. This solution $\phi$ turns out to define an operator in
$h$, namely $\phi=\Phi(h)$, which exhibits a Lipschitz character in
norms $\| \cdot \|_{C^{2,\lambda}_{\mu,\sigma}(\R^2)}$. We collect
the discussion in the following proposition.

\begin{prop}{\label{PropNonlinearTheory}
Given $\lambda\in(0,1),\ \mu\in(0,2+\alpha]$ and
$\sigma\in(0,\sqrt{2})$, there exists a constant $K>0$ such that the
nonlinear projected problem \eqref{ProjectedProblem} has a unique
solution $\phi=\Phi(h)$ with
\begin{equation*}
\|\phi\|_{C^{2,\lambda}_{\mu,\sigma}(\R^2)}\leq
K\ep^4
\end{equation*}
Besides $\Phi$ has small a Lipschitz dependence on $h$ satisfying
condition \eqref{BallConditionH1}, in the sense
\begin{equation}
\|\Phi(h_1)-\Phi(h_2)\|_{C^{2,\lambda}_{\mu,\sigma}(\R^2)}\leq
C\ep^3\|h_1-h_2\|_{C^{2,\lambda}_{\mu,*}(\R)}\label{LipschitzPhi}
\end{equation}
for any $h_1,h_2\in C^{2,\lambda}_{loc}(\R)$ with
$\|h_i\|_{C^{2,\lambda}_{\mu,*}(\R)}\leq \K\ep$. }
\end{prop}

The proof of this proposition is left to section 5, where a complete
study of the linear theory needed to solve is discussed.

\medskip
{\bf 4.3 Adjusting the nodal set. } In order to conclude the proof
of Theorem~\ref{TheoremMainResult}, we have to adjust the parameter
function $h$ so that the nonlocal term
\begin{align}
c(s)\int_{\R}|w'(t)|^2dt=\int_{\R}\tilde{S}(u_1)(\ep s,t)
w'(t)dt+\int_{\R}\N(\Phi(h))(s,t)w'(t)dt\label{ProofThmEq12}
\end{align}
becomes identically zero, and consequently we obtain a genuine
solution to equation \eqref{IntrodEq12}.

Setting $c_{*}:=\int_{\R}|w'(t)|^2dt$, using
expression~\eqref{Su1tilde}, and carrying out the same computation
we did in \eqref{ApproximationOfASolutionEq3}, we obtain that
\begin{equation*}
\int_{\R}\tilde{S}(u_1)(\ep s,t)w'(t)dt=-c_*\ep^2\J_{a}[h](\ep
s)+c_*\ep^2G_1(h)(\ep s)
\end{equation*}
where
\begin{align}
c_*G_1(h)(\ep s)&:= -\ep\left[k^3(\ep
s)+\frac{1}{2}\partial_{tt}\left(\frac{\partial_{t}a}{a}\right) (\ep
s,0)\right]\frac{\int_{\R}\tau^2
w'(\tau)^2d\tau}{\int_{\R}w'(\tau)^2d\tau}w'(t)\notag\\
&+\ep h''(\ep s)\int_{\R}\zeta_0(t+h)A_0(\ep s,\ep(t+h))w''(t)w'(t)dt\notag\\
&+\ep^2Q(\ep s)h''(\ep
s)\int_{\R}\psi'_0(t)w'(t)dt+\ep^{-2}\int_{\R}\zeta_0\ R_1(\ep
s,t,h,h')w'(t)dt\label{ProofThmEq14}
\end{align}
and we recall that $R_1$ is of size $O(\ep^4)$ in the sense of
\eqref{R1Size}. Thus setting
\begin{align}
c_*G_2(h)(\ep s):=\ep^{-2}\int_{\R}\N(\Phi(h))(s,t)w'(t)dt,\quad
\G(h)(\ep s):=G_1(h)(\ep s)+G_2(h)(\ep s)\label{ProofThmEq15}
\end{align}
it turns out that equation~\eqref{ProofThmEq12} is equivalent to
$$
c(s)\cdot c_*= -c_*\ep^2\J_{a,\Gamma}[h](\ep s)+c_*\ep^2G_1(h)(\ep
s)+c_*\ep^2G_2(h)(\ep s)
$$

Therefore the condition $c(s)=0$ is equivalent to the following
nonlinear problem on $h$
\begin{align}
\J_{a,\Gamma}[h](\ep s)=h''(\ep s)+\dfrac{\partial_{\bs{s}}a(\ep
s,0)}{a(\ep s,0)}h'(\ep s)-Q(\ep s)h(\ep s)=\G[h](\ep s), \quad
\hbox{in }\R \label{JacobiEquation}
\end{align}
Consequently, we will have proved Theorem~\ref{TheoremMainResult},
if we find a function $h$, solving
equation~\eqref{JacobiEquation}.

Hence, we need to devise a corresponding solvability theory for the
linear problem
\begin{align}
\qquad \J_{a}[h](\bs{s})=f(\bs{s}),\quad \hbox{in
}\R\label{LinearProjectedJacobiProblem}.
\end{align}

The next result addresses this matter.

\begin{prop}{\label{PropLinearJacobiTheory}
Given $\alpha>0$, $\lambda\in(0,1)$, and a function $f$ with
$\|f\|_{C^{0,\lambda}_{2+\alpha,*}(\R)}<\infty$, assume that
$\Gamma$ is a smooth curve satisfying \eqref{IntrodEq13}. If,
$\Gamma$ is nondegenerate respect to the potential $a$ and
conditions \eqref{smoothnessfora}-\eqref{IntrodEq14} hold, then
there exists a unique bounded solution $h$ of problem
\eqref{LinearProjectedJacobiProblem}, and there exists a positive
constant $C=C(a,\Gamma,\alpha)$ such that
\begin{equation*}
\|h\|_{C^{2,\lambda}_{2+\alpha,*}(\R)}\leq C\|f\|_{C^{0,\lambda}_{2+\alpha,*}(\R)}
\end{equation*}
with the norms defined in \eqref{DefNorm3}-\eqref{DefNorm4}.
}\end{prop}

In section 6, we study in detail the proof of this proposition. For
the time being, let us note that, $\G$ is a small operator of size
$\mathit{O}(\ep)$ uniformly on functions $h$ satisfying
\eqref{BallConditionH1}. Hence Proposition
\ref{PropLinearJacobiTheory} plus the contraction mapping principle
yield the next result, which ensures the solvability of the
nonlinear Jacobi equation. Its detailed proof can be found in
section 7.

\begin{prop}{\label{PropNonlinearJacobiTheory}
Given $\alpha>0$ and $\lambda\in(0,1)$, there exist a positive
constant $\K>0$ such that for any $\ep>0$ small enough the following
holds. There is a unique solution $h$ of \eqref{JacobiEquation} on
the region \eqref{DefNorm4}, namely
$\|h\|_{C^{2,\lambda}_{2+\alpha,*}(\R)}\leq \K\ep$. }
\end{prop}

and this concludes the proof of Theorem \ref{TheoremMainResult}.

\medskip
The rest of the paper is devoted to give fairly detailed proofs of
every result stated in this section.

%%%%%%%%%%%%%%%%%%%%%%%%%%%%%%%%%%%%%%%%%%%%%%%%%%%%%%%%%%%%%%%%%%%%%%%%
%%%%%%%%%%%%%%%%%%%%%%%%%%%%%%%%%%%%%%%%%%%%%%%%%%%%%%%%%%%%%%%%%%%%%%%%

\section{Gluing reduction and solution to the projected problem}

This section is devoted to give fairly detailed proves of
propositions \ref{LemmaNonlinearPsiEq} and
\ref{PropNonlinearTheory}. In what follows, we refer to the notation
and to the objects introduced in sections 3 and 4.

\medskip
\subsection{The proof of proposition
\ref{LemmaNonlinearPsiEq}.} In this part we prove proposition
\ref{LemmaNonlinearPsiEq}. To do so, let us first consider the
linear problem
\begin{align}
\Delta_{x}\psi+\ep \dfrac{\nabla_{\bar{x}} a}{a}\nabla_x\psi
-W_{\ep}(x)\psi=g(x), \quad\text{ in }\quad \R^2
\label{LinearOuterEq}
\end{align}
where
$$
-W_{\ep}(x):=[(1-\zeta_4)f'(u_1)+\zeta_4f'(H(t))].
$$

Observe that the dependence in $\ep$ is implicit on the cut-off
function $\zeta_4$, defined in \eqref{ProofThmEq3}.

Let us observe that for any $\ep>0$ small enough, the term $W_{\ep}$
satisfies the global estimate $0<\beta_1<W_{\ep}(x)<\beta_2$ for a
certain positive constants $\beta_1,\beta_2$. In fact, we can chose
$\beta_1:=\sqrt{2}-\tau$ for any arbitrary small $\tau>0$. To
address the study of this equation, recall the definition of the
weighted norms:
$$ \|g\|_{L^{\infty}_{K}(\R^2)}:=\sup_{x\in\R^2}K(x)\|g\|_{L^{\infty}(B_1(x))}, \quad
\|g\|_{C^{0,\lambda}_{K}(\R^2)}:=\sup_{x\in\R^2}K(x)\|g\|_{C^{0,\lambda}(B_1(x))}$$
with $K$ is given by~\eqref{WeightFunction}.

%%---------------------------------------------------------------------------------------------------------
\begin{lemma}{\label{LemmaOuterProblem}
For any $\lambda\in(0,1)$, there are numbers $C>0$, and $\ep_0>0$
small enough, such that for $0<\ep<\ep_0$ and any given continuous
function $g=g(x)$ with $\|g\|_{C^{0,\lambda}_{K}(\R^2)}<+\infty$,
the equation~\eqref{LinearOuterEq} has a unique solution
$\psi=\Psi(\phi)$ satisfying the a priori estimate:
\begin{equation*}
\|\psi\|_{X}:=\|D^2\psi\|_{C^{0,\lambda}_{K}(\R^2)}+\|D\psi\|_{L^{\infty}_{K}(\R^2)}+\|\psi\|_{L^{\infty}_{K}(\R^2)}\leq
C\|g\|_{C^{0,\lambda}_{K}(\R^2)}
\end{equation*}
}
\end{lemma}

The proof of this lemma follows the same lines of lemma 7.1 in
\cite{9} with no significant changes. We leave details to the
reader, but we do comment on the estimate
\begin{align}
\|\psi\|_{L^{\infty}_{K}(\R^2)}\leq
C\|g\|_{C^{0,\lambda}_{K}(\R^2)}\label{FirstOuterEstimate}
\end{align}

for $\ep>0$ small enough and any bounded solution $\psi$ of
\eqref{LinearOuterEq}. It follows directly from a sub-supersolution
scheme, using that $b_1^2+b_2^2<(\sqrt{2}-\tau)/2$ and the fact that
the function
$$
\psi_0(x):=e^{R_0}\|\psi\|_{\infty}\cdot
\left\{\zeta_3(x)[e^{-\sigma|t|/2}(1+|\ep
s|)^{-\mu}]+(1-\zeta_3(x))e^{-b_1|x_1|-b_2|x_2|}\right\}
$$
ca be readily checked to be a positive supersolution of
\eqref{LinearOuterEq}, provided that $R_0>0$ sufficiently large.

Hence, we can use the maximum principle within the annulus
$B_{R_1}(\vec{0})\setminus B_{R_0}(\vec{0})$ with a barrier function
of the form $\psi_0+\theta e^{\sqrt{\beta_1/2}(|x_1|+|x_2|)}$ for
$\theta>0$ small, to find that
$$
K(x)|\psi(x)| \leq M \|\psi\|_{L^{\infty}(\R^2)}\leq
\tilde{M}\|g\|_{C^{0,\lambda}_{K}(\R^2)}, \quad x\in\R^2.
$$
from which estimate~\eqref{FirstOuterEstimate} follows.

Now we have all the ingredients need for the proof of proposition
\ref{LemmaNonlinearPsiEq}. Let us set $\psi:=\Upsilon(g)$ the
solution of equation~\eqref{LinearOuterEq} predicted by lemma
\ref{LemmaOuterProblem}. We can write problem \eqref{ProofThmEq8} as
a fixed point problem in the space $X$ of functions $\psi\in
C^{2,\lambda}_{loc}(\R^2)$ with $\|\psi\|_{X}<\infty$, as
\begin{align}
\psi=\Upsilon(g_1+G(\psi)),\quad \psi\in
X\label{FixedPointPsiProblem}
\end{align}
where
$$
g_1:=(1-\zeta_3)S(\w)+2\nabla_x\zeta_3\nabla_x\phi+\phi\Delta_x\zeta_3
+\ep\phi\dfrac{\nabla_{\bar{x}} a}{a}\nabla_x\zeta_3,
$$

$$
G(\psi):=(1-\zeta_4)N_1(\psi+\zeta_3\phi). \label{NonlinearityPsi}
$$

Consider $\mu\in(0,2+\alpha)$, $\sigma\in(0,\sqrt{2})$ and
$\alpha>0$ fixed and a function $h$ satisfying \eqref{DefNorm4}.
Consider also a function $\phi=\phi(s,t)$, satisfying
$\|\phi\|_{C^{2,\lambda}_{\mu,\sigma}(\R^2)}\leq 1$.

Note that the derivatives of $\zeta_3$ are nontrivial only within
the region $\rho_{\ep}-2<|t+h(\ep s)|<\rho_{\ep}-1$, with
$\rho_{\ep}$ defined in \eqref{ApproximationOfASolutionEq12}. Taking
into account the weight $K(x)$~\eqref{WeightFunction}, we find that
\begin{align*}
K(x)\left|2\nabla_x\zeta_3\nabla_x\phi+\phi\Delta_x\zeta_3+\ep\phi\dfrac{\nabla_{\bar{x}}
a}{a}\nabla_x\zeta_3\right|
&\leq C_a K(x)e^{-\sigma |t|}(1+|\ep s|)^{-\mu}\|\phi\|_{C^{2,\lambda}_{\mu,\sigma}(\R^2)}\\
&\leq C_a\ e^{-\sigma\delta/2\ep}e^{\sigma/2(-c_0|s|+2+|h|)}
\|\phi\|_{C^{2,\lambda}_{\mu,\sigma}(\R^2)}
\end{align*}

provided that
$$
c_0<\frac{b_2\delta}{a_2}, \qquad\frac{c_0\,\theta}{1-\theta}\leq
b_2
$$

conditions that holds, since we can take  $c_0>0$ small
enough,independent of $\ep>0$ and $\theta$ small depending maybe on
$c_0$. At the end, there are some constants $\tilde{c}_0$ and
$\tilde{\delta}>0$, depending on $\Gamma$ and $a(x,y)$, such that
the right hand side satisfies for
$x\in\R^2$$$\left\|2\nabla_x\zeta_3\nabla_x\phi+\phi\Delta_x\zeta_3+\phi\dfrac{\nabla_{\bar{x}}
a}{a}\nabla_x\zeta_3\right\|_{C^{0,\lambda}B(x,1)}\leq C_{a,\Gamma}\
e^{-\sigma\tilde{\delta}/\ep}e^{-\tilde{c}_0|x|}\|\phi\|_{C^{2,\lambda}_{\mu,\sigma}(\R^2)}$$where
these constants are explicitly $\tilde{\delta}:=\delta-c_0 a_2/b_2$,
$\tilde{c}_0:=\sigma\theta c_0/b_2$, and where we emphasize that
$C_{a,\Gamma}$ does not depend on $\ep$.

Expressions \eqref{ErrorGlobalApprox}-\eqref{ErrorGlobalApprox2} for
$S(\w)$ imply that
$\|S(\w)\|_{C^{0,\lambda}_{\mu,\sqrt{2}}(\R^2)}\leq C\ep^3$. In
par\-ticular, the exponential decay exhibited by
$w',w'',\psi_0,\psi_1$ in $t-$variable imply
$$|(1-\zeta_3)S(\w)|=|(1-\zeta_3)\zeta_3S(u_1)+(1-\zeta_3)E|
\leq C_a\ e^{-\sigma|t|}(1+|\ep s|)^{-2-\alpha}$$ Now since this
error term is vanishing everywhere but on the region
$\rho_{\ep}-2<|t+h(\ep s)|<\rho_{\ep}-1$, we can use the
definition~\eqref{WeightFunction} of the weight function $K(x)$ to
prove that
\begin{align*}
K(x)|(1-\zeta_3)S(\w)(x)|&\leq e^{\sigma|t|/2}(1+|\ep s|)^{\mu-2-\alpha}\ C_a e^{-\sigma|t|/2}e^{-(\sqrt{2}-\sigma/2)|t|}\\
&\leq C_ae^{-(\sqrt{2}-\sigma/2)(\delta/\ep+c_0|s|-|h|-2)} \leq
Ce^{-\sigma\tilde{\delta}/\ep}
\end{align*}
where we have used the expression~\eqref{ApproximationOfASolutionEq12} for $\rho_{\ep}$, and we set $\tilde{\delta}:=(\sqrt{2}/\sigma-1/2)\delta>>\delta/2$.\\
Further, the regularity in the $s-$variable of the functions
involved in $g_1$, imply that
$$ \|g_1\|_{C^{0,\lambda}_{K}(\R^2)}\leq C e^{-\sigma\delta/2\ep}$$
On the other hand, consider the set for $A>0$ large
\begin{align}
\Lambda=\{\psi\in X:\ \|\psi\|_{X}\leq A\cdot
e^{-\sigma\delta/2\ep}\} \label{SetLambda}
\end{align}
The definitions of $N_1$ in \eqref{DefN1} and $G$ in
\eqref{NonlinearityPsi}, lead us to the following computations
\begin{align*}
&(1-\zeta_4)|N_1(\Psi(\phi_1)+\zeta_3\phi_1)-N_1(\Psi(\phi_2)+\zeta_3\phi_2)|\leq\\[0.2cm] &C_{\w}(1-\zeta_4)\sup_{t\in(0,1)}|t\Psi(\psi_1)+(1-t)\Psi(\psi_2)+\zeta_3(t\phi_1+(1-t)\phi_2)|\cdot |\Psi(\psi_1)-\Psi(\psi_2)|
\end{align*}
together with
\begin{align*}
|G(\psi_1)-G(\psi_2)|&\leq (1-\zeta_4)\sup_{\xi\in(0,1)}\left|DN_1(\xi\psi_1+(1-\xi)\psi_2+\zeta_3\phi)[\psi_1-\psi_2]\right|\\
&\leq
C\|f''(\w)\|_{\infty}(1-\zeta_4)\sup_{\xi\in(0,1)}|\xi\psi_1+(1-\xi)\psi_2+\zeta_3\phi|\cdot
|\psi_1-\psi_2|
\end{align*}
The latter, plus the regularity in the $s-$variable leads the
Lipschitz character of $G$:
\begin{align*}
\|G(\psi_1)-G(\psi_2)\|_{C^{0,\lambda}_{K}(\R^2)}\leq
C\,A\,e^{-\sigma\delta/\ep}\|\psi_1-\psi_2\|_{C^{0,\lambda}_{K}(\R^2)}
\end{align*}
while
$$\|G(0)\|_{C^{0,\lambda}_{K}(\R^2)}\leq C_{\w}\|(1-\zeta_4)\zeta^2_3\phi^2\|_{C^{0,\lambda}_{K}(\R^2)} \leq C e^{-\sigma\delta/\ep}$$
In order to use the fixed point theorem, we need to estimate the
size of the nonlinear operator
\begin{align*}
\|\Upsilon(g_1+G(\psi))\|_{X}&\leq \|\Upsilon(g_1+G(\psi)-G(0))\|_{X}+
\|\Upsilon(G(0))\|_{X}\\[0.2cm]
&\leq C(\|g_1\|_{C^{0,\lambda}_{K}(\R^2)}+\|G(\psi)-G(0)\|_{C^{0,\lambda}_{K}(\R^2)}+\|G(0)\|_{C^{0,\lambda}_{K}(\R^2)})\\[0.2cm]
& \leq C( C_a\ e^{-\sigma\delta/2\ep}
+e^{-\sigma\delta/\ep}\|\psi\|_{C^{0,\lambda}_{K}(\R^2)})\\
&\leq Ce^{-\sigma\delta/2\ep}(1+\|\psi\|_{X})
\end{align*}
additionally, we also have
\begin{align*}
\|\Upsilon(g_1+G(\psi_1))-\Upsilon(g_1+G(\psi_2)) \|_{X}&\leq
C \|G(\psi_1)-G(\psi_2)\|_{C^{0,\lambda}_{K}(\R^2)}\\
&\leq C e^{-\sigma\delta/\ep}\|\psi_1-\psi_2\|_{X}
\end{align*}
where in both inequalities we used that $\Upsilon$ is a linear and
bounded operator.

This means that the right hand side of
equation~\eqref{FixedPointPsiProblem} defines a contraction mapping
on $\Lambda$ into itself, provided that the number $A$ in definition
\eqref{SetLambda} is taken large enough and
$\|\phi\|_{C^{2,\lambda}_{\mu,\sigma}}\leq 1$. Hence applying Banach
fixed point theorem follows the existence of a unique solution
$\psi=\Psi(\phi)\in\Lambda$.

In addition, it is direct to check that
\begin{equation*}
\|\Psi(\phi_1)-\Psi(\phi_2)\|_{X}\leq C_a\
e^{-\sigma\delta/2\ep}\|\phi_1-\phi_2\|_{C^{2,\lambda}_{\mu,\sigma}(\R)}+C\
e^{-\sigma\delta/\ep} \|\Psi(\phi_1)-\Psi(\phi_2)\|_{X}
\end{equation*}
from where the Lipschitz dependence \eqref{LipschitzPsi} of
$\Psi(\phi)$  follows and this concludes the proof of Lemma
\ref{LemmaNonlinearPsiEq}

\medskip
\subsection{The proof of proposition
\ref{PropNonlinearTheory}.} The purpose of the whole section is to
give a proof of proposition \ref{PropNonlinearTheory}, which deals
with the solvability of the nonlinear projected problem
\eqref{ProjectedProblem} for $\phi$.

\medskip
At the core of the proof of proposition \ref{PropNonlinearTheory},
is the fact that the heteroclinic solution $w(t)$ of the ODE
$$
w''(t) +w(t)(1 -w^2(t))=0, \quad w'(t)>0, \quad w(\pm\infty)=\pm1
$$
is $L^{\infty}-$nondegenerate in the sense of the following lemma.

\begin{lemma}{\label{LemmaBoundedKernelLinearizedOp}
Let $\phi$ be a bounded and smooth solution of the problem
\begin{equation*}
L(\phi)=0, \;\text{ in }\;\R^2
\end{equation*}
Then necessarily $\phi(s,t)=Cw'(t)$, with $C\in\R$. }
\end{lemma}

For a detailed proof of this lemma we refer the reader to \cite{9}
and references there in.

\noindent Next, let us consider the linear projected problem
\begin{equation}\label{LinearProjectedProblem}
\left\{
\begin{aligned}
&\partial_{tt}\phi+\partial_{ss}\phi+f'(w)\phi=g(s,t)+c(s)w'(t) \;\text{ in }\; \R\times\R\\[0.2cm]
&\int_{\R}\phi(s,t)w'(t)dt=0, \quad\forall s\in\R
\end{aligned}\right.
\end{equation}
\noindent Assuming that the corresponding operations can be carried
out, for every fixed $s$, we can multiply the equation by $w'(t)$
and integrate by parts, to find that

\begin{align}
c(s)=-\dfrac{\int_{\R}g(s,t)w'(t)dt}{\int_{\R}|w'(t)|^2dt}\label{ProjectionC}
\end{align}
Hence, if $\phi$ solves problem \eqref{LinearProjectedProblem}, then
$\phi$ eliminates the part of $g$ which does not contribute to the
projection onto $w'(t)$. This means, that $\phi$ solves the same
equation, but with $g$ replaced by $\tilde{g}$, where

\begin{equation*}
\tilde{g}(s,t)=g(s,t)-\dfrac{\int_{\R}g(s,\tau) w'(\tau)d\tau}{\int_{\R}|w'(\tau)|^2d\tau}w'(t)
\end{equation*}

\medskip
Observe that the term $c(s)$ in problem \eqref{ProjectedProblem} has
a similar role, except that we cannot find it so explicitly, since
this time the PDE in $\phi$ is nonlinear and nonlocal.

\medskip
Now, we show that the linear problem \eqref{LinearProjectedProblem}
has a unique solution $\phi$, which respects the size of $g$ in norm
\eqref{DefNorm10}, up to its second derivatives. We collect the
discussion in the following proposition, whose proof is basically
that contained in \cite{48}, \cite{9}.

\begin{prop}{\label{PropLinearTheory}
Given $\mu\geq 0$ and $0<\sigma<\sqrt{2}$, there is a constant $C>0$
such that for all sufficiently small $\ep>0$ the following holds.
For any $g$ with $\|g\|_{C^{0,\lambda}_{\mu,\sigma}(\R^2)}<\infty$,
the problem \eqref{LinearProjectedProblem} with $c(s)$ defined in
\eqref{ProjectionC}, has a unique solution $\phi$ with
$\|\phi\|_{C^{2,\lambda}_{\mu,\sigma}(\R^2)}<\infty$. Furthermore,
this solution satisfies the estimate
\begin{equation*}
\|\phi\|_{C^{2,\lambda}_{\mu,\sigma}(\R^2)}\leq
C\|g\|_{C^{0,\lambda}_{\mu,\sigma}(\R^2)}
\end{equation*}
}
\end{prop}

\medskip
Now, we are in a position to proof proposition
\ref{PropNonlinearTheory}. Recall from section 4, that
proposition~\ref{PropNonlinearTheory} refers to the solvability of
the projected problem
\begin{equation}
\left\{
\begin{aligned}
&\partial_{tt}\phi+\partial_{ss}\phi+f'(w)\phi=-\tilde{S}(u_1)-\N(\phi)+c(s)w'(t)\quad \text{ in }\quad\R\times\R\\[0.2cm]
&\int_{\R}\phi(s,t)w'(t)dt=0, \quad \text{ for all }s\in \R.
\end{aligned}\label{NonlinearProjectedProblem2}
\right.
\end{equation}
where we recall that $\N(\phi)$ is made out of the operators $H(t)$ and $\B$, given in \eqref{DefN1}-\eqref{ProofThmEq9}, like 
\begin{multline}
\N(\phi):=\underbrace{\B(\phi)+[f'(u_1)-f'(w)]\phi +
\ep\nabla_{\bar{x}} a/a\cdot \nabla_x\phi}_{\N_1(\phi)}+\\[-0.3cm]
\underbrace{\zeta_4[f'(u_1)-f'(H(t))]\Psi(\phi)}_{\N_2(\phi)}+
\underbrace{\zeta_4N_1(\Psi(\phi)+\phi)}_{\N_3(\phi)}\label{OperatorN2}
\end{multline}

Let us define $\phi:=T(g)$ as the operator providing the solution
predicted in proposition~\ref{PropLinearTheory}. Then
\eqref{NonlinearProjectedProblem2} can be recast as the fixed point
problem
\begin{align}
\phi=T(-\tilde{S}(u_1)-\ep^2\J_{a,\Gamma}[h]\,w'(t)-\N(\phi))=:\T(\phi),\quad
\| \phi \|_{C^{2,\lambda}_{\mu,\sigma}(\R^2)}\leq K\ep^4
\label{ProofPropNPPEq1}
\end{align}

%%%------ BEGIN CLAIM  --------------------------------------------------------------------------------------------------------------------
\begin{claim}{\label{LipschitzCharacterN}
Given $\alpha>0$, $0<\mu<2+\alpha$ and $0<\sigma<\sqrt{2}$, there is
some constant $C>0$, possibly depending on the constant $\K$ of
\eqref{DefNorm4} but independent of $\ep$, such that for $M>0$ and
$\phi_1,\phi_2$ satisfying
\begin{align*}
\|\phi_i\|_{C^{2,\lambda}_{\mu,\sigma}(\R^2)}\leq M\ep^4,\quad
i=1,2.
\end{align*}
then the nonlinearity $\N$ behaves locally Lipschitz, as
\begin{align}
\|\N(\phi_1)-\N(\phi_2)\|_{C^{0,\lambda}_{\mu,\sigma}(\R^2)}\leq
C\ep\ \|\phi_1-\phi_2\|_{C^{2,\lambda}_{\mu,\sigma}(\R^2)}
\label{ClaimLipsNEq1}
\end{align}
where the operator $\N$ is given in \eqref{OperatorN2}. }
\end{claim}

To prove this claim, we analyze each of its components $\N_i$ from
\eqref{OperatorN2}. Let us start with $\N_1$. Note that its first
term corresponds to a second order linear operator with coefficients
of order $\ep$ plus a decay of order at least $\mathit{O}((1+|\ep
s|)^{-1-\alpha})$. In particular, recall from \eqref{ProofThmEq9}
that $\B=\zeta_0\tilde{\B}_0$, where in coordinates
$[\Delta_{x}-\partial_{tt}-\partial_{ss}]$ amounts
\begin{align*}
\tilde{B}_0=&-2\ep h\partial_{st}-\ep[k(\ep s)+\ep(t+h)k^2(\ep s)]
\partial_t+\ep (t+h)A_0(\ep s,\ep(t+h))\\[0.1cm]
&\cdot[\partial_{ss}-2h'\partial_t+\ep^2|h'|^2\partial_{tt}]+\ep^2(t+h)B_0(\ep
s,\ep(t+h))[\partial_s-\ep h'\partial_t]\\[0.1cm]
&-\ep^2h''\partial_t +\ep^2|h'|^2\partial_{tt}\\[0.1cm]
&+\ep^3(t+h)^2C_0(\ep s, \ep(t+h))\partial_t
\end{align*}

Analyzing each term, leads to
$$
\|\B(\phi)\|_{C^{0,\lambda}_{\mu,\sigma}(\R^2)}\leq
C\ep  \|\phi\|_{C^{2,\lambda}_{\mu,\sigma}(\R^2)}
$$

Thus, $\N_1$ satisfies
\begin{align}
\|\N_1(\phi)\|_{C^{0,\lambda}_{\mu,\sigma}(\R^2)}\leq C\ep
\|\phi\|_{C^{2,\lambda}_{\mu,\sigma}(\R^2)}.\label{LipschitzCharacterNEq1}
\end{align}
\noindent On the other hand, consider functions $\phi_i$, with
$$ \|\phi_i\|_{C^{2,\lambda}_{\mu,\sigma}(\R^2)}\leq M\ep^3,\quad i=1,2$$

\medskip
Now, let us analyze $\N_2$, by noting that for any $(s,t)\in \R^2$
the definition \eqref{ImproveAproximation} implies that
\begin{align}
\|\N_3(\phi_1)-\N_3(\phi_2)& \|_{C^{0,\lambda}_{\mu,\sigma}(\R^2)}\notag\\
&\leq C\sup_{(s,t)\in\R^2} e^{(\sigma/2-\sqrt{2})|t|}
\sup_{x\in\R^2}K(x)\|\Psi(\phi_1)-\Psi(\phi_2)\|_{C^{0,\lambda}(B_1(x))}\notag\\[0.2cm]
&\leq C\|\Psi(\phi_1)-\Psi(\phi_2)\|_{C^{0,\lambda}_{K}(\R^2)}
=Ce^{-\sigma\delta/2\ep}\|\phi_1-\phi_2\|_{C^{2,\lambda}_{\mu,\sigma}(\R^2)}\label{LipschitzCharacterNEq3}
\end{align}

\noindent In order to analyze $\N_4$, note that the definition
\eqref{DefN1} of $N_1$ also implies
\begin{align*}
&|\N_4(\phi_1)-\N_4(\phi_2)|\leq |\zeta_4N_1(\Psi(\phi_1)+\phi_1)-\zeta_4 N_1(\Psi(\phi_2)+\phi_2)|\\[0.2cm]
&\quad \leq C\zeta_4\sup_{\xi\in(0,1)}|\xi(\Psi(\phi_1)+\phi_1)+(1-
\xi)(\Psi(\phi_2)+\phi_2)|\cdot(|\phi_1-\phi_2|+|\Psi(\psi_1)-\Psi(\psi_2)|)
\end{align*}
taking into account the region of $\R^2$ we are considering, it is
possible to make appear de weight $K(x)$ in~\eqref{WeightFunction}.
Therefore thanks to the hypothesis on $\phi_i$ and Lemma
\ref{LemmaNonlinearPsiEq}, we obtain
\begin{align}
& \|\N_4(\phi_1)-\N_4(\phi_2)\|_{C^{0,\lambda}_{\mu,\sigma}(\R^2)}\notag\\
&\quad \leq
C\sup_{(s,t)\in\R^2}\biggl\{e^{\sigma|t|/2}[\|\phi_1\|_{C^{0,\lambda}(B_1(s,t))}+\|\phi_2\|_{C^{0,\lambda}(B_1(s,t))}
+\|\Psi(\phi_1)\|_{C^{0,\lambda}(B_1(x))}+\|\Psi(\phi_2)\|_{C^{0,\lambda}(B_1(x))}] \biggr.\notag\\[0.1cm]
&\qquad \qquad \cdot \biggl. e^{\sigma|t|/2}(1+|\ep s|)^{\mu}
[\|\phi_1-\phi_2\|_{C^{0,\lambda}(B_1(s,t))}+\|\Psi(\phi_1)-\Psi(\phi_2)\|_{C^{0,\lambda}(B_1(x))}]\biggr\}\notag\\[0.1cm]
&\leq
C\sup_{(s,t)\in\R^2}\biggl\{\left[\|\phi_1\|_{C^{2,\lambda}_{\mu,\sigma}(\R^2)}
+\|\phi_2\|_{C^{2,\lambda}_{\mu,\sigma}(\R^2)}+K(x)
(\|\Psi(\phi_1)\|_{C^{0,\lambda}(B_1(x))}+\|\Psi(\phi_2)\|_{C^{0,\lambda}(B_1(x))})\right]\biggr. \notag\\
& \qquad\qquad \biggl.\cdot
(e^{-\sigma|t|/2}\|\phi_1-\phi_2\|_{C^{0,\lambda}_{\mu,\sigma}(\R^2)}+K(x)\|\Psi(\phi_1)-\Psi(\phi_2)\|_{C^{0,\lambda}(B_1(x))})
\biggr\}
\notag\\[0.1cm]
&\leq
C(\|\phi_1\|_{C^{2,\lambda}_{\mu,\sigma}(\R^2)}+\|\phi_2\|_{C^{2,\lambda}_{\mu,\sigma}(\R^2)}
+\|\Psi(\phi_1)\|_{X}+\|\Psi(\phi_2)\|_{X})[\|\phi_1-\phi_2\|_{C^{0,\lambda}_{\mu,\sigma}}+\|\Psi(\phi_1)-\Psi(\phi_1)\|_{X}]\notag\\[0.1cm]
&\leq
2C(\ep^3+e^{-\sigma\delta/2\ep})\left\{\|\phi_1-\phi_2\|_{C^{2,\lambda}_{\mu,\sigma}(\R^2)}+
e^{-\sigma\delta/2\ep}\|\phi_1-\phi_2\|_{C^{2,\lambda}_{\mu,\sigma}(\R^2)}\right\}
\label{LipschitzCharacterNEq4}
\end{align}
%\begin{align}
%& \|\N_4(\phi_1)-\N_4(\phi_2)\|_{C^{0,\lambda}_{\mu,\sigma}(\R^2)}\leq C(\|\phi_1\|_{C^{2,\lambda}_{0,0}(\R^2)}+\|\phi_2\|_{C^{2,\lambda}_{0,0}(\R^2)}+\|\Psi(\phi_1)\|_{X}+\|\Psi(\phi_2)\|_{X}) \notag\\[0.2cm]
%&\qquad \qquad \qquad \qquad \cdot(\|\phi_1-\phi_2\|_{C^{2,\lambda}_{\mu,\sigma}(\R^2)}+\|\Psi(\phi_1)-\Psi(\phi_2)\|_{X})\notag\\
%&\leq C(\ep^3+e^{-\sigma\tilde{\delta}/\ep})
%(1+e^{-\sigma\tilde{\delta}/\ep})\|\phi_1-\phi_2\|_{C^{2,\lambda}_{\mu,\sigma}(\R^2)}
%\leq C(\ep^3+e^{-\sigma\tilde{\delta}/\ep})\|\phi_1-\phi_2\|_{C^{2,\lambda}_{\mu,\sigma}(\R^2)}\label{LipschitzCharacterNEq4}
%\end{align}
To reach a conclusion,  we note from
\eqref{LipschitzCharacterNEq1}-\eqref{LipschitzCharacterNEq3} and
\eqref{LipschitzCharacterNEq4} that choosing $\ep>0$ small enough we
obtain the validity of inequality \eqref{ClaimLipsNEq1}. The proof
of Claim~\ref{LipschitzCharacterN} is concluded.\hfill$\square$

\medskip
In order to conclude the proof of proposition\ref{PropNonlinearTheory},we make the observation that
the formula \eqref{Su1tilde} and estimate \eqref{R1Size} ensure
that, for any $0<\mu\leq 2+\alpha$, $\sigma\in(0,\sqrt{2})$ and
$\lambda\in(0,1)$ it holds
\begin{multline}
\left\|\tilde{S}(u_1)+\ep^2\J_{a}[h](\ep s)\cdot
w'(t)-\ep^3\left[k^3(\ep
s)+\frac{1}{2}\partial_{tt}\left(\frac{\partial_{t}a}{a}\right) (\ep
s,0)\right]\hat{c}w'(t)\,
\right\|_{C^{0,\lambda}_{\mu,\sigma}(\R^2)} \leq C\ep^4
\label{ConclusionProp2Eq1}
\end{multline}

Let us assume now that $\phi_1,\phi_2\in B_{\ep}$, where
$$ B_{\ep}:=\{\phi\in C^{2,\lambda}_{loc}(\R^2)\ /\
\|\phi\|_{C^{2,\lambda}_{\mu,\sigma}(\R^2)}\leq K \ep^4 \}$$ for a
constant $K$ to be chosen. Note that using
Claim~\ref{LipschitzCharacterN}, we are able to estimate the size of
$\N(\phi)$ for any $\ep>0$ sufficiently small, as follows
\begin{align}
\|\N(\phi)\|_{C^{0,\lambda}_{\mu,\sigma}(\R^2)}&\leq
C\|\N(0)\|_{C^{0,\lambda}_{\mu,\sigma}(\R^2)}
+C\ep \|\phi\|_{C^{2,\lambda}_{\mu,\sigma}(\R^2)}\notag\\[0.1cm]
&=C\|\zeta_4[f'(u_1)-f'(H)]\Psi(0)+\zeta_4
N_1(\Psi(0))\|_{C^{0,\lambda}_{\mu,\sigma}(\R^2)}
+C\ep \|\phi\|_{C^{2,\lambda}_{\mu,\sigma}(\R^2)}\notag\\[0.1cm]
&\leq C\sup_{t\in\R}e^{-\sigma\delta/2\ep}\cdot \|\Psi(0)\|_{X}+\|\Psi(0)\|^2_{X}+C\ep\cdot K\ep^4\notag\\[0.1cm]
&\leq \tilde{C}\ep^5\quad\forall\,\phi\in
B_{\ep}\label{ConclusionProp2Eq2}
\end{align}
for some constant $\tilde{C}$, independent of $K$.

Then from the estimates
\eqref{ConclusionProp2Eq1}-\eqref{ConclusionProp2Eq2} follows that
the right hand side of the projected problem
\eqref{NonlinearProjectedProblem2} defines an operator $\T$ applying
the ball $B_{\ep}$ into itself, provided $K$ is fixed sufficiently
large and independent of $\ep>0$. Indeed using the definition of
$\T$ from \eqref{ProofPropNPPEq1}, and
Proposition~\ref{PropLinearTheory}, we can easily find an estimate
for the size of $\phi$, through
\begin{align*} \|\T(\phi)\|_{C^{2,\lambda}_{\mu,\sigma}(\R^2)} &=\|T(-\tilde{S}(u_1)-\ep^2\J_{a}[h]w'-\N(\phi))\|_{C^{2,\lambda}_{\mu,\sigma}(\R^2)}\\[0.2cm]
&\leq\|T\|
(\|\tilde{S}(u_1)+\ep^2\J_{a}[h]w'\|_{C^{0,\lambda}_{\mu,\sigma}(\R^2)}+\|\N(\phi)\|_{C^{0,\lambda}_{\mu,\sigma}(\R^2)})\leq
C\ep^4
\end{align*}
Further, $\T$ is also a contraction mapping of $B_{\ep}$ in norm
$C^{2,\lambda}_{\mu,\sigma}$ provided that $\mu\leq 2+\alpha$, since
Claim \eqref{LipschitzCharacterN} asserts that $\N$ has Lipschitz
dependence in $\phi$:
\begin{align*} \|\T(\phi_1)-\T(\phi_2)\|_{C^{2,\lambda}_{\mu,\sigma}(\R^2)}&=\|-T(\N(\phi_1)-\N(\phi_2))\|_{C^{2,\lambda}_{\mu,\sigma}(\R^2)}\\[0.2cm]
& \leq C
\|\N(\phi_1)-\N(\phi_2)\|_{C^{0,\lambda}_{\mu,\sigma}(\R^2)}\leq
C\ep\ \|\phi_1-\phi_2\|_{C^{2,\lambda}_{\mu,\sigma}(\R^2)}
\end{align*}
So by taking $\ep>0$ small, we can use the contraction mapping
principle to deduce the existence of a unique fixed point $\phi$ to
equation~\eqref{ProofPropNPPEq1}, and thus $\phi$ turns out to be
the only solution of problem \eqref{NonlinearProjectedProblem2}.
This justify the existence of $\phi$, as required.

On the other hand, the Lipschitz dependence \eqref{LipschitzPhi} of
$\Phi$ in $h$, follows from the fact that
\begin{multline*}
\|\T(\Phi(h_1)) - \T(\Phi(h_2))\|_{C^{2,\lambda}_{\mu,\sigma}(\R^2)}\leq C
(\|\tilde{S}(u_1,h_1)-\tilde{S}(u_1,h_2)\|_{C^{0,\lambda}_{\mu,\sigma}(\R^2)}
+\|\N(\Phi_1)-\N(\Phi_1)\|_{C^{0,\lambda}_{\mu,\sigma}(\R^2)})
\end{multline*}
A series of lengthy but straightforward computations, leads to
\eqref{LipschitzPhi} and so the proof is complete.

\section{The proof of proposition~\ref{PropNonlinearJacobiTheory}}

In this section, we will finish the proof of Theorem
\ref{TheoremMainResult} by proving
proposition~\ref{PropNonlinearJacobiTheory}. Recall that the reduced
problem \eqref{JacobiEquation} reads as
\begin{align}
\J_{a}[h](\ep s):=h''(\ep s)+\frac{\partial_{\bs{s}}a(\ep
s,0)}{a(\ep s,0)}h'(\ep s)-Q(\ep s)h(\ep s)\,=\, \G(h)(\ep s) \quad
\hbox{in }\R \label{ReducedProblemEq1}
\end{align}
where $Q(\bs{s})$ was defined in \eqref{OperatorQ}, and the operator
$\G=G_1+G_2$ was given in
\eqref{ProofThmEq14}-\eqref{ProofThmEq15}.

We will make use of the following technical lemma, whose proof is
left to the reader.

\begin{lemma}{\label{LemmaSizeProjection}
Let $\Theta=\Theta(s,t)$ be a function defined in $\R\times\R$, such
that, for any $\lambda\in(0,1)$, $\mu\in(1,2+\alpha]$ and
$\sigma\in(0,\sqrt{2})$
\begin{align*}
\|\Theta\|_{C^{0,\lambda}_{\mu,\sigma}(\R^2)}:=\sup_{(s,t)\in\R\times\R}e^{\sigma|t|}(1+|\ep
s|)^{\mu}\|\Theta\|_{C^{0,\lambda}(B_1(s,t))}<+\infty
\end{align*}
Then the function defined in $\R$ as
\begin{align*}
Z(\ep s):= \int_{\R} \Theta(s,t)w'(t)dt
\end{align*}
satisfies for some constant $C=C(w,\mu,\sigma)>0$ the following
estimate:
\begin{align}
\|Z\|_{C^{0,\lambda}_{\mu,*}(\R)}\leq
C\ep^{-1}\|\Theta\|_{C^{0,\lambda}_{\mu,\sigma}(\R^2)}\label{LemmaSizeProjectionEq1}
\end{align}
}\end{lemma}

Let us apply Lemma~\ref{LemmaSizeProjection} to the function
$\Theta(s,t):=\N(\Phi(h))(s,t)$, to estimate the size of the
operator $G_2$ in \eqref{ProofThmEq15}, where we recall that
$$
G_2(h)(\ep s):=c^{-1}_{*}\ep^{-2}\int_{\R}\N(\Phi(h))(s,t)w'(t)dt
$$

We can estimate the size of the projection of $\N$ using the
previous estimate \eqref{LemmaSizeProjectionEq1}, and the bound
\eqref{ConclusionProp2Eq2} for the size of $\N$:
\begin{align}
\|G_2(h)\|_{C^{0,\lambda}_{\mu,*}(\R)}\leq
C\ep^{-3}\|\N(\Phi(h))\|_{C^{0,\lambda}_{\mu,\sigma}(\R^2)}\leq
C\ep^{2}\label{ReducedProblemEq3}
\end{align}
Likewise, for $\phi_i=\Phi(h_i)$, $i=1,2$  it holds similarly that
$$
\|G_2(h_1)-G_2(h_2)\|_{C^{0,\lambda}_{\mu,*}(\R)}\leq C\ep^{-3}
\|\N(\phi_1)-\N(\phi_2)\|_{C^{0,\lambda}_{\mu,\sigma}(\R^2)}
$$

Nonetheless, using \eqref{ClaimLipsNEq1} and proposition
\ref{PropNonlinearTheory}, it follows that
$$
\|\N(\phi_1)-\N(\phi_1)\|_{C^{0,\lambda}_{\mu,\sigma}(\R^2)} \leq
 C\ep^4\|h_1-h_2\|_{C^{2,\lambda}_{\mu,*}(\R)}.
$$

The previous estimates allow us to deduce
\begin{align*}
\|G_2(h_1)-G_2(h_2)\|_{C^{0,\lambda}_{\mu,\sigma}(\R^2)}&\leq C \ep
\|h_1-h_2\|_{C^{2,\lambda}_{\mu,*}(\R)}
\end{align*}
Furthermore, from \eqref{ReducedProblemEq3} we also have that
\begin{align}
\|G_2(0)\|_{C^{0,\lambda}_{\mu,*}(\R)}\leq
C\ep^2\label{ReducedProblemEq4}
\end{align}
for some $C>0$ possibly depending on $\K$.
\medskip

Next, we consider
\begin{align*}
c_{*}G_1(h_1)&= -\ep\left[k^3(\bs{
s})+\frac{1}{2}\partial_{tt}\left(\frac{\partial_{t}a}{a}\right)
(\bs{
s},0)\right]c^*\hat{c}w'(t)\\
&+\ep h''_1(\bs{s})\int_{\R}\zeta_0(t+h_1)A_0(\bs{s},
\ep(t+h_1))w''(t)w'(t)dt\notag\\[0.2cm]
&+\ep^2Q(\bs{s})h''_1(\bs{s})
\int_{\R}\psi'_0(t)w'(t)dt+\ep^{-2}\int_{\R}\zeta_0\
R_1(\bs{s},t,h_1,h_1')w'(t)dt
\end{align*}

It is direct to check, from \eqref{ErrorA0LaplacianXeph} and
\eqref{DefinitionR1} the following estimate on the Lipschitz
character for $G_1(h)$
$$
\|G_1(h_1)-G_1(h_2)\|_{C^{0,\lambda}_{\mu,*}(\R)}\leq C
\ep\|h_1-h_2\|_{C^{2,\lambda}_{\mu,*}(\R)}.
$$

Now, a simple but crucial observation we make is that
\begin{align*}
c_{*}G_1(0)&=\ep^{-2}\int_{\R}\zeta_0\ R_1(\ep s,t,0,0)w'(t)dt
\end{align*}
has the size
\begin{align}
\|G_1(0)\|_{C^{0,\lambda}_{\mu,*}(\R)}\leq
C\ep^{-2}\cdot\ep^{-1}\|R_1\|_{C^{0,\lambda}_{\mu,\sigma}(\R)}\leq
C_2\ep\label{ReducedProblemEq5}
\end{align}
for some universal constant $C_2>0$. Therefore, the entire operator
$\G(h)$ inherits a Lipschitz character in $h$, from those of $G_1$,
$G_2$:
\begin{align}
\|\G(h_1)-\G(h_2)\|_{C^{0,\lambda}_{\mu,*}(\R)}\leq C
\ep\|h_1-h_2\|_{C^{2,\lambda}_{\mu,*}(\R)}.\label{LipschitzG}
\end{align}

Moreover, estimates
\eqref{ReducedProblemEq4}-\eqref{ReducedProblemEq5} imply that $\G$
is such
\begin{align}
\|\G(0)\|_{C^{2,\lambda}_{\mu,*}(\R)}\leq
2C_2\ep\label{ReducedProblemEq6}
\end{align}
Now let $h=T(f)$ be the linear operator defined in
Proposition~\ref{PropLinearJacobiTheory}, and let $\G$ be the
nonlinear operator given in \eqref{ProofThmEq15}. Consider the
Jacobi nonlinear equation~\eqref{ReducedProblemEq1}, but this time
written as a fixed point problem: Find some $h$ such that
\begin{align}
h=T(\G(h)), \quad \|h\|_{C^{2,\lambda}_{2+\alpha,*}(\R)}\leq
\K\ep\label{FixedPointProblem}
\end{align}

Observe that
\begin{align*}
\|T(\G(h))\|_{C^{2,\lambda}_{2+\alpha,*}(\R)}
& \leq C\left(\|\G(h)-\G(0)\|_{C^{0,\lambda}_{2+\alpha,*}(\R)}
+\|\G(0)\|_{C^{0,\lambda}_{2+\alpha,*}(\R)}\right)\\[0.2cm]
& \leq  C\,\ep\left(1+\|h\|_{C^{2,\lambda}_{2+\alpha,*}(\R)}\right)
\end{align*}

where we made use of \eqref{LipschitzG}-\eqref{ReducedProblemEq6}.
Observe also that
$$
\|T(\G(h_1))-T(\G(h_2))\|_{C^{2,\lambda}_{2+\alpha,*}(\R)}\leq
C\|\G(h_1)-\G(h_2)\|_{C^{0,\lambda}_{2+\alpha,*}(\R)}\leq
C\ep\|h_1-h_2\|_{C^{2,\lambda}_{2+\alpha,*}(\R)}
$$

Hence choosing $\K>0$, large enough but independent of $\ep>0$, we
find that if $\ep$ is small, the operator $T\circ\G$ is a
contraction on the ball $\|h\|_{C^{2,\lambda}_{2+\alpha,*}(\R)} \leq
\K\ep$. As a consequence of the Banach's fixed point theorem, obtain
the existence of a unique fixed point of the problem
\eqref{FixedPointProblem}. This finishes the proof of
Proposition~\ref{PropNonlinearJacobiTheory} and consequently, the
proof of our theorem.

%%%%%%%%%%%%%%%%%%%%%%%%%%%%%%%%%%%%%%%%%%%%%%%%%%%%%%%%%%%%%%
%%%%%%%%%%%%%%%%%%%%%%%%%%%%%%%%%%%%%%%%%%%%%%%%%%%%%%%%%%%%%%
%%%%%%%%%%%%%%%%%%%%%%%%%%%%%%%%%%%%%%%%%%%%%%%%%%%%%%%%%%%%%%%
\section{Appendix}

This section is mainly oriented in finding expressions for each one
of the terms in the Allen-Cahn equation~\eqref{IntrodEq12} written
in Fermi coordinates, that are suitable for the geometrical study of
this equation. Since they define a local change of variables in a
neighborhood of $\Gamma$, our effort will focus on finding an
equivalent form of~\eqref{IntrodEq12} in these coordinates.

\subsection{Laplacian in Fermi Coordinates: Proof of
Lemma~\ref{LemmaEuclideanLaplacian}}\label{sec:LaplacianoCoordFermiDilatadasTrasladadas}

In order to characterize the Euclidean Laplacian $\Delta_{x,y}$ in
dilated and translated Fermi coordinates, we will follow the scheme
from the appendix in \cite{9}.

For any $\delta>0$ small but fixed, and a curve $\Gamma\subset\R^2$
parameterized by $\gamma\in C^{2}(\R,\R^2)$, let us consider first
the local \textsf{Fermi coordinates} induced by $\Gamma$
\begin{equation*}
X:\R\times (-\delta,\delta)\to \Nn_{\delta}\ ,\quad
X(\bs{s},\bs{t})=\gamma(\bs{s})+\bs{t}\nu(\bs{s})
\end{equation*}
where $\nu(\bs{s})$ denotes the normal vector to the curve $\Gamma$
at the point $\gamma(\bs{s})$.

It can be seen that $X$ defines a local change of variables on the
tubular open neighborhood
$$\Nn_{\delta}:=\{(\bar{x},\bar{y})=\gamma(\bs{s})+\bs{t}\nu(\bs{s})\ /\; \bs{s}\in \R, \;|\boldsymbol{t}|< \delta+\ep\cdot 2c_0|\bs{s}|\}$$
of $\Gamma$, where $c_0>0$ is a fixed number, and
$|\bs{t}|=\dist((\bar{x},\bar{y}),\Gamma)$ for every
$(\bar{x},\bar{y})=X(\bs{s},\bs{t})$. 

Given that $X(\Nn_{\delta})\subset \R^2$ is a 2-dimensional
manifold, we can employ a formula from Differential Geometry that
allow us to compute the Euclidean Laplacian in terms of Fermi
coordinates, for points $(\bar{x},\bar{y})=X(\bs{s},\bs{t})\in
\Nn_{\delta}$ as follows
\begin{equation}
\Delta_{X}=\dfrac{1}{\sqrt{\det(g(\bs{s},\bs{t}))}}\;
\partial_{i}\left(\sqrt{\det(g(\bs{s},\bs{t}))}\cdot
g^{ij}(\bs{s},\bs{t})\ \partial_{j}\right),\quad i,j=\bs{s},\bs{t}
\label{LaplacianoCoordFermiEq1}
\end{equation}
where $g_{ij}(\bs{s},\bs{t})=\
\pdi{\partial_{\bs{i}}X(\bs{s},\bs{t})}{\partial_{\bs{j}}X(\bs{s},\bs{t})}$
corresponds to the $ij$th entry of metric $g$ of $\Gamma$, and we
regard $g^{ij}=(g^{-1})_{i,j}$ as the respective entry for the
inverse of the metric.
%\begin{equation}
%\; g(\bar{s},\bar{t})=\begin{bmatrix} g_{11}(\bar{s},\bar{t})& g_{12}(\bar{s},\bar{t})
%\\ g_{21}(\bar{s},\bar{t})& g_{22}(\bar{s},\bar{t})\end{bmatrix}=
%\begin{bmatrix} \norm{\partial_{\bar{s}}X(\bar{s},\bar{t})}^2& \pdi{\partial_{\bar{s}}X(\bar{s},\bar{t})}{\partial_{\bar{t}}X(\bar{s},\bar{t})} \\
%\pdi{\partial_{\bar{t}}X(\bar{s},\bar{t})}{\partial_{\bar{s}}X(\bar{s},\bar{t})}  & \norm{\partial_{\bar{s}}X(\bar{s},\bar{t})}^2\end{bmatrix}\qquad \label{LaplacianoCoordFermiEq2}
%\end{equation}
%\begin{equation}
%\Delta_{x}=\partial_{\bs{t}\bs{t}}+g^{ss}\cdot\partial_{\bs{s}\bs{s}}-\dfrac{1}{2}\dfrac{\partial_{\bs{s}}g_{11}}{g_{11}^2}\cdot \partial_{\bs{s}}+\dfrac{1}{2}\dfrac{\partial_{\bs{t}}g_{11}}{g_{11}}\cdot \partial_{\bs{t}} \label{LaplacianoCoordFermiEq7}
%\end{equation}
Performing explicit calculations, and using the relations between
the tangent and the normal to the curve $\Gamma$, it follows
\begin{equation}
\partial_{\bs{s}}X(\bs{s},\bs{t})=\dot{\gamma}(\bs{s})+\bs{t}\dot{\nu}(\bs{s}),\quad \partial_{\bs{t}}X(\bs{s},\bs{t})=\nu(\bs{s})\;\; \label{LaplacianoCoordFermiEq3}
\end{equation}
And so by \eqref{LaplacianoCoordFermiEq3}, the metric $g$ can be
computed as
\begin{align*}
g_{ss}(\bs{s},\bs{t})&=|\dot{\gamma}(\bs{s})|^2+2\bs{t}\dot{\gamma}(\bs{s})\cdot\dot{\nu}(\bs{s})+\bs{t}^2|\dot{\nu}(\bs{s})|^2
%= 1-2\bs{t}k(\bs{s})+\bs{t}^2k(\bs{s})^2
=(1-\bs{t}k(\bs{s}))^2\\
g_{st}(\bs{s},\bs{t})&=g_{ts}(\bs{s},\bs{t})= 0,\quad
g_{tt}(\bs{s},\bs{t})=1 
\end{align*}
Hence, the components of the $g^{-1}$ are
\begin{equation}
g^{ss}(\bs{s},\bs{t})=\dfrac{1}{(1-\bs{t}k(\bs{s}))^2},\;\;
g^{st}(\bs{s},\bs{t})= g^{ts}(\bs{s},\bs{t})=0,\;\;
g^{tt}(\bs{s},\bs{t})=1\label{LaplacianoCoordFermiEq6}
\end{equation}
Replacing formula \eqref{LaplacianoCoordFermiEq1} and using values
obtained in \eqref{LaplacianoCoordFermiEq6}, we get
\begin{equation}
\Delta_{x}=\partial_{\bs{t}\bs{t}}+g^{ss}\partial_{\bs{s}\bs{s}}+\dfrac{1}{\sqrt{\det
g}}
\partial_{\bs{t}}(\sqrt{\det g})\partial_{\bs{t}}
+\dfrac{1}{\sqrt{\det g}}\partial_{\bs{s}}(\sqrt{\det g}\cdot
g^{ss})\partial_{\bs{s}} \label{LaplacianoCoordFermiEq7}
\end{equation}
where $\sqrt{\det g}=1-\bs{t}k(\bs{s})$, and with
\begin{align*}
\frac{1}{\sqrt{\det g}}\partial_{\bs{t}}(\sqrt{\det
g})=\dfrac{-k(\bs{s})}{1-\bs{t}k(\bs{s})},\quad \frac{1}{\sqrt{\det
g}}\partial_{\bs{s}} (\sqrt{\det g}\cdot g^{ss})
=\frac{\bs{t}\dot{k}(\bs{s})}{(1-\bs{t}k(\bs{s}))^3}%\label{LaplacianoCoordFermiEq8.0}
\end{align*}

Since $k$ vanishes at infinity, we can make an approximation of this
operator at main order,
\begin{align*}
\dfrac{1}{(1-\bs{t}k)^2}=\underbrace{\left(\sum\limits_{m=0}^{\infty}(\bs{t}k)^{m}\right)^2}_{\bs{t}\
\text{small}}=
[1+\bs{t}k+\bs{t}^2\O(k^2)]^2=[ 1+2\bs{t}k+\bs{t}^2\O(k^2)]\notag\\
\dfrac{1}{(1-\bs{t}k)^3}=\underbrace{\left(\sum\limits_{m=0}^{\infty}(\bs{t}k)^{m}\right)^3}_{\bs{t}\
\text{small}}= [1+\bs{t}k+\bs{t}^2\O(k^2)]^3=[
1+3\bs{t}k+\bs{t}^2\O(k^2)]\notag.
\end{align*}

In this way, we deduce the expansion
\begin{align}
g^{ss}&=1+\bs{t} A_0(\bs{s},\bs{t})\label{LaplacianoCoordFermiEq8}\\
\frac{1}{\sqrt{\det g}}\partial_{\bs{t}}(\sqrt{\det g})&=-k(\bs{s})-\bs{t}k^2(\bs{s})+t^2C_0(\bs{s},\bs{t})\label{LaplacianoCoordFermiEq9}\\
\frac{1}{\sqrt{\det g}}\partial_{\bs{s}} (\sqrt{\det g}\cdot
g^{ss})&=\bs{t}B_0(\bs{s},\bs{t})\label{LaplacianoCoordFermiEq10}
\end{align}
where $A_0(\bs{s},\bs{t})$, $B_0(\bs{s},\bs{t})$ and
$C_0(\bs{s},\bs{t})$ are are the smooth functions described in
\eqref{ErrorA0LaplacianXeph}, \eqref{ErrorB0LaplacianXeph} and
\eqref{ErrorC0LaplacianXeph}.

So using relations \eqref{LaplacianoCoordFermiEq7} to
\eqref{LaplacianoCoordFermiEq10}, we get the Euclidean Laplacian in
Fermi coordinates
\begin{equation}
\Delta_{X}=
\partial_{\bs{t}\bs{t}}+\partial_{\bs{s}\bs{s}}-[k(\bs{s})+\bs{t}k^2(\bs{s})]\partial_{\bs{t}}
+\bs{t} A_0(\bs{s},\bs{t})\partial_{\bs{s}\bs{s}}
+\bs{t}B_0(\bs{s},\bs{t})\partial_{\bs{s}}
+\bs{t}^2C_0(\bs{s},\bs{t})\partial_{\bs{t}}\label{LaplacianoCoordFermiEq14}
\end{equation}
%\begin{align}
%\nabla_{\bs s,\bs t} A(\bs{s},\bs{t})&=\left(\mathit{O}(|\dot{k}(\bs{s})+2\bs{t}k(\bs{s})\cdot\dot{k}(\bs{s})|)\ ,\
%\mathit{O}(|k^2(\bs{s})|) \right)\notag\\
%\nabla_{\bs s,\bs t} B(\bs{s},\bs{t})&=\left(\mathit{O}(|\ddot{k}(\bs{s})\cdot k^2(\bs{s})+2 k(\bs{s})\cdot \dot{k}^2(\bs{s})|),\ \mathit{O}(1)\right) \notag\\
%\nabla_{\bs s,\bs t} C(\bs{s},\bs{t})&=\left(\mathit{O}(|k^2(\bs{s})\cdot \dot{k}(\bs{s})|),\ \mathit{O}(1)\right)\notag
%\end{align}

Next, we consider the dilated curve $\Gamma_{\ep}=\ep^{-1}\Gamma$
by $\gamma_{\ep}:s\mapsto \ep^{-1}\gamma(\ep s)$, and we define
associated local \textbf{dilated Fermi coordinates} in $\R^2$ by
\begin{equation*}
X_{\ep}(s,t):=\frac{1}{\ep}X(\ep s,\ep t)= \frac{1}{\ep}\gamma(\ep s)+t \nu(\ep s)
\end{equation*}
on a dilated tubular neighborhood $\ep^{-1}\Nn_{\delta}$ of the
curve $\Gamma_\ep$
\begin{equation*}
\Nn_{\ep}=\left\{(x,y)=X_{\ep}(s,t)\in \R^2/\; s\in \R, \;|t|<
\frac{\delta}{\ep}+2c_0|s|) \right\}
\end{equation*}
where $c_0>0$ is a fixed number, and in such way that $X_{\ep}$
defines a local change of variables. Consequently, scaling formula
\eqref{LaplacianoCoordFermiEq14} yields the expression
\begin{align*}
\Delta_{X_{\ep}}&:=\partial_{tt}+\partial_{ss}-\ep[k(\ep s)+\ep t k^2(\ep s)]\partial_{t}+\ \ep t A_0(\ep s ,\ep t)\partial_{ss}\\
&+\ \ep^2 t B_0(\ep s ,\ep t)\partial_{s}+\ \ep^3 t^2 C_0(\ep s ,\ep t)\partial_{t}
\end{align*}

Setting $z=t-h(\ep s)$, it is possible to compute $\Delta_{x,y}$ in
terms of dilated and translated Fermi coordinates $(s,t)$, as
\begin{align*}
\Delta_{X_{\ep,h}}&=
\partial_{tt}\,+\,\partial_{ss}-2\ep h'(\ep
s)\partial_{st}
-\ep^2 h''(\ep s)\partial_{t}-\ep[k(\ep s)+\ep(t+h(\ep s))k^2(\ep s)]\partial_{t}\\
&+\ep^2 |h'(\ep s)|^2\partial_{tt}+D_{\ep,h}(s,t)
\end{align*}
where
\begin{align*}
D_{\ep,h}(s,t)&:=
\ep(t+h(\ep s))A_0(\ep s,\ep(t+h))[\partial_{ss}v^{*}-2\ep h'(\ep s)\partial_{ts}v^{*}-\ep^2h''(\ep s)\partial_tv^{*}+\ep^2|h'(\ep s)|^2\partial_{tt}v^{*}]\\
&+\ep^2(t+h(\ep s))B_0(\ep s,\ep(t+h))[\partial_{s}v^{*}(s,t)-\ep h'(\ep s)\partial_tv^{*}(s,t)]\\
&+\ep^3(t+h(\ep s))^2C_0(\ep s,\ep(t+h))\partial_{t}v^{*}(s,t)
\end{align*}
thus finishing the proof of Lemma \ref{LemmaEuclideanLaplacian}
%\begin{align}
%&\Delta_{X_{\ep,h}}\tilde{u}(x,y)= \partial_{tt}v^{*}(s,t)+\partial_{ss}v^{*}(s,t)-2\ep h'(\ep s)\cdot\partial_{st}v^{*}(s,t)
%+\ep^2|h'(\ep s)|^2\cdot\partial_{tt}v^{*}(s,t)\notag\\[0.1cm]
%&-\ep^2 h''(\ep s)\cdot\partial_{t}v^{*}(s,t)-\ep[k(\ep s)+\ep(t+h)k^2(\ep s)]\cdot\partial_{t}v^{*}(s,t)
%+\ep(t+h)A_0(\ep s ,\ep [t+h])\cdot\notag\\[0.1cm]
%& (\partial_{ss}v^{*}(s,t)-2\ep h'(\ep s)\cdot\partial_{st}v^{*}(s,t)+\ep^2|h'(\ep s)|^2 \cdot \partial_{tt}v^{*}(s,t)
%-\ep^2h''(\ep s)\cdot\partial_{t}v^{*}(s,t))\notag\\[0.1cm]
%&+\ep^2(t+h)B_0(\ep s,\ep [t+h])\cdot[\partial_{s}v^{*}(s,t)-\partial_{t}v^{*}(s,t)\cdot(\ep h'(\ep s))]\notag\\[0.1cm]
%&+\ep^3(t+h)^2 C_0(\ep s,\ep[t+h])\cdot\partial_{t}v^{*}(s,t)\label{LaplacianoCoordFermiDilatadasTrasladadasEq6}
%\end{align}
 %Finally, from \eqref{LaplacianoCoordFermiDilatadasTrasladadasEq6} the laplacian in dilated and translated Euclidean coordinates can be written in terms of $(s,t)$ as\\[0.2cm]
%\end{enumerate}
%%-----------------------  END SUBSUBSECTION: Laplacian in dilated and translated Fermi coordinates -----------------------------------%%
%%-------------------------------------------------------------------------------------------------------------------------------------%%
%%%%%%%%%%%%%%%%%%%%%%%%%%%%%%%%%%%%%%%%%%%%%%%%%%%%%%%%%%%%%%%%%%%%%%%%%%%%%%%%%%%%%%%%%%%%%%%%%%%%%%%%%%%%%%%%%%%%%%%%%%%%%%%%%%%%%%%%%
%%-----------------------  END SECTION: Calculation of the laplacian in dilated and traslated Fermi coordinates  ------------------------

\subsection{Proof of
Lemma~\ref{LemmaEuclideanGradients}}\label{sec:GradienteCoordFermiDilatadasTrasladadas}

Analogously to what performed in the last section, our main goal is
finding a characterization for the product
$\ep\nabla_{\bar{x}}a/a\cdot\nabla_x u$ in Fermi coordinates. To
achieve this we will follow a scheme in 3 steps, for which the
analysis is simplified.

Hereinafter, once again we adopt the convention
$a=a(\bar{x},\bar{y})$ and $u=u(\bar{x},\bar{y})$, where
$(\bar{x},\bar{y})$ denotes the non-dilated Euclidean coordinates of
the space.

We had that $X(\bs{s},\bs{t})=\gamma(\bs{s})+\bs{t}\nu(\bs{s})$
provided a local change of variables, implying that 
\begin{align*}
(\bar{x},\bar{y})&=X(X^{-1}(\bar{x},\bar{y})), \quad \text{ for }(\bar{x},\bar{y})\in\Nn \\
\nabla_{\bar{x},\bar{y}}&=\nabla_{\bs{s},\bs{t}}\cdot
\left[D_{\bs{s},\bs{t}}X(X^{-1}(\bar{x},\bar{y}))\right]^{-1}
\end{align*}

Therefore, we obtain that
\begin{align*}
\frac{\nabla_{\bar{x},\bar{y}}a}{a}\nabla_{\bar{x},\bar{y}}
%\begin{pmatrix}\partial_{\bs{s}}a, \partial_{\bs{t}}a\end{pmatrix}\cdot
%\begin{bmatrix}\;\dfrac{1}{(1-\bs{t}k(\bs{s}))^2} & 0 \\[0.3cm] \; 0 & 1\;\end{bmatrix}\cdot
%\begin{pmatrix}\partial_{\bs{s}}v_u \\ \partial_{\bs{t}}v_u\end{pmatrix}\notag\\
&=\dfrac{1}{(1-\bs{t}k(\bs{s}))^2}\left(\frac{\partial_{\bs{s}}a}{a}\cdot\partial_{\bs{s}}v_u\right)+ \frac{\partial_{\bs{t}}a}{a}\cdot\partial_{\bs{t}}v_u\\[0.2cm]
&=\ \frac{\partial_{\bs{s}}a}{a}\cdot \partial_{\bs{s}}+
\frac{\partial_{\bs{t}}a}{a}\cdot\partial_{\bs{t}}+\bs{t}A_0(\bs{s},\bs{t})\frac{\partial_{\bs{s}}a}{a}\cdot
\partial_{\bs{s}}
\end{align*}
where we have made use of the expansion \eqref{LaplacianoCoordFermiEq8} of entry $g^{\bs{s}\bs{s}}$ of the metric.\\

Further, using the above expression in the Taylor expansion for $\nabla a/a$ around the curve $\Gamma$, we find 
\begin{align*}
\frac{\partial_{\bs{s}}a}{a}(\bs{s},\bs{t})&=\frac{\partial_{\bs{s}}a}{a}(\bs{s},0)+\bs{t}\ \partial_{\bs{t}}\left(\frac{\partial_{\bs{s}} a}{a}\right)(\bs{s},0)+\mathit{O}\left(\bs{t}^2\partial_{\bs{t}\bs{t}}\left(\frac{\partial_t a}{a}\right)\right)\\
\frac{\partial_{\bs{t}}a}{a}(\bs{s},\bs{t})&=
\frac{\partial_{\bs{t}}a}{a}(\bs{s},0)+\bs{t}\left(\frac{\partial_{\bs{t}\bs{t}}a}{a}(\bs{s},0)-
\biggl\lvert\frac{\partial_{\bs{t}}a}{a}(\bs{s},0)\biggr\rvert^2\right)+
\dfrac{\bs{t}^2}{2}\partial_{\bs{t}\bs{t}}\left(\frac{\partial_t
a}{a}(\bs{s},0)\right)
+\mathit{O}\left(\bs{t}^3\partial_{\bs{t}\bs{t}\bs{t}}\left(\frac{\partial_t
a}{a}\right)\right)
\end{align*}

Replacing these expansions in the above equation we obtain
\begin{align}
\dfrac{\nabla_{X}a}{a}\nabla_{X}
&=\frac{\partial_{\bs{s}}a}{a}(\bs{s},0)\partial_{\bs{s}}+
\left[\frac{\partial_{\bs{t}}a}{a}(\bs{s},0)+\bs{t}
\left(\frac{\partial_{\bs{tt}}a}{a}(\bs{s},0)-\biggl\lvert\frac{\partial_{\bs{t}}a}{a}(\bs{s},0)\biggr\rvert^2\right)\right]
\partial_{\bs{t}}\notag\\[0.2cm]
&+\bs{t}D_0(\bs{s},\bs{t})\partial_{\bs{s}}
+\bs{t}^2F_0(\bs{s},\bs{t})\partial_{\bs{t}}\label{GradienteCoordFermiEq10}
\end{align}
for which
\begin{align*}
D_0(\bs{s},\bs{t})&=\partial_{\bs{t}}\left(\frac{\partial_{\bs{s}}a}{a}\right)(\bs{s},0)
+\mathit{O}\left(\bs{t}\partial_{\bs{t}\bs{t}}\left(\frac{\partial_t a}{a}\right)\right)
+A_0(\bs{s},\bs{t})\frac{\partial_{\bs{s}}a}{a}(\bs{s},\bs{t})\\
F_0(\bs{s},\bs{t})&=\dfrac{1}{2}\partial_{\bs{t}\bs{t}}\left(\frac{\partial_t
a}{a}(\bs{s},0)\right)
+\mathit{O}\left(\bs{t}\partial_{\bs{t}\bs{t}\bs{t}}\left(\frac{\partial_t
a}{a}\right)\right)
\end{align*}
\noindent and $A_0(\bs{s},\bs{t})$ given by
\eqref{ErrorA0LaplacianXeph}.
%%-----------------------  END SUBSUBSECTION: Gradient in Fermi coordinates   --------------------------------------------------------%%

Let us recall the convention $(\bar{x},\bar{y}):=(\ep x, \ep y),
\;(\bs{s},\bs{t}):=(\ep s, \ep t)$ and observe that
\begin{equation*}
\nabla_{x,y}=\ep\nabla_{\bar{x},\bar{y}}
\end{equation*}
Directly from a this scaling and formula
\eqref{GradienteCoordFermiEq10} we find that
\begin{align*} \ep
\dfrac{\nabla_{X}a}{a}(\bar{x},\bar{y})\nabla_{X_{\ep}}&=
\ep^2\Biggl\{\frac{\partial_{\bs{s}}a}{a}(\bs{s},0)\partial_{\bs{s}}+
\left[\frac{\partial_{\bs{t}}a}{a}(\bs{s},0)+\bs{t}
\left(\frac{\partial_{\bs{tt}}a}{a}(\bs{s},0)-\biggl\lvert\frac{\partial_{\bs{t}}a}{a}(\bs{s},0)\biggr\rvert^2\right)\right]
\partial_{\bs{t}}\Biggr.\\
&+ \Biggl.\bs{t}D_0(\bs{s},\bs{t})\partial_{\bs{s}}+\bs{t}^2F_0(\bs{s},\bs{t}) \partial_{\bs{t}}\Biggr\}
\end{align*}
which can be written in terms of dilated Fermi coordinates $(s,t)$
\begin{align*}
\ep \dfrac{\nabla_{X}a}{a}(\bar{x},\bar{y})\nabla_{X_{\ep}}&=
\ep\frac{\partial_{\bs{s}}a}{a}(\ep s,0)\partial_{s}\\
&+\ep \left[\frac{\partial_{\bs{t}}a}{a}(\ep s,0)+\ep t
\left(\frac{\partial_{\bs{tt}}a}{a}(\ep s,0)
-\biggl\lvert\frac{\partial_{\bs{t}}a}{a}(\ep s,0)\biggr\rvert^2\right)\right]
\partial_{t}+E_{\ep}(s,t)
\end{align*}
where the derivatives of function $a$ are with respect to shrink
Fermi variables $(\bs{s},\bs{t})$, and with
\begin{align*}
E_{\ep}(s,t):=\ep^2 t D_0(\ep s,\ep t)\cdot\partial_{s}+\ep^3
t^2F_0(\ep s,\ep t)\cdot
\partial_{t}
\end{align*}
for which
\begin{align*}
D_0(\ep s ,\ep
t)&=\partial_{\bs{t}}\left(\frac{\partial_{\bs{s}}a}{a}\right)(\ep
s,0) +\ep \mathit{O}\left(t\partial_{\bs{t}\bs{t}}\left[\frac{\partial_t
a}{a}\right]\right)
+A_0(\ep s,\ep t)\frac{\partial_{\bs{s}}a}{a}(\ep s,\ep t)\\
F_0(\ep s ,\ep t)&=\dfrac{1}{2}\partial_{\bs{t}\bs{t}}\left[\frac{\partial_t
a}{a}(\ep s,0)\right]+\ep
\mathit{O}\left(t\partial_{\bs{t}\bs{t}\bs{t}}\left(\frac{\partial_t a}{a}\right)\right)
\end{align*}
\noindent and the function $A_0(\ep s,\ep t)$ given by
\eqref{ErrorA0LaplacianXeph}

Next, we observe that for $z=t-h(\ep s)$, this product in dilated
and translated Fermi coordinates amounts to

\begin{align*}
\ep\dfrac{\nabla_{X}a}{a}(\ep x,\ep y)\nabla_{X_{\ep,h}}&=
\ep\frac{\partial_{\bs{s}}a}{a}(\ep s,0)[\partial_{s}-\ep h'(\ep s) \partial_{t}]\\
&+\ep \left[\frac{\partial_{\bs{t}}a}{a}(\ep s,0)+\ep (t+h(\ep s))
\left(\frac{\partial_{\bs{tt}}a}{a}(\ep
s,0)-\biggl\lvert\frac{\partial_{\bs{t}}a}{a}(\ep
s,0)\biggr\rvert^2\right)\right]\partial_{t}\\
&+E_{\ep,h}(\ep s,t)
\end{align*}
where $E_{\ep,h}$ is a small operator for $\ep>0$ small enough, and
of the form
\begin{align*}
E_{\ep,h}(\ep s,t)&:=\ep^2 (t+h(\ep s)) D_0(\ep s,\ep(t+h(\ep
s)))[\partial_{s}v_u^{*}(s,t)-\ep h'(\ep s) \partial_{t}v_u^{*}(s,t)]\\[0.2cm]
&+\ep^3 (t+h(\ep s))^2F_0(\ep s,\ep (t+h(\ep
s)))\partial_{t}v^{*}_u(s,t)
\end{align*}
such that the following functions are smooth, and these relations
can be derived.
\begin{align*}
D_0(\ep s ,\ep
(t+h))&=\partial_{\bs{t}}\left[\frac{\partial_{\bs{s}}a}{a}\right](\ep
s,0)
+\ep \mathit{O}\left((t+h(\ep s))\partial_{\bs{t}\bs{t}}\left[\frac{\partial_t a}{a}\right]\right)\\
&+A_0(\ep s,\ep (t+h))\frac{\partial_{\bs{s}}a}{a}(\ep s,\ep (t+h))\\
F_0(\ep s ,\ep
(t+h))&=\dfrac{1}{2}\partial_{\bs{t}\bs{t}}\left[\frac{\partial_t
a}{a}\right](\ep s,0)+\ep \mathit{O}\left((t+h(\ep
s))\partial_{\bs{t}\bs{t}\bs{t}}\left[\frac{\partial_t
a}{a}\right]\right)
\end{align*}
this completes the proof of
Lemma~\ref{LemmaEuclideanGradients}.\hfill$\square$

\bigskip
{\textbf{Acknowledgements}: 
This work has been partially supported by the research grant Fondecyt 1110181.}

%%-----------------------  END SUBSUBSECTION: Gradient in dilated and translated Fermi coordinates ------------------------------------%%
%%-------------------------------------------------------------------------------------------------------------------------------------%%
%%-----------------------  END SUBSECTION: Calculation of gradients in dilated and translated Fermi coordinates  ----------------------%%

%%%%%%%%%%%%%%%%%%%%%%%%%%%%%%%%%%%%%%%%%%%%%%%%%%%%%%%%%%%%%%%
%%%%%%%%%%%%%%%%%%%%%%%%%%%%%%%%%%%%%%%%%%%%%%%%%%%%%%%%%%%%%%%
%%%%%%%%%%%%%%%%%%%%%%%%%%%%%%%%%%%%%%%%%%%%%%%%%%%%%%%%%%%%%%

\end{document}